\documentclass[11pt,reqno]{amsart}

\usepackage{etex}
\usepackage{amsmath,amssymb,amsthm,amsfonts,mathrsfs,dsfont,mathtools,mathrsfs}
\usepackage[frame,cmtip,arrow,matrix,line,graph,curve]{xy}
\usepackage{graphpap,color,paralist,pstricks}
\usepackage[mathscr]{eucal}
\usepackage{mathabx}
\usepackage[pdftex,colorlinks,backref=page,citecolor=blue]{hyperref}
\usepackage[latin1]{inputenc}
\usepackage{tikz}
\usepackage{tikz-cd}
\usetikzlibrary{fit,matrix}
\usepackage{epic,eepic}
\usepackage{yfonts}
\usepackage[shortlabels]{enumitem}
\usepackage{caption}
\usepackage{extpfeil}
\usepackage[normalem]{ulem}
\usepackage{tikz-cd}

\usepackage{enumitem}

\setlist[enumerate]{label=(\arabic*), itemsep=3pt, leftmargin=*}

\usepackage{todonotes}

\theoremstyle{definition}
\newtheorem*{theorem*}{Theorem}
\newtheorem{thmx}{Theorem}
\newtheorem{theorem}{Theorem}[section]
\newtheorem{definition}[theorem]{Definition}

\newtheorem{lemma}[theorem]{Lemma}
\newtheorem{proposition}[theorem]{Proposition}
\newtheorem{corollary}[theorem]{Corollary}
\theoremstyle{remark}
\newtheorem{remark}[theorem]{Remark}
\newtheorem{example}[theorem]{Example}

\newcommand*{\ub}{\underline{\smash{\beta}}}
\newcommand{\excise}[1]{}

\newcommand{\im}{\operatorname{im}}
\newcommand{\id}{\operatorname{id}}
\newcommand{\coker}{\operatorname{coker}}

\newcommand{\C}{{\mathbb{C}}}
\newcommand{\Z}{{\mathbb{Z}}}
\newcommand{\Q}{{\mathbb{Q}}}
\renewcommand{\H}{\operatorname{H}}
\newcommand{\IH}{\operatorname{IH}}

\renewcommand{\S}{\mathrm{S}}

\newcommand{\N}{\mathrm{N}}

\renewcommand{\P}{\mathbb{P}}
\newcommand{\F}{\mathbb{F}}

\renewcommand{\b}{\ub}

\newcommand{\crk}{\operatorname{crk}}

\newcommand{\subs}{\subseteq}
\newcommand{\sups}{\supseteq}

\newcommand{\mmod}{\text{--}\mathrm{mod}}

\newcommand{\D}{\operatorname{D}}

\newcommand{\M}{\mathrm{M}}

\newcommand{\rk}{\operatorname{rk}}
\newcommand{\IC}{\operatorname{IC}}
\newcommand{\gr}{\operatorname{gr}}

\newcommand{\cC}{\mathcal{C}}
\newcommand{\cV}{\mathcal{V}}
\newcommand{\cW}{\mathcal{W}}

\newcommand{\cL}{\mathscr{L}}
\newcommand{\cE}{\mathcal{E}}
\newcommand{\cM}{\mathcal{M}}

\newcommand{\cF}{\mathcal{F}}
\newcommand{\cG}{\mathcal{G}}
\newcommand{\cH}{\mathcal{H}}
\newcommand{\cS}{\mathcal{S}}

\renewcommand{\and}{\ \ \text{and}\ \ }

\newcommand{\eE}{\mathcal{Q}}

\newcommand{\End}{\operatorname{End}}

\newcommand{\Hom}{\operatorname{Hom}}

\newcommand{\beq}{\begin{eqnarray*}}
\newcommand{\eeq}{\end{eqnarray*}}

\newcommand{\Span}{\operatorname{span}}
\newcommand{\Ann}{\operatorname{ann}}

\newcommand{\mm}{\mathfrak{m}}

\newcommand{\uE}{\underline{\operatorname{E}}}

\newcommand{\CH}{\operatorname{CH}}
\newcommand{\uCH}{\underline{\CH}}

\newcommand{\oL}{\mathrm{L}}

\newcommand{\HH}{\mathbb{H}}
\newcommand{\PP}{\mathbb{P}}

\newcommand{\reg}{\mathrm{reg}}
\newcommand{\pure}{\mathrm{pure}}

\newcommand{\V}{\mathcal{V}}
\newcommand{\W}{\mathcal{W}}

\newcommand{\kk}{\Bbbk}
\newcommand{\olkk}{\overline{\kk}}
\newcommand{\ukk}{\underline{\kk}}

\newcommand{\E}{\mathrm{E}}
\newcommand{\Poin}{\operatorname{Poin}}

\renewcommand{\varnothing}{{\hat{0}}}

\newcommand{\hz}{{\hat{0}}}
\newcommand{\ho}{{\hat{1}}}
\newcommand{\nwa}{\nwarrow}
\newcommand{\swa}{\swarrow}

\begin{document}

\title{The parity module of a matroid}

\author{Tom Braden}
\address{Department of Mathematics and Statistics, University of Massachusetts, Amherst, MA.}
\email{braden@umass.edu}

\author{June Huh}
\address{Department of Mathematics, Princeton University, Princeton, NJ, and Korea Institute for Advanced Study, Seoul, Korea.}
\email{huh@princeton.edu}

\author{Jacob P. Matherne}
\address{Department of Mathematics, North Carolina State University, Raleigh, NC.}
\email{jpmather@ncsu.edu}

\author{Nicholas Proudfoot}
\address{Department of Mathematics, University of Oregon, Eugene, OR.}
\email{njp@uoregon.edu}

\author{Botong Wang}
\address{Department of Mathematics, University of Wisconsin-Madison, Madison, WI.}
\email{wang@math.wisc.edu}


\begin{abstract}
We introduce the parity module of a matroid with coefficients in a field $\kk$. When $\kk$ has characteristic zero, we show that the parity module of a matroid is isomorphic to its intersection cohomology module, yielding a new intrinsic characterization of the latter. The structure and dimension of the parity module can vary depending on the characteristic, 
giving rise to characteristic-dependent generalizations of the Kazhdan--Lusztig and Z-polynomials of matroids.  When the matroid is realizable over the complex numbers, the parity module is isomorphic to the hypercohomology of the parity sheaf supported on the arrangement Schubert variety of any such realization.
We also introduce the notion of pure modules over any field $\kk$ and show that  they satisfy a matroid analogue of the Beilinson--Bernstein--Deligne--Gabber decomposition theorem: Every pure module over the graded M\"obius algebra decomposes as a direct sum of shifts of the parity modules 
of matroids obtained by restricting to flats.
\end{abstract}

\maketitle

\section{Introduction}

Parity sheaves were introduced by Juteau, Mautner, and Williamson in 2014, and have since become fundamental objects in modular geometric representation theory \cite{JMW14}. 
They provide a well-behaved substitute for intersection cohomology sheaves with coefficients in a field of positive characteristic, where the usual formalism can fail in important ways. In particular, parity sheaves retain many of the structural features familiar from characteristic zero, including an analogue of the Beilinson--Bernstein--Deligne--Gabber Decomposition Theorem in \cite{BBD}.

The intersection cohomology module of a matroid was introduced in \cite{SHCG} and played a key role there in the proof of the Dowling--Wilson top-heavy conjecture \cite{DW1}, as well as in the proof of the nonnegativity of Kazhdan--Lusztig polynomials of matroids conjectured in \cite{EPW}. In this paper, we introduce the parity module of a matroid (Theorem~\ref{def of E}). This object is defined over any coefficient field and agrees with the intersection cohomology module when the coefficient field has characteristic zero (Theorem~\ref{E=IH}). Thus, in addition to giving a new characterization of the intersection cohomology module of a matroid, our construction initiates the study of positive characteristic analogues of Kazhdan--Lusztig and Z-polynomials of matroids.


\subsection{Pure modules}\label{sec:pure}
Let $\cL$ be a geometric lattice with minimal element $\varnothing$ and maximal element $\hat 1$.
This means that $\cL$ is a finite graded lattice that is {\bf atomic}, meaning that every element is a join of atoms, and {\bf semimodular}, meaning that
\[
\rk(F)+\rk(G) \ge \rk(F \vee G) +\rk(F \wedge G) \ \ \text{for all $F,G \in \cL$.}
\]
Equivalently, $\cL$ is the lattice of flats of a matroid \cite[Theorem 1.7.5]{Oxley}. 
For any field $\kk$, the {\bf graded M\"obius algebra} $\H(\cL;\kk)$
is a $\kk$-algebra with additive basis $\{y_F\mid F\in \cL\}$, where $y_F$ is homogeneous of degree $\rk F$, 
with multiplication given by the rule
\[
y_Fy_G \coloneq \begin{cases} y_{F\vee G} & \text{if $\rk F + \rk G = \rk (F\vee G)$,}\\
0 & \text{if $\rk F + \rk G > \rk (F\vee G)$.}
\end{cases}
\]
The element $y_{\varnothing}$ is the identity element of $\H(\cL;\kk)$.
Whenever we use the word ``module'', we will always mean a module that is finitely generated and $\Z$-graded.
Likewise, the word ``module homomorphism'' will always mean a degree-preserving homomorphism of graded modules.
We write $\rk\cL$ for the rank of the maximal element of $\cL$.

A subset $\Sigma$ of  $\cL$ is called an {\bf upper set} if it is upwardly closed, meaning that, whenever we have $F \le G$ and $F \in \Sigma$, 
we also have $G\in \Sigma$.
Given an upper set $\Sigma$ and an $\H(\cL;\kk)$-module $\M$, we define submodules
\[
\Ann_\Sigma(\M) \coloneq \{m\in \M\mid \text{$y_F\cdot m = 0$ for all $F\in \Sigma$}\}\and
\Span_\Sigma(\M) \coloneq \kk\{y_F \cdot m\mid \text{$F\in \Sigma$ and $m\in \M$}\}.
\]
These two constructions allow us to define a notion of purity for $\H(\cL;\kk)$-modules.

\begin{definition}\label{pure module def}
	Let $\M$ be an $\H(\cL;\kk)$-module.
	\begin{enumerate}[(a)]
		\item  We call $\M$ \textbf{flabby} if $\Ann_{\Sigma_1}(\M) + \Ann_{\Sigma_2}(\M) = \Ann_{\Sigma_1\cap\Sigma_2}(\M)$
		for any upper sets $\Sigma_1$ and $\Sigma_2$.
		\item  We call $\M$ \textbf{coflabby} if 
		$\Span_{\Sigma_1}(\M) \cap \Span_{\Sigma_2}(\M) = \Span_{\Sigma_1\cap\Sigma_2}(\M)$
		for any upper sets $\Sigma_1$ and $\Sigma_2$.
		\item If $\M$ is both flabby and coflabby, we say that $\M$ is \textbf{pure}.
	\end{enumerate}
\end{definition}

For an $\H(\cL;\kk)$-module $\M=\bigoplus_{i} \M^i$, we write
$\M^*$ for its  {\bf dual module}, whose degree \(i\) piece is
$\operatorname{Hom}_{\kk}(\M^{-i},\kk)$.
Note that the flabby and coflabby conditions are interchanged by duality: if $\M$ is flabby, then $\M^*$ is coflabby, and vice versa.  Note also that we always have
$$\Ann_{\Sigma_1}(\M) + \Ann_{\Sigma_2}(\M) \subs \Ann_{\Sigma_1\cap\Sigma_2}(\M)
\and \Span_{\Sigma_1}(\M) \cap \Span_{\Sigma_2}(\M) \sups \Span_{\Sigma_1\cap\Sigma_2}(\M);$$
the opposite inclusions characterize flabbiness and coflabbiness.

Our first main results are as follows (see the beginning of Section \ref{sec:structure} for a discussion of the distinction between lettered and numbered
theorems).

\begin{thmx}\label{def of E}
Up to isomorphism of $\H(\cL;\kk)$-modules, there is a unique $\H(\cL;\kk)$-module $\E(\cL;\kk)$ 
	satisfying the following three conditions:
	\begin{enumerate}[(a)]
		\item $\E(\cL;\kk)$ is pure, 
		\item $\E(\cL;\kk)$ is indecomposable, 
		\item there is an element $x\in \E^0(\cL;\kk)$ such that $y_{\hat 1}\cdot x$ is nonzero.
	\end{enumerate}
\end{thmx}

\begin{thmx}\label{props of E}
This module $\E(\cL;\kk)$ has the following additional properties:
\begin{enumerate}[(a)]
 \item $\E^i(\cL;\kk) = 0$ if $i<0$ or $i>\rk \cL$,
 \item $\E^0(\cL;\kk)$ and $\E^{\rk \cL}(\cL;\kk)$ are $1$-dimensional $\kk$-vector spaces.
  \item If $\cL = \cL_1\times\cL_2$, then $\E(\cL;\kk)\cong \E(\cL_1;\kk)\otimes\E(\cL_2;\kk)$
 as modules over the graded M\"obius algebra $\H(\cL;\kk)\cong \H(\cL_1;\kk)\otimes\H(\cL_2;\kk)$.
 \end{enumerate} 
\end{thmx}

The module $\E(\cL;\kk)$ is called the {\bf parity module} of $\cL$ with coefficients in $\kk$. See Remark~\ref{characterization} for a more explicit way of expressing the indecomposability condition (b) in Theorem~\ref{def of E}, in terms of costalk-to-stalk homomorphisms. In addition to this abstract characterization, we will give an inductive construction of $\E(\cL;\kk)$ using elementary linear algebra. 


\begin{remark}
The terms flabby and coflabby
refer to a sheaf and a cosheaf on $\cL$ which can be associated with an $\H(\cL;\kk)$-module $\M$, where $\cL$ is given the topology with upper sets as open sets.  This point of view is important for the proof of Theorem \ref{def of E} that begins in Section \ref{sec:sheafifying}, but it is not needed for the rest of the introduction or for the applications in the earlier sections. 
\end{remark}

\begin{remark}
In \cite{SHCG}, we gave a different definition of purity for modules over $\H(\cL;\Q)$.  See Proposition \ref{prop:pure equals pure} for a proof that the two definitions agree.
\end{remark}

Suppose that $\cL$ is the lattice of flats of a matroid,
and that this matroid is realizable by a linear subspace $V\subseteq\C^{\cL^1}$, where $\cL^1$ is the set of atoms of $\cL$.  
The {\bf arrangement Schubert variety} $Y_V$ of $V$
is by definition the closure of $V$ inside of $(\mathbb{CP}^1)^{\cL^1}$, and it has the property that there
exists a degree-doubling isomorphism $\H(\cL;\kk)\cong \H^\bullet(Y_V;\kk)$
\cite{HW}.  Let $\cE(Y_V;\kk)$ denote the parity sheaf of $Y_V$
with coefficients in $\kk$ \cite{JMW14}.  The following theorem justifies our use of the name ``parity module''.

\begin{thmx}\label{topological interpretation}
If $\cL$ is the lattice of flats of a matroid realized by a linear subspace $V\subseteq\C^{\cL^1}$,
then there exists a degree-doubling isomorphism $\E(\cL;\kk)\cong \HH^\bullet(\cE(Y_V;\kk))$ as modules over 
the ring $\H(\cL;\kk)\cong \H^\bullet(Y_V;\kk)$.
\end{thmx}

The {\bf intersection cohomology} $\H(\cL;\Q)$-module $\IH(\cL)$ was introduced in \cite{SHCG}, where it was the key character
in the proof of the Dowling--Wilson conjecture for $\cL$.  The definition of $\IH(\cL)$ given there is rather complicated \cite[Section 3]{SHCG}, and involves taking the orthogonal complement of a recursively defined submodule of the {\bf augmented Chow ring} $\CH(\cL;\Q)$ introduced in \cite{BHMPW}.  
(See Section~\ref{Section:AugmentedChow} for a review of the augmented Chow ring.) The following theorem
provides a direct and intrinsic characterization of the module $\IH(\cL)$.

\begin{theorem}\label{E=IH}
There is an isomorphism of $\H(\cL;\Q)$-modules $\E(\cL;\Q)\cong\IH(\cL)$, and this isomorphism is unique up to multiplication by a nonzero scalar in  $\Q$.
\end{theorem}

\begin{remark}
As in the setting of Theorem \ref{topological interpretation}, suppose that $\cL$ comes from a linear subspace $V\subseteq\C^E$.
If $\kk=\Q$, the parity sheaf of $Y_V$ with coefficients in $\kk$ coincides with the intersection cohomology sheaf.
Thus, combining Theorems \ref{topological interpretation} and \ref{E=IH},
we obtain the statement that $\IH(\cL)$ is isomorphic to the intersection cohomology module $\IH^\bullet(Y_V;\kk)$.
This is implicit in \cite[Section 1.3]{SHCG}, but it is not stated or proved explicitly.
\end{remark}

For any $F \in \cL$, we consider the {\bf lower interval} $\cL^F$ and the {\bf upper interval} $\cL_F$ defined by
\[
\cL^F \coloneq \{G\in\cL\mid G\leq F\}
\and \cL_F \coloneq \{G\in\cL\mid F\leq G\}.
\] 
The former corresponds to the restriction of the matroid of $\cL$ to the flat $F$, while the latter corresponds to the contraction of the matroid of $\cL$ by the flat $F$.
Note that the inclusion of $\cL^F$ into $\cL$ is rank-preserving, while the inclusion of $\cL_F$ into $\cL$ increases rank by $\rk F$.
There is a surjective homomorphism $\H(\cL;\kk) \to \H(\cL^F;\kk)$ taking $y_G$ to $y_G$ if $G\le F$ and to $0$ otherwise.  
Thus an $\H(\cL^F;\kk)$-module can be given an induced $\H(\cL;\kk)$-module structure.

\begin{thmx}\label{indecomposables}
Up to isomorphism and degree shift, the modules $\E(\cL^F;\kk)$ for  $F\in \cL$ give a complete set of indecomposable pure $\H(\cL;\kk)$-modules.
Consequently, every pure $\H(\cL;\kk)$-module decomposes as a direct sum of shifts of these modules.
\end{thmx}

\begin{example}
The augmented Chow ring $\CH(\cL;\kk)$ is a pure module over $\H(\cL;\kk)$ (Proposition \ref{true facts}(e)), and therefore decomposes as a direct sum of shifts of the parity modules $\E(\cL^F;\kk)$. The shape of this decomposition reflects an interesting interaction between the combinatorics of $\cL$ and the characteristic of $\kk$. For example, when $\cL$ is the lattice of flats of the uniform matroid $U_{2,3}$, we have 
\[
\CH(\cL;\kk)\cong \begin{cases} \E(\cL;\kk) & \text{if $\operatorname{char}(\kk)=2$,}\\  \E(\cL;\kk)\oplus \E(\cL^{\hat 0};\kk)[-1] &  \text{if $\operatorname{char}(\kk)\neq 2$.} \end{cases}
\]
In Theorem~\ref{general decomposition}, we describe how an arbitrary pure module decomposes into shifts of parity modules using  costalk-stalk homomorphisms.
\end{example}

\subsection{Polynomials}
Consider the {\bf Z-polynomial} $Z_\cL(t)$, introduced in \cite{PXY}, and the {\bf KL-polynomial} $P_\cL(t)$, introduced in \cite{EPW}.
The {\bf stalk} of the parity module of $\cL$ at $\hat 0$ is the graded vector space 
\[
\uE(\cL;\kk) \coloneq \E(\cL;\kk) \big{/} \Span_{\cL\setminus\varnothing} \E(\cL;\kk).
\]
By Theorem \ref{E=IH} along with \cite[Theorem 1.19]{SHCG}, we have
\begin{equation}\label{PZ0}
Z_\cL(t) = \sum_{i\geq 0}t^i \dim \E^i(\cL;\Q) \and P_\cL(t) = \sum_{i\geq 0}t^i \dim \uE^i(\cL;\Q).
\end{equation}
By analogy with classical Kazhdan--Lusztig theory, for any prime $p$, we define the {\bf \boldmath{$p$}-Z-polynomial} and {\bf \boldmath{$p$}-KL-polynomial}
\begin{equation}\label{PZp}
{}_pZ_\cL(t) \coloneq \sum_{i\geq 0}t^i \dim \E^i(\cL;\F_p) \and _pP_\cL(t) \coloneq \sum_{i\geq 0}t^i \dim \uE^i(\cL;\F_p). 
\end{equation}
We will show in Corollary~\ref{cor:field extensions} 
that field extensions do not affect the structure of parity modules. 
Thus, in fact, Equation \eqref{PZ0} holds with $\Q$ replaced with any field of characteristic zero,
and Equation \eqref{PZp} holds with $\F_p$ replaced with any field of characteristic $p$.
By convention, we will write $$_0Z_\cL(t) = Z_\cL(t) \and {}_0P_\cL(t) = P_\cL(t),$$ and we will sometimes write $\F_0=\Q$
so that Equation \eqref{PZp} holds when $p=0$.

\begin{theorem}\label{combinatorial recursion}
Let $p$ be a prime number or zero.
The $p$-Z-polynomial and $p$-KL-polynomial satisfy the following properties:
\begin{enumerate}[(a)]
	\item If $\rk \cL = 0$, then ${}_pP_\cL(t) = 1$.
	\item For all $\cL$, the $p$-Z-polynomial is palindromic:  ${}_pZ_\cL(t) = t^{\rk \cL}{}_pZ_\cL(t^{-1})$. 
	\item For all $\cL$, we have $\displaystyle {}_pZ_\cL(t) = \sum_{F\in \cL} t^{\rk F}{}_pP_{\cL_F}(t)$.
\end{enumerate} 
\end{theorem}

In \cite{SHCG}, it was shown that $\deg {}_0P_\cL(t) < (\rk \cL)/2$ for all $\cL$ with $\rk \cL > 0$.
This inequality, along with the three properties of Theorem \ref{combinatorial recursion},
completely characterizes the polynomials $_0Z_\cL(t)$ and $_0P_\cL(t)$ \cite[Theorem 2.2]{BV}.
For positive $p$, the degree of $_pP_\cL(t)$ can be larger than $(\rk \cL)/2$, and Theorem \ref{combinatorial recursion}
is not sufficient to compute the polynomials $_pZ_\cL(t)$ and $_pP_\cL(t)$.  We do not know of a purely combinatorial
algorithm to compute these polynomials, except for the linear term --- see Proposition \ref{resonance linear} below.

\begin{remark}
It follows from Theorem~\ref{E=IH} and the hard Lefschetz property for $\IH(\cL)$ of \cite[Theorem 1.6]{SHCG} that the coefficients of  $Z_\cL(t)$ form a unimodal sequence.
Do the coefficients of  ${}_pZ_\cL(t)$ form a unimodal sequence for $p$ positive?
\end{remark}

If $f(t), g(t)\in\mathbb{N}[t]$, we write $f(t) \leq g(t)$ to mean that $g(t)-f(t)\in \mathbb{N}[t]$.
Theorem \ref{it goes up} below asserts that working in positive characteristic can only result in larger Z- and KL-polynomials.
The analogous result for classical Kazhdan--Lusztig polynomials can be found in \cite[Proposition 4.2(3)]{JensenWilliamson}.

\begin{thmx}\label{it goes up}
If $p$ is prime, we have
$$_0Z_\cL(t) \leq {}_pZ_\cL(t)\and _0P_\cL(t) \leq {}_pP_\cL(t),$$
with equality for all but finitely many primes.
\end{thmx}

We define an element $F$ in $\cL$ to be {\bf \boldmath{$p$}-resonant} if $F$ is not $\hat 1$ and,
for all $G>F$ with $\rk(G) = \rk(F) + 2$, the cardinality of the open interval $(F,G)$
is congruent to $1$ modulo $p$.  The only $0$-resonant elements are those of corank $1$,
for which there are no such $G$.  However, there may be $p$-resonant elements
of  corank larger than $1$.  For example,  every proper flat of the Fano matroid is $2$-resonant, and every nonempty proper flat of the ternary affine plane is $3$-resonant.

\begin{proposition}\label{resonance linear}
The coefficient of $t$ in $_pZ_\cL(t)$ is equal to the number of $p$-resonant elements of $\cL$.
The coefficient of $t$ in $_pP_\cL(t)$ is equal to the number of $p$-resonant elements of $\cL$
minus the number of rank $1$ elements of $\cL$.
\end{proposition}

\subsection{Relation with the augmented Chow ring}\label{sec:Chow}

Recall from Theorem \ref{E=IH} that $\E(\cL;\Q)$ is isomorphic to $\IH(\cL)$, which is a submodule of the
augmented Chow ring $\CH(\cL;\Q)$ containing the graded M\"obius algebra $\H(\cL;\Q)$. 
The following proposition gives an analogous pair of inclusions over an arbitrary coefficient field.

\begin{proposition}\label{H-E-CH}
For any geometric lattice $\cL$ and field $\kk$, there are inclusions of $\H(\cL;\kk)$-modules $$\H(\cL;\kk)\subs\E(\cL;\kk)\subs\CH(\cL;\kk).$$
\end{proposition}

The graded dimensions of $\H(\cL;\kk)$ and $\CH(\cL;\kk)$ depend only on $\cL$, not on $\kk$: The dimension of $\H^k(\cL;\kk)$ is the number of rank $k$ elements of $\cL$, while the Hilbert series of $\CH(\cL;\kk)$ is the {\bf augmented Chow polynomial} $C_\cL(t)$ studied in \cite{FMSV}. 
When $\kk=\Q$, the first inclusion, together with the hard Lefschetz property of $\E(\cL;\kk)$,  was used in \cite{SHCG} to prove the Dowling--Wilson conjecture:
\[
|\cL^k|\leq |\cL^{d-k}|
\ \ \text{for all $k\leq d/2=\rk\cL/2$}.
\]
The second inclusion is built into the construction of $\IH(\cL)$ in \cite{SHCG}.

In the next two theorems, we specify exactly when each of these inclusions is an equality.
A simple, connected, modular matroid of rank greater than $1$ is called a {\bf projective geometry} \cite[Chapter 12]{Welsh}.
Given a projective geometry, there exists an integer $q\geq 2$ such that every rank $2$ flat has cardinality $q+1$; this number
is called the {\bf order}.  Examples of projective geometries include the uniform matroid $U_{2,n}$ for $n \ge 3$, which has order $n-1$,
the projective geometry $\P^{d-1}(\F_q)$ for a prime power $q$, which has order $q$, and all non-Desarguesian projective planes.
We write $\P^{d-1}(\F_q)$ to denote both the matroid and its lattice of flats.

\begin{theorem}\label{E=H}
The following conditions are equivalent for $\cL$ and $\kk$, with $p$ the characteristic of $\kk$:
\begin{enumerate}[(a)]
\item The $\H(\cL;\kk)$-module $\H(\cL;\kk)$ is pure;
\item $\E(\cL;\kk) \cong \H(\cL;\kk)$ as $\H(\cL;\kk)$-modules;
\item ${}_pP_\cL(t) = 1$;
\item $\cL$ is modular and every $p$-resonant element of $\cL$ has corank $1$;
\item $\cL$ is isomorphic to a product of a Boolean lattice and a collection of finite projective spaces, each with an order that is not
a multiple of $p$.
\end{enumerate}
\end{theorem}

The augmented Chow ring $\CH(\cL;\kk)$ has as a quotient the {\bf Chow ring}
$\uCH(\cL;\kk)$; see \cite{AHK} and \cite{BHMPW}.  The {\bf Chow polynomial} $\underline{C}_\cL(t)$, which depends only on $\cL$ and not on $\kk$, is defined to be the Poincar\'e polynomial
of  $\uCH(\cL;\kk)$.  For combinatorial descriptions of Chow and augmented Chow polynomials, see
\cite{FMSV}.

\begin{theorem}\label{E=CH}
The following conditions are equivalent for $\cL$ and $\kk$, with $p$ the characteristic of $\kk$:
\begin{enumerate}[(a)]
\item The $\H(\cL;\kk)$-module $\CH(\cL;\kk)$ is indecomposable;
\item $\E(\cL;\kk) \cong \CH(\cL;\kk)$ as $\H(\cL;\kk)$-modules;
\item ${}_pZ_\cL(t) = C_\cL(t)$;
\item ${}_pP_\cL(t) = \underline{C}_\cL(t)$;
\item Every non-maximal element of $\cL$ is $p$-resonant.
\end{enumerate}
\end{theorem}

We say that $\cL$ is {\bf \boldmath{$p$}-resonant} if $\cL$ and $p$ satisfy the equivalent conditions of Theorem~\ref{E=CH}.
A matroid is $p$-resonant if its lattice of flats is $p$-resonant.

\begin{example}
We collect here some infinite families of $p$-resonant matroids.
\begin{enumerate}[(1)]
\item The uniform matroid $U_{2,n}$ is $p$-resonant if and only if $n$ is congruent to $1$ modulo $p$.
\item A projective plane of order $q$ is $p$-resonant if and only if $p$ divides $q$. The {\bf prime power conjecture} predicts that every projective plane is $p$-resonant for a unique $p$ \cite[Section 6.1]{Oxley}.  
\item The projective geometry $\mathbb{P}^{d-1}(\mathbb{F}_q)$ is $p$-resonant if and only if $q$ is a power of $p$. More generally,  a {$q$-niform matroid} \cite{QNiform} of rank at least $3$ obtained by truncating a projective geometry over $\mathbb{F}_q$  is $p$-resonant if and only if $q$ is a power of $p$. 
\item For $1<t<k<n$, the blocks of a Steiner system $S(t,k,n)$ form the set of hyperplanes of a rank $t+1$ matroid on $n$ elements \cite[Chapter 12]{Welsh}. A Steiner system $S(2,3,n)$ exists if and only if $n$ is congruent to $1$ or $3$ modulo $6$, and the corresponding matroid is $2$-resonant if and only if $n$ is congruent to $3$ or $7$ modulo $12$.
A Steiner system $S(2,4,n)$ exists if and only if $n$ is congruent to $1$ or $4$ modulo $12$, and the corresponding matroid is $3$-resonant if and only if $n$ is congruent to $4$ or $13$ modulo $36$.
\item A hyperoval is a set of $q+2$ points in $\mathbb{P}^2(\mathbb{F}_q)$ that intersects every line in either $0$ or $2$ points; such a set exists if and only if $q$ is even. The restriction of $\mathbb{P}^2(\mathbb{F}_q)$ to the $q^2-1$ points in the complement of a hyperoval yields a $2$-resonant matroid.
\item 
Start with a projective space of order $q$, and choose a codimension $2$ subspace $S$. 
Consider the pencil of hyperplanes containing $S$, choose $k$ hyperplanes from the pencil, and restrict to the union of those hyperplanes. If $k$ is congruent to $1$ modulo $p$, then the resulting matroid is $p$-resonant, and the same holds for its truncations.
\end{enumerate}
\end{example}

\subsection{Inclusions of geometric lattices}\label{inclusions of lattices}
Consider a rank-preserving injective map of posets $\cL\to\cM$, where $\cL$ and $\cM$ are geometric lattices.
In Proposition~\ref{prop:rank-preserving}, we show that any such inclusion comes from the operation of deleting elements
from the ground set of the corresponding simple matroid.

\begin{example}
If $\cL$ is the lattice of flats of a simple matroid of rank $n$ that is realizable over $\F_q$, then $\cL$ includes into the lattice $\mathbb{P}^{n}(\F_q)$
of linear subspaces of $\F_q^{n+1}$, with different realizations inducing different inclusions.  If $\cL$ is the lattice of flats of a matroid associated with a simple
graph with $n$ vertices, then $\cL$ includes into the lattice $\Pi_{n}$
of partitions of the set $[n]$, which is the lattice of flats of the matroid of the complete graph on $n$ vertices.
\end{example}

The inclusion $\varphi\colon\cL\to\cM$ induces an inclusion of rings $\H(\cL;\kk)\to\H(\cM;\kk)$, taking $y_F$ to $y_{\varphi(F)}$.

\begin{theorem}\label{pullback of pure}
If $\M$ is a pure $\H(\cM;\kk)$-module, then it is also a pure $\H(\cL;\kk)$-module.
\end{theorem}


Recall that the category of finite-dimensional graded modules over a finite-dimensional graded algebra satisfies the Krull--Schmidt theorem \cite{Atiyah}. Therefore, when the graded  
$\H(\cL;\kk)$-module $\E(\cM;\kk)$ is decomposed into a direct sum of indecomposable graded submodules, the graded isomorphism classes of the indecomposable summands and their graded multiplicities are uniquely determined. Consequently, the Poincar\'e polynomials of the corresponding multiplicity spaces are well-defined invariants of the $\H(\cL;\kk)$-module $\E(\cM;\kk)$.

Let $\S(\varphi;\kk) \coloneq \Ann_{\cL\setminus\varnothing}(\E(\cM;\kk))$ be the socle of $\E(\cM;\kk)$ as a module over $\H(\cL;\kk)$,
and let $$\oL(\varphi;\kk) \coloneq \S(\varphi;\kk) \big{/} \big( \Span_{\cL\setminus\varnothing}(\E(\cM;\kk))\cap \S(\varphi;\kk) \big).$$ 
For each $F\in\cL$, the injection $\varphi$ induces an injection between the upper intervals $\cL_F\to\cM_{\varphi(F)}$, denoted $\varphi_F$. 
The following theorem describes a decomposition of $\E(\cM;\kk)$ into indecomposable $\H(\cL;\kk)$-modules in terms of the graded vector spaces $\oL(\varphi_F;\kk)$.

\begin{theorem}\label{decomposition}
There is a (noncanonical) isomorphism of $\H(\cL;\kk)$-modules
$$\E(\cM;\kk)\cong \bigoplus_{F\in\cL} \oL(\varphi_F;\kk)\otimes\E(\cL^F;\kk).$$
\end{theorem}

In analogy with \cite{subdivisions}, we define
the {\bf local \boldmath{$h$}-polynomials} of $\varphi$ by
$$\ell_\varphi(t) \coloneq  {}_0\ell_\varphi(t) \coloneq \sum_{i\geq 0} t^i \dim \oL^i(\varphi;\Q)
\and _p\ell_\varphi(t) \coloneq \sum_{i\geq 0} t^i \dim \oL^i(\varphi;\F_p).$$

The module decomposition in Theorem \ref{decomposition} has the following polynomial interpretation. It gives a recursive method for computing ${}_p\ell_{\varphi}(t)$, assuming that the $p$-Z-polynomials of all lower intervals in $\cL$ and all upper intervals in $\cM$ are known.

\begin{corollary}\label{Z and l}
Let $p$ be prime or zero.  We have 
\[
 _pZ_\cM(t) = \sum_{F\in\cL} {}_p\ell_{\varphi_F}(t)\ {}_pZ_{\cL^F}(t).
 \]
\end{corollary}

\begin{remark}
Suppose that $\cL$ and $\cM$ have the same rank and that the number of atoms of $\cM$ is one more than that of $\cL$.
In matroid terminology, this means that $\cL$ is obtained from $\cM$ by deleting a single element that is not a coloop.
In this case, with $p=0$, Corollary \ref{Z and l} specializes to the second part of \cite[Theorem 2.8]{BV}.
\end{remark}

\begin{remark}
Let $\cL^1$ and $\cM^1$ denote the sets of atoms of $\cL$ and $\cM$, respectively; we regard $\cL^1$ as a subset of $\cM^1$
via $\varphi$.  Suppose that we have a commutative diagram of complex vector spaces
\[
\begin{tikzcd}
V_{\cM} \ar[r,hookrightarrow] \ar[d,twoheadrightarrow]&\mathbb{C}^{\cM^1} \ar[d,twoheadrightarrow] \\
V_{\cL} \ar[r,hookrightarrow]&\mathbb{C}^{\cL^1}
\end{tikzcd}
\]
such that $\cL$ is the lattice of flats of the matroid realized by the linear subspace $V_\cL\subseteq\C^{\cL^1}$, and similarly for $\cM$.
This induces a commutative diagram of arrangement Schubert varieties
\[
\begin{tikzcd}
Y_{V_\cM} \ar[r,hookrightarrow] \ar[d,twoheadrightarrow]&(\mathbb{CP}^1)^{\cM^1} \ar[d,twoheadrightarrow] \\
Y_{V_\cL} \ar[r,hookrightarrow]&(\mathbb{CP}^1)^{\cL^1}.
\end{tikzcd}
\]
The case $p=0$ of Theorem~\ref{decomposition} gives a combinatorial description of the direct image of the intersection complex $\IC(Y_{V_\cM};\mathbb{Q})$ under this projection as a direct sum of  intersection complexes on $Y_{V_\cL}$.  When $p>0$ it describes the decomposition of the direct image of the parity sheaf $\mathcal{E}(Y_{V_\cM};\F_p)$ as a direct sum of parity sheaves \cite[Section 3.3]{JMW14}.
\end{remark}

The second part of the following theorem is proved by showing that the graded vector space
$\oL(\varphi;\kk)$ satisfies a version of 
the relative hard Lefschetz theorem when $p=0$. 
It follows that the polynomial  $_p\ell_\varphi(t)$ is unimodal when $p=0$.
Do the coefficients of  $_p\ell_\varphi(t)$  form a unimodal sequence for $p$ positive?

\begin{theorem}\label{Kahler}
Let $p$ be a prime or zero.
\begin{enumerate}[(a)]
\item The polynomial $_p\ell_\varphi(t)$ is palindromic:
$t^{\rk \cM}{}_p\ell_\varphi(t^{-1}) = {}_p\ell_\varphi(t)$.
\item When $p=0$, the polynomial ${}_p\ell_\varphi(t)$ is unimodal.
\end{enumerate}
\end{theorem}

Let $\cM/\cL$ be the geometric lattice $\cM_{\!\varphi(\hat 1)}$.  
Note that, by definition, $\oL(\varphi_{\hat 1}; \kk)$ is isomorphic to the parity module of  $\cM/\cL$ with coefficients in $\kk$ as a graded vector space over $\kk$.
Thus, Proposition~\ref{resonance linear} and Corollary~\ref{Z and l}  imply the following inequalities.


\begin{corollary}\label{Z inequality}
Let $p$ be a prime or zero.  If there is a rank-preserving inclusion $\varphi\colon\cL\to\cM$, then 
\[
_pZ_{\cM /\cL}(t) \cdot \,  _pZ_\cL(t) \leq {}_pZ_\cM(t).
\]
In particular, the number of $p$-resonant elements of $\cM$ is at least as large as the sum of the corresponding
numbers for $\cL$ and $\cM/\cL$.
\end{corollary}

\begin{corollary}
For any $F \in \cM$ and any $p$ which is prime or zero, we have 
\[
_pZ_{\cM_F}(t) \cdot \, _pZ_{\cM^F}(t) \leq {}_pZ_\cM(t). 
\]
The equality holds if $\cM$ is the product of $\cM_F$ and $\cM^F$.
\end{corollary}



We illustrate Corollary \ref{Z inequality}  with a few examples.

\begin{example}
Corollary \ref{Z inequality} implies that the Fano plane $\P^2(\F_2)$ is not realizable over $\F_3$: An $\F_3$-realization of $\P^2(\F_2)$ would induce a rank-preserving inclusion $\P^2(\F_2)\to \P^2(\F_3)$, but the number of $2$-resonant flats of $\P^2(\F_2)$ is $15$, whereas the number of $2$-resonant flats of $\P^2(\F_3)$ is $13$. 
Similarly, $\P^2(\F_5)$ is not realizable over $\F_7$, because  the number of $5$-resonant flats of $\P^2(\F_5)$ is $63$ and  the number of $5$-resonant flats  of $\P^2(\F_7)$ is $57$. 
\end{example}

\begin{example}\label{An Pn}
Let $\mathbb{A}^n(\F_q)$ be the lattice of affine subspaces of $\F_q^n$.
We may regard these as linear subspaces of $\F_q^{n+1}$ that are not contained in the hyperplane
at infinity, so we have an embedding $\mathbb{A}^n(\F_q)\subs \P^n(\F_q)$.
The lattice $\P^n(\F_q)$ has $q^n+\cdots+q+1$ elements of rank $n$,
whereas $\mathbb{A}^n(\F_q)$ has one fewer such element.  Thus the second inequality of
Corollary \ref{Z inequality} is strict when $p=0$.  If $q$ is odd, then $\P^n(\F_q)$
has no additional $2$-resonant elements, but $\mathbb{A}^n(\F_q)$ has one additional
$2$-resonant element, namely the minimal element.  Thus the second inequality of
Corollary \ref{Z inequality} is an equality when $p=2$.  Note, however, that the inclusion of $\mathbb{A}^n(\F_q)$ into $\P^n(\F_q)$
does {\em not} take $2$-resonant elements to $2$-resonant elements.
\end{example}

\begin{example}\label{exotic}
There exists an exotic inclusion of $\mathbb{A}^2(\mathbb{F}_3)$ into 
$\mathbb{P}^2(\mathbb{F}_4)$. 
The lattice $\mathbb{P}^2(\mathbb{F}_4)$ has $21$ $0$-resonant flats, and the same number
of $3$-resonant flats, namely all flats of rank $2$.
The sublattice $\mathbb{A}^2(\mathbb{F}_3)$ has only $12$ $0$-resonant elements, but it has
$21$ $3$-resonant elements.  Thus the inequalities of 
Corollary \ref{Z inequality} are equalities when $p=3$.
\end{example}

\begin{example}\label{Steiner}
Let $\cL$ be the geometric 
lattice associated with a Steiner system $S(2,5,45)$, viewed as a rank $3$ paving matroid on $45$ elements.\footnote{For constructions, see for example \cite{CZ-Designs}. There are at least $1072$ non-isomorphic $S(2,5,45)$ designs.}  Let $\cM = \P^2(\mathbb{F}_{11})$.  
Then $\cL$ has $99$ $0$-resonant elements and $\cM$ has $133$ $0$-resonant elements.
However, $\cL$ has $145$ $2$-resonant elements, and $\cM$ has only $133$ $2$-resonant elements.
Thus the $p=0$ case of the second inequality of Corollary \ref{Z inequality} does not provide an obstruction to
embedding $\cL$ into $\cM$, but the $p=2$ case does.
\end{example}

\subsection{Upper homogeneous geometric lattices}

We call  $\cM$ {\bf upper homogeneous} if, for all $F$ and  $G$ in $\cM$ of the same rank,
there is an isomorphism of posets $\cM_F\cong\cM_G$.  Examples of upper homogeneous
geometric lattices include the projective geometry $\P^n(\F_q)$, the affine geometry
$\mathbb{A}^n(\F_q)$, and the partition lattice $\Pi_n$. 

Recall from Theorem~\ref{pullback of pure} that a rank-preserving inclusion $\varphi\colon \cL\to\cM$ gives the parity module  $\E(\cM;\kk)$ the structure of a pure module over $\H(\cL;\kk)$.
 The following result follows from
Corollary \ref{Z and l} via induction.

\begin{theorem}\label{homogeneous}
If $\cM$ is upper homogeneous, then the isomorphism type of   $\E(\cM;\kk)$ as an $\H(\cL;\kk)$-module is independent of the choice of $\varphi$.
In particular, the local $h$-polynomials for inclusions  of $\cL$ into $\cM$
do not depend on the choice of inclusion. 
\end{theorem}


The following  example shows that the conclusion of Theorem~\ref{homogeneous} can fail if $\cM$ is not upper homogeneous.

\begin{example}
Let $\Gamma$ be the graph obtained from the complete graph $K_4$ by deleting
a single edge, and let $e$ be the unique edge of $\Gamma$ that is not adjacent
to the deleted edge.  Let $\cM$ be the lattice of flats of the corresponding graphic matroid.
Then $\cM$ is {\em not} upper homogeneous, because the matroid of the graph obtained by contracting
$e$ and simplifying is not isomorphic to the matroid of the graph obtained by contracting a different edge
and simplifying.

Let $\cL$ be the Boolean lattice of rank $3$.  Inclusions of $\cL$ into $\cM$ correspond
to choices of a spanning tree for $\Gamma$, and there are two kinds of spanning trees:
those that contain $e$ and those that do not.  If $\varphi\colon\cL\to\cM$ is an inclusion corresponding
to a spanning tree that contains $e$, then we may use Corollary \ref{Z and l} to
compute that ${}_0\ell_\varphi(t) = t+t^2$.
On the other hand, if $\psi\colon\cL\to\cM$ is an inclusion corresponding
to a spanning tree that does not contain $e$, then ${}_0\ell_\psi(t) = 0$.
\end{example}


For each $\cL$ and $p$, we recursively define
a polynomial $_p \ell_\cL(q,t)\in\Z[q,t]$ as follows:
\begin{enumerate}[(1)]
\item If $\rk\cL = 0$, then $_p \ell_\cL(q,t) = 1$.
\item If $\rk\cL = d>0$, then
$_p \ell_\cL(q,t) = \sum_{k=0}^d \genfrac{[}{]}{0pt}{}{d}{k}_q t^k- \sum_{F\neq\varnothing} \, {}_p\ell_{\cL_F}(q,t) \, {}_pZ_{\cL^F}(t).$
\end{enumerate}
This recursion indeed defines an element of $\mathbb{Z}[q,t]$, since the $q$-binomial coefficients $\genfrac{[}{]}{0pt}{}{d}{k}_q$  are polynomials in $q$ with integer coefficients.

The following proposition is a consequence of Corollary \ref{Z and l}, as well as an illustration
of Theorem~\ref{homogeneous} in the case where $\rk \cL = d$ and $\cM = \P^{d-1}(\F_q)$.

\begin{proposition}\label{qt}
If $\rk\cL = d$ and there is an inclusion $\varphi\colon\cL\to\P^{d-1}(\F_q)$, then
$ {}_p\ell_\cL(q,t)={}_p\ell_\varphi(t)$ for  $p=0$ and for every prime $p$ not dividing $q$.
\end{proposition}

It follows that, if $\cL$ is realizable over $\F_q$, then $ {}_p\ell_\cL(q,t) \in \mathbb{N}[t]$ for $p=0$ and for every prime $p$ not dividing $q$.

\begin{example}\label{u24}
Let $U_{2,4}$ be the lattice of flats for the uniform matroid of rank $2$ on $4$ elements.
We may compute ${}_0 \ell_{U_{2,4}}(q,t) = (q-3)t$.  In particular, ${}_0\ell_{U_{2,4}}(2,t) = -t$.
This reflects the fact that $U_{2,4}$ is not realizable over the field $\F_2$.
\end{example}

\subsection{Stalks, costalks, and decompositions}
For any $F\in\cL$, consider the upper sets 
\[
\Sigma_{\geq F} \coloneq \cL_F = \{G\in\cL\mid G\geq F\}\and
\Sigma_{>F} \coloneq \Sigma_{\geq F}\setminus\{F\} = \{G\in\cL\mid G> F\}.
\]
For any pure $\H(\cL;\kk)$-module $\M$ and $F \in \cL$, we define  the {\bf costalk} $\M_{[F]}$ and the {\bf stalk} $\M_F$  by
\[
 \M_{[F]} \coloneq \Ann_{\Sigma_{>\varnothing}}\big(\Span_{\Sigma_{\geq F}}(\M)\big)[\rk F]
 \ \ \text{and} \ \ 
\M_F \coloneq \big(\Span_{\Sigma_{\geq F}}(\M) / \Span_{\Sigma_{>F}}(\M) \big)[\rk F] .
\]
Observe that there is a natural {\bf costalk-stalk homomorphism}
\[
\beta_F(\M)\colon\M_{[F]}\longrightarrow \M_F,
\]
 given by including $\M_{[F]}$ into $y_F\M[\rk F]$ and then projecting to $\M_F$.

\begin{example}\label{beta0}
By Theorem~\ref{pullback of pure}, a rank-preserving inclusion $\varphi\colon\cL\to\cM$ gives  $\E(\cM;\kk)$ the structure of a pure module over $\H(\cL;\kk)$.
Its costalk at $\hat 0$ in $\cL$ is the socle $\S(\varphi;\kk)$, and the image of the costalk-stalk homomorphism at $\hat 0$ is the quotient 
$\oL(\varphi;\kk)$.
\end{example}

Theorem \ref{indecomposables} implies that the $\H(\cL;\kk)$-module $\M$ decomposes as a direct sum of shifts
of the parity modules $\E(\cL^F;\kk)$. Theorem~\ref{general decomposition} below characterizes the multiplicity
spaces in terms of the costalk-stalk homomorphisms, generalizing Theorem~\ref{decomposition}.

\begin{thmx}\label{general decomposition}
For any pure $\H(\cL;\kk)$-module $\M$, there is an isomorphism of $\H(\cL;\kk)$-modules
\[
\M \cong \bigoplus_{F\in\cL} \im(\beta_F(\M)) \otimes\E(\cL^F;\kk).
\]
\end{thmx}


\begin{remark}\label{characterization}
It is typically difficult to check
whether a given module over a ring is indecomposable.  However, Theorem \ref{general decomposition}
provides us with such a procedure for pure modules over the ring $\H(\cL;\kk)$:  a pure module $\M$ is indecomposable
if and only if the costalk-stalk homomorphism $\beta_F(\M)$ is equal to zero for all but one choice of $F\in\cL$, and the one nonvanishing
homomorphism has rank $1$.  In particular, $\E(\cL;\kk)$ is characterized, up to shift, by purity and the fact that
$$\rk \beta_{F}(\E(\cL;\kk)) = \begin{cases} 1 & \text{if $F=\hat 1$,}\\
0 & \text{otherwise.}\end{cases}$$
\end{remark}

\subsection{Structure of the paper}\label{sec:structure}
The technical core of this paper consists of Theorems \ref{def of E}, \ref{props of E}, \ref{topological interpretation}, \ref{indecomposables}, \ref{it goes up},
and \ref{general decomposition}, all of which will be proved by introducing auxiliary categories of sheaves on posets.
Rather than beginning with this somewhat dense material, we will instead begin by assuming these six results,
and using them\footnote{Theorem \ref{topological interpretation} will not be needed
to prove any of the numbered results.} to prove the numbered theorems, propositions, and corollaries from the introduction.

In Section~\ref{sec:filtrations}, we define a filtration of an arbitrary $\H(\cL;\kk)$-module and use it to prove
Theorem \ref{combinatorial recursion}. 
In Section~\ref{H}, we study the Poincar\'e pairing on $\H(\cL;\kk)$ and use it to prove
Theorem~\ref{E=H}. 
In Section~\ref{Section:AugmentedChow},
we study the augmented Chow ring $\CH(\cL;\kk)$ as a module over $\H(\cL;\kk)$, and we use
it to prove Theorems \ref{E=IH} and \ref{E=CH}, as well as Propositions~\ref{resonance linear} and ~\ref{H-E-CH}.
In Section~\ref{sec:inclusion}, we consider inclusions of geometric lattices $\cL \to \cM$ and prove Theorems~ \ref{pullback of pure} , ~\ref{decomposition}, 
and ~\ref{Kahler}.

The remaining sections are devoted to the proofs of Theorems \ref{def of E}, \ref{props of E}, \ref{topological interpretation}, \ref{indecomposables}, \ref{it goes up},
and \ref{general decomposition}.  Section~\ref{sec:sheafifying} gives a way of replacing an $\H(\cL;\kk)$-module $\M$ with an ``exploded" version.  Roughly, this means replacing $\M$ by the graded vector spaces $y_F\M$ for all $F\in \cL$, together with the maps between them induced by factorizations $y_F=y_Gy_H$ in $\H(\cL;\kk)$.  This generalizes to give a category $\cC(\Sigma;\kk)$ of ``exploded modules" for any upwardly closed subset $\Sigma\subseteq\cL$, 
in which only the spaces corresponding to elements $F\in\Sigma$ are retained.
This makes it possible to construct $\H(\cL;\kk)$-modules inductively, by adding elements to $\Sigma$ one at a time.

In Section~\ref{sec:pure modules}, we define the notion of flabbiness, coflabbiness, and purity for objects in $\cC(\Sigma;\kk)$, agreeing with the notions for $\H(\cL;\kk)$-modules when $\Sigma = \cL$.  We also generalize the stalk, costalk, and duality functors to this setting.  
In order to prove Theorem \ref{it goes up}, we need to  allow $\kk$ to be a DVR and define the notion of purity in this setting.

In Section~\ref{sec:existence}, we apply this to give an inductive construction and classification of indecomposable objects in $\cC(\Sigma;\kk)$.  When applied to the case $\Sigma=\cL$, this gives Theorems \ref{def of E}, \ref{indecomposables}, and \ref{general decomposition}.  In Section~\ref{sec:properties}, we analyze the construction in greater detail in order to prove Theorems \ref{props of E} and \ref{it goes up}.
In Section~\ref{sec:hom}, we give a formula for the dimension of the Hom space between two pure $\H(\cL;\kk)$-modules in terms of their stalks and costalks, and use it to characterize parity modules with only scalar endomorphisms.  Finally, in Section \ref{sec:relation with topology}, we prove Theorem \ref{topological interpretation} by showing that the hypercohomology of the parity sheaf $\cE(Y_V;\kk)$ satisfies the conditions of Theorem \ref{def of E}.

\excise{
\begin{example}\label{P2 decomp}
Suppose that $\cL$ is given by a subset of a finite $\mathbb{P}^2$. Then $\H(\mathbb{P}^2)$ is a pure $\H(\cL)$-module. By Theorem \ref{decomposition} 
(we need to understand why), when $\kk$ has characteristic zero,
\begin{multline*}
\H(\mathbb{P}^2,\kk)=\E(\cL,\kk)\ \oplus  \ (\text{semismall summand, one for each line of $\mathbb{P}^2$ containg a unique point from $\cL$})\\
\oplus \ (\text{pair of point summands, one for each line of $\mathbb{P}^2$ containing no point from $\cL$}).
\end{multline*}
We can probably find a formula that works for arbitrary $\kk$.
\end{example}
}

\subsection*{Acknowledgements} 
June Huh was partially supported by the Simons Investigator Grant
MPS-SIM-00825156 and NSF grant DMS-2053308.  Jacob Matherne received support from NSF Grant DMS-2452179 and Simons Foundation Travel Support for Mathematicians Award MPS-TSM00007970.  Nicholas Proudfoot was supported
by NSF grants DMS-1954050, DMS-2053243, and DMS-2344861, and by All Souls College.
Botong Wang was partially supported by the Simons Foundation. 
Part of this work was carried out during the 2024--2025 special year on Algebraic and Geometric Combinatorics at the Institute for Advanced Study. The authors thank Matt Larson for helpful discussions during the preparation of this manuscript.

\section{Filtrations}\label{sec:filtrations}
In this section, we define a filtration of an arbitrary $\H(\cL;\kk)$-module and use it to prove
Theorem \ref{combinatorial recursion}, assuming Theorems \ref{def of E} 
and \ref{general decomposition}.
We begin by proving that the parity module $\E(\cL;\kk)$ satisfies the following version of Poincar\'e duality.

\begin{proposition}\label{EPD}
There is an isomorphism  of $\H(\cL;\kk)$-modules $\E(\cL;\kk)^*[-\rk\cL]\cong\E(\cL;\kk)$.
\end{proposition}

It follows that the polynomial $_pZ_\cL(t)$ is palindromic: $t^{\rk \cL}   {}_pZ_\cL(t^{-1})={}_pZ_\cL(t) $.

\begin{proof}
Flabbiness of $\E(\cL;\kk)$ implies coflabbiness of $\E(\cL;\kk)^*$, and 
coflabbiness of $\E(\cL;\kk)$ implies flabbiness of $\E(\cL;\kk)^*$, so $\E(\cL;\kk)^*$ is pure.
Indecomposability of $\E(\cL;\kk)$ implies indecomposability of $\E(\cL;\kk)^*$.
Let $x\in \E^0(\cL;\kk)$ be an element such that $y_{\hat 1}\cdot x\neq 0$, and let $\varphi\in  \E^{\rk \cL}(\cL;\kk)^*$
be a linear functional that does not vanish on $y_{\hat 1}\cdot x$.  Then $y_{\hat 1}\cdot \varphi$ does not vanish on $x$, and is therefore nonzero.
Thus $\E(\cL;\kk)^*[-\rk\cL]\ \cong\E(\cL;\kk)$ by Theorem \ref{def of E}.
\end{proof}

\begin{remark}
Proposition~\ref{EPD} is equivalent to the existence of an $\H(\cL;\kk)$-invariant nondegenerate bilinear pairing $B$ on $\E(\cL;\kk)$ that satisfies the orthogonality relations
\[
B(\E^i(\cL;\kk),\E^j(\cL;\kk))=0 \ \ \text{unless $i+j=\rk \cL$.}
\]
Proposition \ref{H-E-CH}, which we will prove in Section \ref{Section:AugmentedChow}, says that $\E(\cL;\kk)$ embeds into $\CH(\cL;\kk)$,
and Proposition \ref{true facts}(b) says that the Poincar\'e pairing on $\CH(\cL;\kk)$ is nondegenerate.
Restricting the Poincar\'e pairing from $\CH(\cL;\kk)$ to $\E(\cL;\kk)$
yields a nondegenerate symmetric bilinear pairing on $\E(\cL;\kk)$.\footnote{
Let $x$ be an element of $\E^0(\cL;\kk)$ such that $y_{\hat 1}\cdot x$ is nonzero. Using Fitting's lemma, one can show that an $\H(\cL;\kk)$-module endomorphism $\varphi$ of $\E(\cL;\kk)$ is an automorphism if and only if $\varphi(x) \neq 0$.  It follows that
an $\H(\cL;\kk)$-invariant bilinear pairing $B$ on $\E(\cL;\kk)$ is nondegenerate if and only if $B(x,y_{\hat 1} \cdot x)\neq 0$. Clearly, this condition is satisfied by the restriction of the multiplicative pairing on $\CH(\cL;\kk)$ to its summand $\E(\cL;\kk)$, for any splitting of $\CH(\cL;\kk)$ by $\E(\cL;\kk)$.}
\end{remark}

For any $\H(\cL;\kk)$-module $\M$, the submodule $y_F\M$ is naturally an $\H(\cL_F;\kk)$-module. To see this, we use the surjective
 homomorphism of graded $\kk$-algebras
 \[
 \sigma_F\colon \H(\cL;\kk)\longrightarrow \H(\cL_F;\kk), \qquad y_G \longmapsto y_Fy_G,
 \]
where the product $y_Fy_G$ is computed in $\H(\cL;\kk)$ and then the resulting element, either $y_{F\vee G}$ or zero, is interpreted as an element in $\H(\cL_F;\kk)$.  For instance, $\sigma_F$ sends $y_\hz$ to $y_F$, which is the identity in $\H(\cL_F;\kk)$.
 Then the action of $\H(\cL_F;\kk)$ on $y_F\M$ is given by the formula
 \[
 \sigma_F(z)\cdot y_Fm = zy_Fm \ \ \text{for any $z \in \H(\cL;\kk)$.}
 \]
This is well-defined because the kernel of $\sigma_F$ is the kernel of multiplication by $y_F$.  Put another way, if we  
write $\widetilde y_G$ to denote the basis element of $\H(\cL_F;\kk)$ indexed by $G \in \cL_F$, we have
 \[
 \widetilde y_G\cdot y_Fm = y_Gm.
 \]
 It follows from this that, for any upper set $\Sigma \subseteq \cL_F \subseteq \cL$ and any $\H(\cL;\kk)$-module $\M$, we have
 \begin{equation}\label{eq:spanann}
 \Span_\Sigma y_F\M = \Span_\Sigma \M \;\; \text{and}\;\; \Ann_\Sigma y_F\M = \Ann_\Sigma \M \cap y_F\M,
 \end{equation}
 where the left sides of these equalities use the $\H(\cL_F;\kk)$-module structure and the right sides use the $\H(\cL;\kk)$-module structure.

\begin{lemma}\label{pure and indecomposable}
Let $\M$ be an $\H(\cL;\kk)$-module, and $y_F\M$ the corresponding $\H(\cL_F;\kk)$-module.
\begin{enumerate}[(a)]
\item If $\M$ is a pure $\H(\cL;\kk)$-module, then $y_F\M$ is a pure $\H(\cL_F;\kk)$-module.
\item For all $G\in\cL_F\subs\cL$, we have canonical stalk and costalk isomorphisms
$$\left(y_F\M\right)_G \cong \M_G[-\rk F] \and \left(y_F\M\right)_{[G]} \cong \M_{[G]}[-\rk F] ,$$
along with an identification of the costalk-stalk homomorphisms
$\beta_G(y_F\M) = \beta_G(\M)$.
\item If $\M$ is an indecomposable pure $\H(\cL;\kk)$-module, then $y_F\M$ is either zero or an indecomposable $\H(\cL_F;\kk)$-module.
\end{enumerate}
\end{lemma}

\begin{proof}
Parts (a) and (b) follow from Equation~\eqref{eq:spanann}.  For Part (c), suppose that $\M$ is indecomposable.
By Theorem~\ref{general decomposition} and Remark~\ref{characterization},
the homomorphism $\beta_G(\M)$ is equal to zero for all but one element $G$,
and the one nonvanishing homomorphism has rank $1$.  If $G$ lies in $\cL_F$, then the same property
will hold for the $\H(\cL_F;\kk)$-module $y_F\M$, and hence $y_F\M$ is an indecomposable $\H(\cL_F;\kk)$-module.  Otherwise, we have $y_F\M=0$.
\end{proof}

\begin{proposition}\label{contraction module}
There is an isomorphism of $\H(\cL_F;\kk)$-modules $y_F\E(\cL;\kk)[\rk F] \cong \E(\cL_F;\kk)$.
\end{proposition}

\begin{proof}
Lemma \ref{pure and indecomposable} implies that  $y_F\E(\cL;\kk)[\rk F]$ is a pure and indecomposable  $\H(\cL_F;\kk)$-module.
If $x$ is an element of $\E^0(\cL;\kk)$ such that $y_{\hat 1}\cdot x$ is nonzero, then $\widetilde y_{\hat 1}\cdot (y_Fx)$ is nonzero in $y_F\E(\cL;\kk)[\rk F]$.
 Therefore, by Theorem \ref{def of E}, we have $y_F\E(\cL;\kk)[\rk F] \cong \E(\cL_F;\kk)$.
\end{proof}

\begin{remark}
	This result also follows directly from our construction of the module $\E(\cL;\kk)$; see Remark \ref{rmk:alternate contraction module}.  Indeed, this isomorphism is fundamental to the inductive construction that we give.
\end{remark}

\begin{lemma}\label{sigma doesn't really matter}
Let  $\Sigma\subs \cL$ be an upper set, and let $F\in \Sigma$ be a minimal element.  
If $\M$ is a coflabby $\H(\cL;\kk)$-module, then the following natural map is an isomorphism:
	\[
	\M_F [-\rk F] = \Span_{\Sigma_{\geq F}}(\M) / \Span_{\Sigma_{>F}}(\M) \longrightarrow \Span_{\Sigma}(\M)/ \Span_{\Sigma\setminus F}(\M).
	\]
\end{lemma}

\begin{proof}
Surjectivity follows from 
$\Sigma_{\ge F} \cup (\Sigma \setminus F) = \Sigma$.
Injectivity follows from coflabbiness, since 
\[
\Span_{\Sigma_{\geq F}}(\M) \cap \Span_{\Sigma\setminus F}(\M) = \Span_{\Sigma_{\geq F}\cap (\Sigma\setminus F)}(\M)= \Span_{\Sigma_{>F}}(\M). \qedhere
\]
\end{proof}

Let $\mm$ be the unique graded maximal ideal of $\H(\cL;\kk)$, the span of all $y_F$ of positive degree. 
For any $\H(\cL;\kk)$-module $\M$,
we have a canonical filtration $\mm^0\M \sups \mm^1\M \sups \mm^2\M\sups\cdots$ of $\M$. 
The {\bf associated graded} $\gr\M$ is the vector space
\[
\gr\M \coloneq \bigoplus_{k\geq 0}\mm^k\M / \mm^{k+1}\M.
\]
Note that $\gr\M$ is (noncanonically) isomorphic to $\M$ as a graded vector space, where the grading is induced by the original grading on $\M$.  
The proof of the following proposition is adapted from the proof of \cite[Theorem 1.8]{SHCG}.

\begin{proposition}\label{associated graded}
For any coflabby $\H(\cL;\kk)$-module $\M$, there is a canonical isomorphism of graded vector spaces
	\[
	\gr\M \cong \bigoplus_{F\in \cL} \M_F[-\rk F].
	\]
\end{proposition}

\begin{proof}
Let $\Pi \coloneq \{F\in\cL\mid \rk F \geq k\}$ and $\Sigma \coloneq \{F\in\cL\mid \rk F > k\}$.
We then have
\[
\mm^k\M / \mm^{k+1}\M = \Span_{\Pi}(\M) / \Span_{\Sigma}(\M).
\]
Write $\cL^k=\{F_1, F_2, \ldots, F_r\}$. Thus, $F_1, \dots, F_r$ are precisely the elements of $\Pi$ that do not belong to $\Sigma$. 
For all $i\in\{1,\ldots,r\}$, let $\Sigma_i \coloneq \Sigma\cup\{F_i\}$.
Then
\[
\bigoplus_{F\in \cL^k} \M_F[-k] = \bigoplus_{i=1}^r \M_{F_i}[-k] \cong \bigoplus_{i=1}^r \Span_{\Sigma_i}(\M) / \Span_{\Sigma}(\M),
\]
where the second isomorphism follows from Lemma \ref{sigma doesn't really matter}.
Consider the natural map 
\[
\rho\colon\bigoplus_{i=1}^r \Span_{\Sigma_i}(\M) / \Span_{\Sigma}(\M) \longrightarrow  \Span_{\Pi}(\M) / \Span_{\Sigma}(\M).
\]
Clearly, the map $\rho$ is surjective.  To see that $\rho$ is injective, consider an element of its kernel 
\[
(m_1,\ldots,m_r)\in \Span_{\Sigma_1}(\M)\oplus\cdots \oplus\Span_{\Sigma_r}(\M).
\]
Let 
$S \coloneq \{F_i\mid m_i\notin \Span_{\Sigma}(\M)\}$,
and assume for the sake of contradiction that $S$ is both nonempty and as small as possible.  Since each component of $\rho$ is injective, we have $|S|\geq 2$.
Let $S = S_1 \cup S_2$ be a nontrivial partition of $S$, and define
\[n_1 \coloneq \sum_{F_i\in S_1} m_i\in\Span_{\Sigma\cup S_1}(\M) \and n_2 \coloneq \sum_{F_i\in S_2} m_i \in \Span_{\Sigma\cup S_2}(\M).\]
By minimality of $S$, neither $n_1$ nor $n_2$ lies in $\Span_{\Sigma}(\M)$.
But $n_1 + n_2\in \Span_{\Sigma}(\M)$, which implies that $n_1$ and $n_2$ both live in 
$\Span_{\Sigma\cup S_1}(\M)\cap \Span_{\Sigma\cup S_2}(\M) = \Span_{\Sigma}(\M)$.
\end{proof}

\begin{proof}[Proof of Theorem \ref{combinatorial recursion}]
Part (a) is immediate from the definitions.  Part (b) follows from Proposition \ref{EPD}.
Propositions \ref{contraction module} and \ref{associated graded} tell us that
$$\gr\E(\cL;\kk)\cong \bigoplus_{F\in\cL}\E(\cL;\kk)_F[-\rk F] \cong \bigoplus_{F\in\cL}\E(\cL_F;\kk)_F[-\rk F].$$
Taking Poincar\'e polynomials of both sides gives us Part (c).
\end{proof}


\section{The Poincar\'e pairing}\label{H}
In this section, we study the Poincar\'e pairing on $\H(\cL;\kk)$ and use it to prove
Theorem~\ref{E=H}, assuming Theorems ~\ref{def of E}, \ref{props of E},  and ~\ref{it goes up}.  
Let $p$ be zero or prime, and let $\kk$ be a field of characteristic $p$.

\begin{lemma}\label{constant term}
For all $\cL$, we have ${}_pP_\cL(0)=1$.
\end{lemma}

\begin{proof}
By Parts (a) and (b) of Theorem \ref{props of E}, we have $\uE^0(\cL;\kk)=\E^0(\cL;\kk)\cong\kk$.
\end{proof}

\begin{corollary}\label{lower bound}
The coefficient of $t^i$ in ${}_pZ_\cL(t)$ is greater than or equal to $|\{F\in\cL\mid\rk F = i\}|$.
\end{corollary}

\begin{proof}
This follows from Theorem \ref{combinatorial recursion} and Lemma \ref{constant term}.
\end{proof}

\begin{corollary}\label{one implication}
If ${}_pP_\cL(t)=1$, then $\E(\cL;\kk)\cong \H(\cL;\kk)$.
\end{corollary}

\begin{proof}
If ${}_pP_\cL(t)=1$, then $\uE(\cL;\kk)$ is $1$-dimensional and concentrated in degree $0$, which means
that $\E(\cL;\kk)$ is generated as an $\H(\cL;\kk)$-module by a single element $x\in\E^0(\cL;\kk)$.
This gives us a surjection $\H(\cL;\kk)\to \E(\cL;\kk)$ sending $y$ to $y\cdot x$.
By Corollary \ref{lower bound}, this must be an isomorphism.
\end{proof}

Consider the trace function $\tau\colon\H(\cL;\kk)\to\kk$ that picks out the coefficient of $y_{\hat 1}$.
We have a symmetric bilinear form on $\H(\cL;\kk)$ given by the formula $\langle\eta,\xi\rangle \coloneq \tau(\eta\xi)$.
We will call this the {\bf Poincar\'e pairing} on $\H(\cL;\kk)$.

\begin{lemma}\label{PD}
We have an isomorphism $\H(\cL;\kk)\cong\E(\cL;\kk)$ if and only if the Poincar\'e pairing
is nondegenerate.
\end{lemma}

The Poincar\'e pairing on $\H(\cL;\kk)$ is analogous to that on the cohomology 
of a singular projective variety over the complex numbers. Lemma \ref{PD} states that, in this case,
 cohomology is isomorphic to intersection cohomology if and only if the Poincar\'e pairing is nondegenerate.

\begin{proof}[Proof of Lemma \ref{PD}]
Let $d=\rk\cL$.
If the Poincar\'e pairing is nondegenerate, then it gives an isomorphism  of $\H(\cL;\kk)$-modules $\H(\cL;\kk)^*[-d]\cong\H(\cL;\kk)$.
The module $\H(\cL;\kk)$ is always coflabby, and hence $\H(\cL;\kk)^*[-d]$ is always flabby.  Consequently, if the Poincar\'e pairing
is nondegenerate, then $\H(\cL;\kk)$ is pure, and is therefore isomorphic to $\E(\cL;\kk)$ by Theorem \ref{def of E}.

Conversely, suppose that $\H(\cL;\kk)\cong\E(\cL;\kk)$.  Proposition~\ref{EPD} implies that 
there is an isomorphism $\rho\colon\H(\cL;\kk)\cong \H(\cL;\kk)^*[-d]$.  Multiplying $\rho$ by a nonzero scalar if necessary,
we may assume that $\rho(y_{\varnothing})(y_{\hat 1}) = 1$.  Then, for 
any $\eta,\xi$ in $\H(\cL;\kk)$, we have
\[
\rho(\eta)(\xi) = \rho(y_{\varnothing})(\eta\xi) = \tau(\eta\xi) = \langle\eta,\xi\rangle.
\]
The fact that $\rho$ is an isomorphism means that the Poincar\'e pairing is nondegenerate.
\end{proof}



\begin{proposition}\label{determinant}
Let $\cL$ be a projective geometry of order $q$, and $\kk$ a field of characteristic $p$.
Then $\H(\cL;\kk)\cong\E(\cL;\kk)$ if and only if $q$ is not a multiple of $p$.
\end{proposition}

\begin{proof}
Let $\D(\cL;\kk)$ denote the {\bf ungraded M\"obius algebra} of $\cL$ with coefficients in $\kk$.
This algebra has a $\kk$-basis $\{u_F\mid F\in\cL\}$ and multiplication given by  $u_F u_G = u_{F\vee G}$ for all $F$ and $G$.
It is a filtered algebra whose associated graded is canonically isomorphic to $\H(\cL;\kk)$.

Consider the trace function $\tau_{\D}\colon{\D}(\cL;\kk)\to \kk$ given by taking the coefficient of $u_{\hat 1}$, along with the 
symmetric bilinear form on ${\D}(\cL;\kk)$ given by the formula $\langle \xi,\eta\rangle_{\D} \coloneq \tau_{\D}(\eta\, \xi)$.
If we choose a total order on $\cL$ that refines the partial order, then the matrix representing $\langle -\, ,-\rangle_{\D}$
will be block lower-right triangular, with diagonal blocks given by the restricted pairing between
$\{u_F\mid \rk F = i\}$ and $\{u_F\mid \crk F = i\}$, where $\crk F \coloneq \rk \cL - \rk F$.
Setting the off-diagonal blocks to zero gives the matrix that represents the Poincar\'e pairing 
on $\H(\cL;\kk)$, thus each of the two bilinear forms is nondegenerate if and only if the other is nondegenerate.

Consider an alternate basis $\{v_F\mid F\in\cL\}$ for $\D(\cL;\kk)$,
where $$v_F \coloneq \sum_{F\leq G} \mu(F,G) u_G.$$
This basis has the property that $v_F v_G = \delta_{FG} v_F$ for all $F$ and $G$, and therefore
$$\langle v_F,v_G\rangle_{\D} = \delta_{FG}\, \mu(F,\hat 1).$$
It follows that the determinant of the bilinear form $\langle -\, ,-\rangle_{\D}$ is 
equal to $$\prod_{F\in\cL}\mu(F,\hat 1) = \prod_{F\in \cL}\left((-1)^{\crk F} q^{\binom{\crk F}{2}}\right).$$
This determinant is equal to zero in $\kk$ if and only
if $q$ is a multiple of $p$.  The proposition therefore follows from the conclusion of the previous paragraph
along with Lemma~\ref{PD}.
\end{proof}

\begin{lemma}\label{prod res}
If $\cL = \cL_1\times\cL_2$, then $\cL$ has a $p$-resonant element of corank greater than $1$
if and only if either $\cL_1$ or $\cL_2$ has a $p$-resonant element of corank greater than $1$.
\end{lemma}

\begin{proof}
We begin by noting that $F_1$ is $p$-resonant in $\cL_1$ if and only if $(F_1,\hat 1)$ is $p$-resonant in $\cL$,
and similarly $F_2$ is $p$-resonant in $\cL_2$ if and only if $(\hat 1,F_2)$ is $p$-resonant in $\cL$.
This proves the ``if'' direction.  For the other direction, it is sufficient to show that every $p$-resonant element
of $\cL$ has one of these two forms.  Indeed, suppose that $F_1\in\cL_1$ and $F_2\in\cL_2$ with neither one maximal.
Choose $G_i\in\cL_i$ covering $F_i$.  Then $\rk(G_1,G_2)= \rk(F_1,F_2)+2$, and the open interval between $(F_1,F_2)$
and $(G_1,G_2)$ has cardinality $2$, which is not congruent to $1$ modulo $p$.
\end{proof}

\begin{lemma}\label{mod res}
If $\cL$ is Boolean, then an element of $\cL$ is $p$-resonant if and only if it has corank $1$.
If $\cL$ is a projective space of order $q$, then an element of $\cL$ of corank greater than $1$ is $p$-resonant
if and only if $q$ is a multiple of $p$.
\end{lemma}

\begin{proof}
Suppose $F<G\in\cL$ with $\rk G = \rk F +2$.
If $\cL$ is Boolean, then the open interval $(F,G)$ has cardinality $2$, which is never congruent to $1$ modulo $p$.
If $\cL$ is a projective space of order $q$, then $(F,G)$ has cardinality $q+1$, which is congruent to $1$
modulo $p$ if and only if $p$ divides $q$.
\end{proof}

\begin{proof}[Proof of Theorem~\ref{E=H}]
If $\H(\cL;\kk)$ is pure, then $\H(\cL;\kk)$ satisfies the three conditions of Theorem \ref{def of E}, thus $\E(\cL;\kk)\cong\H(\cL;\kk)$.
Conversely, if $\E(\cL;\kk)\cong\H(\cL;\kk)$, then $\H(\cL;\kk)$ is pure by definition of $\E(\cL;\kk)$.  Thus condition (a)
is equivalent to condition (b).

If $\E(\cL;\kk)\cong\H(\cL;\kk)$, then $\uE(\cL;\kk)\cong \kk$, thus ${}_pP_\cL(t) = 1$.
Conversely, if ${}_pP_\cL(t) = 1$, then $\E(\cL;\kk)\cong\H(\cL;\kk)$ by Corollary \ref{one implication}.
Thus condition (c) is equivalent to condition (b).

We next prove the equivalence of conditions (d) and (e).  
Birkhoff's theorem states that a geometric lattice $\cL$ is modular if and only if $\cL\cong \cL_0\times\cL_1\times\cdots\times\cL_r$,
where $\cL_0$ is Boolean and $\cL_i$ is a projective space of order $q_i$ for all $i \geq 1$, see \cite[Chapter 12]{Welsh}.
By Lemmas \ref{prod res} and \ref{mod res}, 
such a lattice has a $p$-resonant element of corank greater than $1$ if and only if $q_i$ is a multiple
of $p$ for some $i$.
This establishes the equivalence between (d) and (e).

It remains to prove the equivalence between conditions (b) and (e).
First, we assume that we have an isomorphism $\cL\cong \cL_0\times\cL_1\times\cdots\times\cL_r$,
where $\cL_0$ is Boolean and $\cL_i$ is a projective space of order $q_i$ for some $q_i$ that
is not a multiple of $p$  for each $1 \leq i\leq r$.  By Theorem~\ref{props of E}(c) and Proposition \ref{determinant}, 
we have
$$\H(\cL;\kk)\cong \H(\cL_0;\kk)\otimes\cdots\otimes\H(\cL_r;\kk)
\cong \E(\cL_0;\kk)\otimes\cdots\otimes\E(\cL_r;\kk) \cong \E(\cL;\kk).$$
Thus condition (e) implies condition (b). 
Conversely, suppose that $\H(\cL;\kk)\cong\E(\cL;\kk)$.  This implies that ${}_pP_\cL(t) = 1$, and therefore
that ${}_0P_\cL(t) = 1$ by Theorem \ref{it goes up}.
By \cite[Proposition 2.14]{EPW}, this implies that $\cL$ is modular.
By Birkhoff's theorem, we have an isomorphism $\cL\cong \cL_0\times\cL_1\times\cdots\times\cL_r$,
where $\cL_0$ is Boolean and $\cL_i$ is a projective space of order $q_i$ for all $i \geq 1$.
By Theorem~\ref{props of E}(c), we have $\H(\cL_i;\kk)\cong \E(\cL_i;\kk)$ for all $i \ge 0$,
and by Proposition \ref{determinant},  $q_i$ is not a multiple of $p$ for all  $i \geq 1$.
This proves that condition (b) implies condition (e), and thus completes the proof of the theorem.
\end{proof}

\section{The augmented Chow ring}\label{Section:AugmentedChow}
In this section, we study the augmented Chow ring $\CH(\cL;\kk)$ as a module over $\H(\cL;\kk)$, and we use
it to prove Theorems \ref{E=IH} and \ref{E=CH} as well as Proposition \ref{resonance linear}.

Let $\cL^1$ be the set of atoms of $\cL$, and let  $S_\cL$ be the ring of polynomials in two sets of variables:
\[
S_\cL\coloneq \kk [y_i \mid \text{$i$ is an element of $\cL^1$}] \;\otimes_\kk \; \kk [x_F \mid \text{$F$ is an element of $\cL$}].
\]
The {\bf augmented Chow ring} of $\cL$ with coefficients in $\kk$ is the quotient algebra
\[
\mathrm{CH}(\cL;\kk) \coloneq S_\cL/ (I_\cL+J_\cL),
\]
where $I_\cL$ is the ideal generated by the linear forms
 \[
\sum_{F \in \cL} x_F \ \ \text{and} \ \ y_i - \sum_{i \nleq F} x_F,  \ \  \text{for every element $i$ of $\cL^1$},
\]
and $J_\cL$ is the ideal generated by  the quadratic monomials
\begin{align*}
x_{F_1}x_{F_2}, \ \  &\text{for every pair of incomparable elements $F_1$ and $F_2$ of $\cL$, and}\\
y_i \hspace{0.5mm} x_F,  \ \  &\text{for every element $i$ of $\cL^1$ and every element $F$ of $\cL$ satisfying $i \nleq F$.}
\end{align*}
By \cite[Proposition 2.18]{BHMPW}, the $\kk$-subalgebra of the augmented Chow ring generated by the elements $y_i$ is isomorphic to the graded M\"obius algebra $\mathrm{H}(\cL;\kk)$.
The {\bf Chow ring} of $\cL$ with coefficients in $\kk$ is defined to be  the stalk 
\[
\uCH(\cL;\kk) \coloneq \CH(\cL;\kk)_{\varnothing} = \CH(\cL;\kk)/\langle y_i\mid i \in \cL^1\rangle. 
\]
The Chow ring of a matroid is precisely the Chow ring of its associated Bergman fan, while the augmented Chow ring is the Chow ring of the augmented Bergman fan. For a comprehensive treatment of these constructions and their properties, we refer the reader to \cite{BHMPW}.

 We will be particularly interested in the class
$\b_\cL \coloneq -x_{\varnothing}$ in $\uCH(\cL;\kk)$.

\begin{lemma}\label{bases}
\
\begin{enumerate}[(a)]
\item The space $\CH^1(\cL;\kk)$ has $\kk$-basis $\{x_F\mid \rk F \geq 1\}$.
\item The space $\uCH^1(\cL;\kk)$ has $\kk$-basis $\{x_F\mid \rk F \geq 2\}$.
\end{enumerate}
\end{lemma}

\begin{proof}
The  relation $\sum_{F \in \cL} x_F$ allows us to express $x_{\varnothing}$ in terms of $x_F$ for $F$ of positive rank in $\CH(\cL;\kk)$.
In the Chow ring, where $y_i=0$ for all $i$, we have
$x_i = -\sum_{i<F} x_F$ for any $i \in \cL^1$.
\end{proof}

For any $G\in \cL\setminus \{\hat{0}\}$, we let $m_G$ be the number of atoms of $\cL^G$.

\begin{lemma}\label{beta}
In the Chow ring $\uCH(\cL;\kk)$, we have the identity
\[
\b_\cL = \sum_{\rk G \geq 2} (1-m_G)x_G.
\]
\end{lemma}

\begin{proof}
For any element $F$ of rank $1$, we have
$$\sum_{F\leq G} x_G = \sum_G x_G - \sum_{F\not\leq G}x_G = 0 - 0 = 0.$$
It follows that
$$\sum_{\rk G \geq 1} m_G x_G = \sum_{\substack{F\leq G\\ \rk F = 1}} x_G = \sum_{\rk F = 1}\sum_{F\leq G} x_G = 0.$$
We therefore have 
$$\sum_{\rk G \geq 2} (1-m_G)x_G = \sum_{\rk G \geq 1} (1-m_G)x_G = \sum_{\rk G \geq 1} x_G - \sum_{\rk G \geq 1} m_Gx_G = \sum_{\rk G \geq 1} x_G.$$
Since $\sum_G x_G=0$, this is equal to $-x_{\varnothing} = \b_\cL$.
\end{proof}

\begin{lemma}\label{two is enough}
If $m_G$ is congruent to $1$ modulo $p$ for every element $G\in\cL$ of rank $2$,
then $m_G$ is congruent to $1$ modulo $p$ for every element $G\in\cL$ of rank at least $2$.
\end{lemma}

\begin{proof}
Let $G$ be given of rank at least $2$, and fix a rank $1$ element $F\leq G$.
There is a surjective map \[\nu\colon(\cL^G)^1 \smallsetminus \{F\}\longrightarrow (\cL^G_F)^1 = \cL^F_G \cap \cL^2\]
that takes $F'$ to $F\vee F'$, and the cardinality of $\nu^{-1}(H)$ is equal to $m_H-1$.
We therefore have
\[m_G-1 = \left|(\cL^G)^1 \smallsetminus \{F\}\right|
= \sum_{H \in (\cL^G_F)^1
} \left(m_H-1\right).\]
Thus, if $m_H$ is congruent to $1$ for all $H$ of rank $2$, then $m_G$ is congruent to $1$, as well.
\end{proof}

Recall that an element $F\in\cL$ is $p$-resonant if and only if $F\neq\hat 1$
and, for all $G>F$ with $\rk(G) = \rk(F) + 2$, the cardinality of the open interval $(F,G)$
is congruent to $1$ modulo $p$.

\begin{corollary}\label{beta resonance}
We have $\b_\cL=0\in\uCH^1(\cL;\kk)$ if and only if  $\varnothing\in\cL$ is $p$-resonant or $\rk\cL=0$.
\end{corollary}

\begin{proof}
We have $\hz \neq \ho$ if and only if $\rk\cL > 0$, and the second condition for $p$-resonance of $\hz$ is equivalent to the statement that $\b_\cL=0$
by Lemmas \ref{bases}, \ref{beta}, and \ref{two is enough}.
\end{proof}

The following statements are proved in \cite{BHMPW} and \cite{SHCG} when $\kk=\Q$, and the proofs go through
for arbitrary $\kk$ without modification.

\begin{proposition}\label{true facts}
Let $\CH(\cL;\kk)$ be the augmented Chow ring of $\cL$ with coefficients in a field $\kk$.
\begin{enumerate}[(a)]
\item The $\H(\cL;\kk)$-module homomorphism from $\H(\cL;\kk)$ to $\CH(\cL;\kk)$ sending $1$ to $1$
is injective \cite[Proposition 2.18]{BHMPW}.
\item The Poincar\'e pairing on $\CH(\cL;\kk)$ is nondegenerate \cite[Theorem 1.6(4)]{BHMPW}.
\item 
For all $F\in\cL$, there is a canonical isomorphism 
$$y_F\CH(\cL;\kk)\cong \CH(\cL_F;\kk)[-\rk F]$$
of  $\H(\cL_F;\kk)$-modules \cite[Propositions 2.15 and 2.17]{SHCG}.
\item 
If $\rk\cL>0$, the ideals $\langle x_{\hat 0}\rangle $ and $\langle y_i\mid i \in \cL^1\rangle$ are mutual annihilators
\cite[Lemma 5.3]{SHCG}.
\item 
The $\H(\cL;\kk)$-module $\CH(\cL;\kk)$ is pure \cite[Proposition 5.14]{SHCG}.
\end{enumerate}
\end{proposition}

The following proposition is proved in \cite[Proposition 5.9]{SHCG} without explicit reference to the commutative diagram;
we provide the full proof here for completeness.

\begin{proposition}\label{Chow stalks}
We have $\CH(\cL;\kk)_{\hat 1} \cong \kk \cong \CH(\cL;\kk)_{[\hat 1]}$,
and the costalk-stalk homomorphism $\beta_{\hat 1}(\CH(\cL;\kk))$ is an isomorphism.
For all $F\neq \hat 1$, we have isomorphisms
\[
  \CH(\cL;\kk)_{[F]}\cong \uCH(\cL_F;\kk)[-1] \and  \CH(\cL;\kk)_{F} \cong \uCH(\cL_F;\kk) 
\]
such that the following diagram commutes:
	\[\begin{tikzcd}
	\CH(\cL;\kk)_{[F]} \arrow[rrr, "\beta_{F}(\CH(\cL;\kk))"]\arrow[d, "\cong"]  &&& \CH(\cL;\kk)_{F} \arrow[d, "\cong"] \\
	\uCH(\cL_F;\kk)[-1] \arrow[rrr, "\cdot\b_{\cL_F}"]            &&& \uCH(\cL_F;\kk).               
	\end{tikzcd}\]
\end{proposition}

\begin{proof}
The first statement is very general: for any $\H(\cL;\kk)$-module $\M$,
we have
\[
\M_{\hat 1} = y_{\hat{1}}\M[\rk \cL]  = \M_{[\hat 1]},
\]
and the costalk-stalk map $\beta_{\hat 1}(\M)$ is an isomorphism.

We prove the second statement for $F = \varnothing$. 
By part (d) of Proposition \ref{true facts}, we have
\[
\CH(\cL;\kk)_{[\varnothing]} = x_{\varnothing}\CH(\cL;\kk)\cong \uCH(\cL;\kk)[-1],
\]
where the isomorphism  sends $x_{\varnothing}x_F$ to $-x_F$.
Using this isomorphism, the composition
\[
\begin{tikzcd}\uCH(\cL;\kk)[-1]\cong \CH(\cL;\kk)_{[\varnothing]} \arrow[rr, "\beta_{\varnothing}(\CH(\cL;\kk))"] &&
\CH(\cL;\kk)_{\varnothing}\cong \uCH(\cL;\kk)\end{tikzcd}
\]
is given by multiplication by $-x_{\varnothing} = \b_{\cL}$.
The second statement for $\varnothing < F < \hat 1$ follows from
the statement for $F = \varnothing$ using part (c) of Proposition \ref{true facts} and Lemma \ref{pure and indecomposable}.\end{proof}

\begin{proof}[Proof of Theorem \ref{E=IH}]
By \cite[Theorem 3.21]{SHCG} and \cite[Proposition 6.4]{SHCG}, $\IH(\cL)$ is an
indecomposable direct summand of $\CH(\cL;\Q)$ that contains $\CH^0(\cL;\Q)$, and its endomorphism ring as an $\H(\cL;\Q)$-module consists only of scalars.  
In particular, it is pure by Proposition \ref{true facts}(e), and it has an element $1\in\IH^0(\cL)=\CH^0(\cL;\Q)$ with
the property that $y_{\hat 1}\cdot 1\neq 0$.
Therefore, by Theorem~\ref{def of E}, $\IH(\cL)$ is isomorphic to  $\E(\cL;\Q)$ as $\H(\cL;\Q)$-modules, and this isomorphism is unique up to a nonzero scalar.
\end{proof}

\begin{proof}[Proof of Proposition \ref{resonance linear}]
By Proposition \ref{true facts}(e), Proposition~\ref{Chow stalks}, 
and Theorem \ref{general decomposition}, there is an isomorphism
\begin{equation}\label{Chow decomposition}\CH(\cL;\kk) \cong \E(\cL;\kk) \oplus \bigoplus_{F<\hat 1}\, \b_{\cL_F}\uCH(\cL_F;\kk)\otimes\E(\cL^F;\kk).\end{equation}
By Corollary \ref{beta resonance}, $\b_{\cL_F}=0$ if and only if $F$ is $p$-resonant,
and hence  \begin{equation}\label{degree 1}\CH^1(\cL;\kk) \cong \E^1(\cL;\kk)\oplus \kk^r,\end{equation}
where $r$ is the number of $F<\hat 1$ which are not $p$-resonant.
Similarly, by taking stalks at $\varnothing$ of both sides of Equation \eqref{Chow decomposition}, we obtain the equation
\begin{equation}\label{underline degree 1}\uCH^1(\cL;\kk) \cong \uE^1(\cL;\kk)\oplus \kk^r.\end{equation}
By parts (a) and (b) of Theorem \ref{props of E}, $\E^0(\cL;\kk)$ and $\uE^0(\cL;\kk)$ are both $1$-dimensional.
By Lemma \ref{bases}, the dimension of $\CH^1(\cL;\kk)$ is equal to the cardinality of $\cL$ minus one,
and the dimension of $\uCH^1(\cL;\kk)$ is $|\cL| - |\cL^1| - 1$.
The proposition therefore follows from Equations \eqref{degree 1} and \eqref{underline degree 1}.
\end{proof}

\begin{remark}
In Section \ref{sec:next}, we give an alternative proof of Proposition \ref{resonance linear} that does not use the augmented Chow ring.
\end{remark}

\begin{proof}[Proof of Proposition~\ref{H-E-CH}]
Theorem~\ref{general decomposition} implies that $\E(\cL;\kk)$ is a direct summand of $\CH(\cL;\kk)$, yielding the second inclusion.
Since $\E^0(\cL;\kk)=\CH^0(\cL;\kk)$, the first inclusion follows from  part (a) of Proposition~\ref{true facts}.
\end{proof}

\begin{remark}
The inclusion $\H(\cL;\kk)\subseteq\E(\cL;\kk)$
can alternatively be deduced from the injectivity of the map $\rho$ appearing in the proof of Proposition~\ref{associated graded}.
\end{remark}

\begin{proof}[Proof of Theorem \ref{E=CH}]
By Equation \eqref{Chow decomposition}, $\E(\cL;\kk)$ is isomorphic to a direct summand of $\CH(\cL;\kk)$,
and hence conditions (a), (b), and (c) are equivalent.  By Equation \eqref{Chow decomposition} and Corollary \ref{beta resonance}, 
conditions (b) and (e) are equivalent.
By taking stalks at $\varnothing$, we see that condition (b) implies condition (d).  Finally, condition (d) implies condition (e) by Equation
\eqref{underline degree 1}.
\end{proof}


\section{Inclusions of lattices}\label{sec:inclusion}
In this section, we prove Theorems~ \ref{pullback of pure}, ~\ref{decomposition}, 
and ~\ref{Kahler}.
As in Section \ref{inclusions of lattices}, we consider a rank-preserving inclusion of geometric lattices 
$\varphi\colon \cL\to \cM$.  We begin by explaining why such an inclusion always comes from
the operation of matroid restriction.


We assume that $\varphi\colon \cL\to \cM$ is an inclusion of posets satisfying
\[
\rk_\cM(\varphi(F))=\rk_\cL(F)
\ \ \text{for every $F\in\cL$}.
\]
Then $\varphi$ preserves joins. Indeed, since every element of $\cL$ is a join of atoms, it is enough to show that
$\varphi(F\vee G)=\varphi(F)\vee\varphi(G)$ for any $F\in \cL$ and $G\in \cL^1$. 
 If $G\le F$, both sides are equal to $\varphi(F)$. If $G\not\le F$, then, since $\varphi$ is order-preserving,
$\varphi(F\vee G)\ge \varphi(F)\vee\varphi(G)$,
and
\[
\rk_\cM \varphi(F\vee G)
=\rk_\cL(F\vee G)
=\rk_\cL(F)+1
=\rk_\cM(\varphi(F))+1
=\rk_\cM\bigl(\varphi(F)\vee\varphi(G)\bigr).
\]
It follows that $\varphi(F\vee G)=\varphi(F)\vee\varphi(G)$, as claimed.

Recall that the geometric lattice $\cM$ may be viewed as the lattice of flats of a simple matroid with ground set $\cM^1$ by identifying $F \in \cM$ with the set
\[
a_\cM(F) \coloneq \{G \in \cM^1\mid G \le F\}.
\]
 In the same way, $\cL$ may be viewed as the lattice of flats of a simple matroid on the ground set $\cL^1$.

\begin{lemma}\label{lem:atoms under restriction}
	For any $F \in \cL$, we have $\varphi^{-1}(a_\cM(\varphi(F))) = a_\cL(F)$.
\end{lemma}
\begin{proof}
	This is obvious when $\rk(F)\leq 1$. If the lemma holds for a particular  $F\in \cL$, 
	then, by the partition axiom for flats of a matroid, it also holds for every $G$ that covers $F$.  Thus, by induction on rank, the statement holds for all $F$.
\end{proof}

Define a map $\delta \colon \cM \to \cL$ by letting 
$\delta(F)$ be the join of the elements of $\varphi^{-1}(a_\cM(F))$ for each $F\in \cM$.  
The map $\delta$ is clearly order-preserving, although it need not be injective.
The composition $\delta \circ \varphi$ is the identity on $\cL$:  indeed, the previous lemma shows that $\delta\big( \varphi(F) \big)$ is the join of $a_\cL(F)$, which is equal to $F$.

The composition $\varphi\circ\delta$ need not be the identity. However, since $\varphi$ preserves joins, $\varphi\bigl(\delta(F)\bigr)$ is the join of the elements of
$\varphi\bigl(\varphi^{-1}(a_\cM(F))\bigr)\subseteq a_\cM(F)$. 
It follows that $\varphi\bigl(\delta(F)\bigr)\leq F$ for every $F\in\cM$.

\begin{proposition}\label{prop:rank-preserving}
	The geometric lattice $\cL$ is naturally isomorphic to the lattice of flats of the restriction of $\cM$ to $S\coloneq \varphi(\cL^1)$, or equivalently, to the lattice of flats of the matroid obtained from $\cM$ by deleting all elements of $\cM^1\setminus S$.
\end{proposition}

\begin{proof}
	Considering every geometric lattice as a matroid, the lattice of flats of the restriction $\cM|S$ is equal to the set of all intersections $a_\cM(F)\cap S$ for $F\in \cM$.  Since $\varphi$ is injective, it induces a bijection between $\cL^1$ and $S$, and so as posets,
	\[
	\cM|S\cong \{\varphi^{-1}(a_\cM(F)) \mid F \in \cM\}.
	\]
Since $\delta$ is surjective, it suffices to prove that
	\[
	a_\cL(\delta(F)) = \varphi^{-1}(a_\cM(F)) \ \ \text{for all $F\in \cM$.}
	\]
	 Since $\delta(F)$ is defined to be the join of the elements on the right hand side, we have 
	 \[
	 \varphi^{-1}(a_\cM(F)) \subseteq a_\cL(\delta(F)).
	 \]
	   On the other hand, using Lemma \ref{lem:atoms under restriction} and the fact that $\varphi\big(\delta(F)\big) \le F$,
we get 
	\[a_\cL(\delta(F)) = \varphi^{-1}\left(a_\cM\left(\varphi\left(\delta(F)\right)\right)\right) \subseteq \varphi^{-1}(a_\cM(F)). \qedhere\]
\end{proof}

We now study how the maps $\varphi$ and $\delta$ interact with upper sets of $\cL$ and $\cM$. For an upper set $\Sigma\subs \cL$, define its \textbf{\boldmath{$\cM$}-saturation}  by
\[
\Sigma^+ \coloneq \{F \in \cM \mid F \ge \varphi(G) \;\text{for some}\; G \in \Sigma\}.
\]
It is the smallest upper set containing $\varphi(\Sigma)$.  For any $\H(\cM)$-module $\M$ and any upper set $\Sigma\subs \cL$, we have
\[
\Ann_\Sigma(\M) = \Ann_{\Sigma^+}(\M) \and \Span_\Sigma(\M) = \Span_{\Sigma^+}(\M),
\]
where on the left-hand sides of the two equations we regard $\M$ as a module over $\H(\cL;\kk)$,
and on the right-hand sides of the two equations we regard $\M$ as a module over $\H(\cM;\kk)$.

\begin{lemma}
	For any upper set $\Sigma\subs \cL$, we have $\Sigma^+ = \delta^{-1}(\Sigma)$.
\end{lemma} 
\begin{proof}
	Suppose first that $F\in \delta^{-1}(\Sigma)$, so that $\delta(F)\in\Sigma$. Since
$F\geq \varphi\bigl(\delta(F)\bigr)$,
we have  $F\in\Sigma^+$. Conversely, suppose that $F\in\Sigma^+$. Then there is $G\in\Sigma$ such that $F\geq \varphi(G)$. 
 It follows that $\delta(F) \ge \delta(\varphi(G))=  G$, and hence $\delta(F) \in \Sigma$.
\end{proof}

\begin{corollary}
	For any pair of upper sets $\Sigma_1, \Sigma_2\subs\cL$, we have 
	\[(\Sigma_1\cap \Sigma_2)^+ = \Sigma_1^+ \cap \Sigma_2^+.\] 
\end{corollary}

\begin{proof}[Proof of Theorem~\ref{pullback of pure}]
For any pair of upper sets $\Sigma_1, \Sigma_2\subs\cL$, we have 
\begin{align*}
\Ann_{\Sigma_1}(\M) + \Ann_{\Sigma_2}(\M) & = \Ann_{\Sigma^+_1}(\M) + \Ann_{\Sigma^+_2}(\M) = \Ann_{\Sigma^+_1\cap\Sigma^+_2}(\M) \\
& = \Ann_{(\Sigma_1\cap\Sigma_2)^+}(\M)  = \Ann_{\Sigma_1 \cap \Sigma_2}(\M),	
\end{align*}
and hence $\M$ is flabby as a module over $\H(\cL;\kk)$.  The proof of coflabbiness is similar.
\end{proof}

\begin{proof}[Proof of Theorem \ref{decomposition}]
By Theorem \ref{general decomposition}, there is an $\H(\cL;\kk)$-module isomorphism
\[
\E(\cM;\kk) \cong \bigoplus_{F\in\cL} \im(\beta_F(\E(\cM;\kk))) \otimes\E(\cL^F;\kk).
\]
Thus we need only show that $\im(\beta_F(\E(\cM;\kk)))\cong \oL(\varphi_F;\kk)$ for all $F\in\cL$.
When $F=\varnothing$, this is immediate from the definition of $\oL(\varphi;\kk)$; see Example \ref{beta0}.  For arbitrary $F$,
by Lemma \ref{pure and indecomposable} and Proposition \ref{contraction module}, we have
\[
\beta_F(\E(\cM;\kk))\cong \beta_F(\E(\cM_{\varphi(F)};\kk)).
\]
Under this isomorphism, $F$ is the minimal element of the geometric lattice $\cL_F$, so the latter agrees with  $\oL(\varphi_F;\kk)$.
\end{proof}

\begin{proof}[Proof of Theorem \ref{Kahler}]
The palindromicity of the local $h$-polynomials follows from
Corollary~\ref{Z and l}, given the palindromicity of the Z-polynomials from Theorem~\ref{combinatorial recursion}. To be precise, say that a polynomial $f(t)$ is
$d$-palindromic if $f(t) = t^d f(t^{-1})$.  We claim that the
local $h$-polynomial of $\varphi$ is $d$-palindromic with $d=\rk(\cM)$.

The statement is proved by induction on $c=\rk(\cL)$.  When $c=0$, the local
$h$-polynomial is $Z_{\cM}(t)$, which is $d$-palindromic by Theorem~\ref{combinatorial recursion}.  When
$c>0$, the left-hand side of the equation of Corollary~\ref{Z and l} is $d$-palindromic,
and all of the terms with $F \ne \hat{0}$ are a product of a
$\rk(\M_{\varphi_{F}})=\crk(\varphi(F))$-palindromic polynomial and a
$\rk(\cL^F) = \rk(F)=\rk(\varphi(F))$-palindromic polynomial.  Since
multiplying palindromic polynomials adds the degree of palindromicity,
the result follows.

The proof of the unimodality of ${}_p\ell_\varphi(t)$ when $p=0$ relies on the decomposition theorem for Lefschetz modules in \cite{AHL}. 
We will show in Proposition~\ref{prop:base change for IH} that extending the coefficient field does not change the structure of parity modules. Thus, when $p=0$, we may assume without loss of generality that $\kk=\mathbb{R}$. By \cite[Theorem 1.6]{SHCG}, the module $\E(\cM;\mathbb{R})\cong \IH(\cM;\mathbb{R})$ satisfies Poincar\'e duality, the hard Lefschetz property, and the Hodge--Riemann relations with respect to the interior of the cone generated by $\{y_i\mid i\in\cM^1\}$. It is therefore a Lefschetz module in the sense of \cite{AHL}. 
We apply the decomposition package of \cite{AHL} to the subalgebra of $\H(\cM;\mathbb{R})$ generated by $\{y_i\mid i\in\cL^1\}$, which is naturally identified with $\H(\cL;\mathbb{R})$. In particular, \cite[Corollary 1.8]{AHL} yields the relative hard Lefschetz theorem for $\oL(\varphi;\mathbb{R})$, from which it follows that the polynomial $\ell_\varphi(t)$ is palindromic and unimodal.
\end{proof}

\section{Exploding $\H(\cL;\kk)$-modules}\label{sec:sheafifying}
We now  begin working toward the proofs 
of Theorems \ref{def of E}, \ref{props of E}, \ref{topological interpretation}, \ref{indecomposables}, \ref{it goes up}, and \ref{general decomposition}.
As a first step, we introduce a new category $\cC_{\reg}(\cL;\kk)$ and prove that it is equivalent to the category of
finitely generated graded $\H(\cL;\kk)$-modules.  

Though we have assumed until now that $\kk$ is a field,
we will now relax this assumption and allow $\kk$ to be any local principal ideal domain, meaning that $\kk$
is either a field or a discrete valuation ring.  We will be particularly interested in the case of the $p$-adic integers,
which will provide a framework for comparing parity modules in characteristics $0$ and $p$, leading ultimately to a proof of Theorem \ref{it goes up}.

\subsection{The maps \boldmath{$\varphi$} and \boldmath{$\psi$}}\label{sec:exploded modules}
Let $\M$ be an $\H(\cL;\kk)$-module.  For any $F\in\cL$, let $\M(F) \coloneq y_F\M[\rk F]$. 
We can factor the multiplication by $y_F$ as follows: 
\[
\xymatrix{
	\M  \ar[rr]^{y_F} \ar[dr]_{\varphi^F}&& \M[\rk F] 
	\\
	& \M(F) \ar[ur]_{\psi^F}&
}
\]
Note that $\varphi^F$ is surjective and $\psi^F$ is injective, and  $\M(F)$ is an $\H(\cL_F;\kk)$-module via the construction in Section \ref{sec:filtrations}.  More generally, given $G\leq F\in\cL$, we choose an element
$H\in\cL^{\rk F - \rk G}$ with $G\vee H = F$, and define a commutative diagram 
\[
\xymatrix{
	\M(G)  \ar[rr]^{y_H} \ar[dr]_{\varphi_G^F}&& \M(G)[\rk F-\rk G] 
	\\
	& \M(F) \ar[ur]_{\psi_G^F}&
}
\]
by setting  $\varphi_G^F(y_G m) \coloneq y_Hy_Gm = y_Fm$ and $\psi_G^F(y_Fm) \coloneq y_Fm$ for all $m\in \M$.
When $G=\varnothing$, we have $\varphi^F_{\varnothing} = \varphi^F$ and $\psi^F_{\varnothing} = \psi^F$.
In general, the element $H$ with $G\vee H = F$ is not unique, but these maps do not depend on the choice.



\begin{proposition}\label{prop:psi phi relations} The maps $\{\varphi_G^F\mid G\leq F\}$ and $\{\psi_G^F\mid G\leq F\}$ satisfy the following relations:
	\begin{enumerate}
	\item[(a)] We have $\varphi_F^F = \operatorname{id}_{\M(F)} = \psi^F_F$.
		\item[(b)] If $H \le G \le F$, then 
		$\varphi_G^F \varphi_H^G = \varphi_H^F$ and 
		$ \psi_H^G\psi_G^F = \psi_H^F$.
		\item[(c)] For any elements $F, G, H\in\cL$ with $H \le F \wedge G$, we have 
		\[
\varphi_H^G \psi_H^F
=
\begin{cases}
\psi_G^{F\vee G}\,\varphi_F^{F\vee G}
& \text{if \(F\) and \(G\) form a modular pair and \(H=F\wedge G\),}\\[2pt]
0
& \text{otherwise.}
\end{cases}
\]
	\end{enumerate}
\end{proposition}

Note that the map $\varphi^G_{H}\psi_{H}^F$ sends $\M(F)$ to $\M(G)[\rk F - \rk H]$, while $\psi_G^{F\vee G}\varphi_F^{F\vee G}$ sends $\M(F)$ to $\M(G)[\rk (F\vee G) - \rk G]$.  The condition that $F$ and $G$ form a modular pair and $H = F\wedge G$ is exactly what is needed for these maps to have the same shift.

\begin{proof}[Proof of Proposition \ref{prop:psi phi relations}]
Statements (a) and (b) are straightforward.  To prove statement (c), 
choose $K\in\cL^{\rk F-\rk H}$ with
$y_F = y_Hy_K$ and $L\in\cL^{\rk G - \rk H}$ with $y_G = y_Hy_L$.  For an element $y_Fm \in \M(F)$, we have 		
		\[\varphi^G_{H}\psi_{H}^F(y_Fm) = \varphi^G_H(y_Fm) = y_Ly_Fm.\] 
	Note that $F \vee L = (H\vee K) \vee L = (H \vee K)\vee (H \vee L) = F\vee G$, and 
		\[\rk F + \rk L = \rk H + \rk K + \rk L= \rk F + \rk G - \rk H.\]
	If $F$ and $G$ form a  modular pair and $H=F\wedge G$, this implies that $\rk F + \rk L = \rk (F\vee G) = \rk(F\vee L)$, and so $y_Ly_F = y_{L\vee F} = y_{F\vee G}$.  This implies that 
		\[\varphi^G_{H}\psi_{H}^F(y_Fm) = y_Ly_Fm = y_{F\vee G}m = \psi_G^{F\vee G}(y_{F\vee G}m) = \psi_G^{F\vee G}\varphi_F^{F\vee G}(y_Fm).\] 
	On the other hand, if $F$ and $G$ do not form a modular pair or $H<F\wedge G$, then  $\rk L + \rk F > \rk(L\vee F)$, and hence $y_Ly_Fm = 0$. 
	%
\end{proof}

\begin{remark}
	To improve readability, we will sometimes use an alternative notation for compositions of the maps $\{\varphi_G^F\mid G\leq F\}$ and 
	$\{\psi_G^F\mid G\leq F\}$ by listing a sequence of elements, from right to left and with each one comparable with the next, separated by $\swa$ or $\nwa$ to indicate the order relation.  For instance, the relation (b) says that
	\[(F \nwa G \nwa H) = (F \nwa H)\;\;\text{and}\;\; (H \swa G \swa F) = (H \swa F),\]
	while relation (c) implies that
	\[(G\nwa (F\wedge G) \swa F) = (G \swa (F\vee G)\nwa F)\]
	whenever $F$ and $G$ form a modular pair.
\end{remark}

\subsection{The category \boldmath{$\cC(\cL;\kk)$}}
Proposition \ref{prop:psi phi relations} motivates the following definition.  Consider the category $\cC(\cL;\kk)$ whose objects are tuples $\cV = \left(\{\cV(F)\}, \{\varphi_G^F\}, \{\psi_G^F\}\right)$ consisting of a graded $\kk$-module $\cV(F)$ 
for each $F\in \cL$ along with graded $\kk$-linear maps $$\varphi_G^F \colon \V(G) \to \V(F)\and  \psi_G^F\colon \V(F) \to \V(G)[\rk F - \rk G]$$ for each $G \le F$, satisfying the three conditions of Proposition \ref{prop:psi phi relations}.  
A morphism from $\cV$ to $\cW$ is a system of linear maps $\{\cV(F) \to \cW(F)\mid F\in\cL\}$ commuting with the maps 
$\{\varphi_G^F\mid G\leq F\}$ and $\{\psi_G^F\mid G\leq F\}$.

Let $\cC_{\reg}(\cL;\kk)$ be the full subcategory of $\cC(\cL;\kk)$ consisting of objects such that the maps $\varphi_G^F$ are all surjective and the maps $\psi_G^F$ are all injective.  We call such objects {\bf regular}.
There is a functor $\Lambda\colon\H(\cL;\kk)\mmod \to \cC_{\reg}(\cL;\kk)$ taking $\M$ to $(\{\M(F)\}, \{\varphi_G^F\}, \{\psi_G^F\})$ as defined in Section \ref{sec:exploded modules}.  We refer to $\Lambda\M$ as the {\bf explosion} of $\M$.

\begin{proposition}\label{thm:main equivalence}
	The explosion functor $\Lambda\colon\H(\cL;\kk)\mmod \to \cC_{\reg}(\cL;\kk)$ is an equivalence of categories.
\end{proposition}

\begin{proof}
	Define a functor $\Theta\colon \cC(\cL;\kk) \to \H(\cL;\kk)\mmod$ by sending an object $(\{\cV(F)\}, \{\varphi_G^F\}, \{\psi_G^F\})$ to $\cV(\hz)$, with the $\H(\cL;\kk)$-module structure defined by letting $y_F$ act by $\psi_{\hz}^F\varphi_{\hz}^F$.  Let us check that this  defines an $\H(\cL;\kk)$-module.  Let $F$ and  $G$ be given elements of $\cL$.  If $\rk (F\vee G) = \rk F + \rk G$, then $(F, G)$
is a modular pair and $F\wedge G=\hz$.  In this case, by properties (b) and (c) of Proposition \ref{prop:psi phi relations}, 
for any $a \in \cV(\hz)$, we have
	\begin{align*}
		y_Fy_Ga & = (\hz \swa F\nwa \hz \swa G\nwa \hz)(a) \\	
		& = (\hz \swa F \swa (F \vee G)\nwa G\nwa \hz)(a) \\
		& = (\hz \swa (F\vee G)\nwa \hz)(a) = y_{F\vee G}a.
	\end{align*}
On the other hand, if $\rk(F\vee G) < \rk F + \rk G$, then either $F$ and $G$ do not form a modular pair or
$F\wedge G \neq \hz$, so $(F\nwa \hz \swa G) = 0$.  This implies that $y_Fy_G$ annihilates $\cV(\hz)$, as required.
	
It is clear that, for any $\H(\cL;\kk)$-module $\M$, we have a natural isomorphism $\Theta\Lambda\M \cong \M$.
In the other direction, if $\V = (\{\cV(F)\}, \{\varphi_G^F\}, \{\psi_G^F\})$ is regular and $\M = \Theta\V$, then there is a natural isomorphism
\[y_F\M = \im(\psi_\hz^F\varphi_\hz^F) \cong \cV(F)[-\rk F],\]
since $\psi_\hz^F$ is injective and $\varphi_\hz^F$ is surjective.  Putting these maps together gives a natural isomorphism $\V\cong \Lambda \Theta\V$.
\end{proof}

\begin{example}
Consider the $\H(\cL;\kk)$-module $\CH(\cL;\kk)$.  
For each $F\in\cL$, we have 
$$\Lambda(\CH(\cL;\kk))(F) = y_F\CH(\cL;\kk)[\rk F] \cong \CH(\cL_F;\kk),$$
where the last isomorphism comes from Proposition \ref{true facts}(c).
For any $G\le F$, the $\kk$-module homomorphism $\varphi_G^F\colon \CH(\cL_G;\kk)\to \CH(\cL_F;\kk)$ agrees with the homomorphism
$\varphi^{\cL_G}_{F\setminus G}$ from \cite[Proposition 2.28]{BHMPW}, while
$\psi_G^F$ agrees with the homomorphism $\psi^{\cL_G}_{F\setminus G}$ from \cite[Proposition 2.29]{BHMPW}.
\end{example}

\begin{remark}
The category $\cC(\cL;\kk)$ can also be viewed as the category of finitely generated graded modules for 
the $\kk$-path algebra $C$ of a quiver with relations.  The quiver has vertices given by elements of $\cL$, and arrows
in both directions between $F$ and $G$ whenever $F$ covers $G$.
If $F$ covers $G$, then we have the relation $(F \nwa G \swa F) = 0$.  If $F\neq F'$ both cover $G$ and $H = F\vee F'$, then we have the relations
\[(F' \nwa G \swa F) = (F' \swa H \nwa F)\]
and
\[(H \swa G \swa F) = (H \swa G' \swa F); \;\; (F \nwa G \nwa H) = (F \nwa G' \nwa H).\] 
The algebra $C$ is graded by assigning degree zero to the upward arrows and degree one to the downward arrows.
If $e$ is the idempotent corresponding to the path of length $0$ at the vertex $\hz$, then
$eCe\cong \H(\cL;\kk)$, and the functor $\Theta$ of Proposition \ref{thm:main equivalence} is identified with the functor
from $C\mmod$ to $eCe\mmod$ taking $\M$ to $e\M$.  Note, however, that the explosion functor $\Lambda$ is {\em not} identified 
either with the left adjoint $\N \mapsto Ce\otimes_{eCe} \N$ nor with the right adjoint $\N \mapsto \Hom_{eCe}(eC, \N)$.  Indeed, 
$\Lambda$ is neither left nor right exact.  The nice feature of the functor $\Lambda$ is that it interacts well with duality; see Section \ref{sec:freeness}.
\end{remark}

\subsection{Localization}\label{sec:localization}
One of the main reasons to consider the category $\cC_\reg(\cL;\kk)$ instead of the equivalent category $\H(\cL;\kk)\mmod$ is that we can consider ``local" categories on subsets of $\cL$.  Given an upper set $\Sigma\subseteq \cL$, let $\cC(\Sigma;\kk)$ denote the category whose objects are tuples $(\{\cV(F)\}, \{\varphi_G^F\}, \{\psi_G^F\})$ satisfying the same properties as before, but with $F, G \in \Sigma$.  
Similarly, we let $\cC_{\reg}(\Sigma;\kk)$ be the full subcategory of $\cC(\Sigma;\kk)$ consisting of objects such that the maps $\varphi_G^F$ are all surjective and the maps $\psi_G^F$ are all injective.  
If $\Sigma'\subseteq \Sigma$ is another upper set, there is a forgetful functor $\cC(\Sigma;\kk) \to \cC(\Sigma';\kk)$, which we denote $\V \mapsto \V|_{\Sigma'}$.  

If we consider $\cL$ as a topological space whose open sets are the upper sets, then the categories $\cC(\Sigma;\kk)$ behave like sheaves on $\cL$, in the sense that objects and morphisms glue.  That is, given upper sets $\Sigma_1$ and $\Sigma_2$,
objects $\V_i$ of $\cC(\Sigma_i;\kk)$, and an isomorphism $\V_1|_{\Sigma_1\cap \Sigma_2} \cong \V_2|_{\Sigma_1\cap \Sigma_2}$, they glue to give an object of $\cC(\Sigma_1\cup \Sigma_2)$ that restricts to both $\V_1$ and $\V_2$.  Furthermore, given objects $\V$ and $\W$ of $\cC(\Sigma_1\cup \Sigma_2)$, we have an exact sequence of $\kk$-modules
\begin{multline*}
	0 \longrightarrow\Hom_{\cC(\Sigma_1\cup \Sigma_2;\kk)}(\V,\W) \longrightarrow \Hom_{\cC(\Sigma_1;\kk)}(\V|_{\Sigma_1},\W|_{\Sigma_1})\oplus\Hom_{\cC(\Sigma_2;\kk)}(\V|_{\Sigma_2},\W|_{\Sigma_2})  \\ \longrightarrow  \Hom_{\cC(\Sigma_1\cap \Sigma_2;\kk)}(\V|_{\Sigma_1\cap \Sigma_2},\W|_{\Sigma_1\cap \Sigma_2}).
\end{multline*}
The inclusion $\cL_F = \Sigma_{\ge F}\subseteq \cL$ induces a functor from $\cC(\cL;\kk)$ to $\cC(\cL_F;\kk)$,
which restricts to a functor from $\cC_{\reg}(\cL;\kk)$ to $\cC_{\reg}(\cL_F;\kk)$. 
Under the equivalence given by Proposition \ref{thm:main equivalence}, this corresponds to the functor
\[
 \H(\cL;\kk)\mmod \longrightarrow \H(\cL_F;\kk)\mmod, \qquad
 \M \longmapsto y_F\M[\rk F].
 \]

\begin{remark}\label{lower intervals}
We also have a relation for lower intervals: an object $\V$ of $\cC(\cL^F;\kk)$ can be extended to an object of $\cC(\cL;\kk)$ by letting $\V(G)= 0$ for all $G \not\le F$. This corresponds to pulling back by the quotient homomorphism $\H(\cL;\kk)\to \H(\cL^F;\kk)$ described at the end of Section \ref{sec:pure}.
\end{remark}

\section{Purity, (co)stalks, and duality}\label{sec:pure modules}
In this section, we revisit various notions that we have already discussed for $\H(\cL;\kk)$-modules, such as purity (flabbiness and coflabbiness), 
stalk and costalk functors, and duality.  
For each of these, we study the analogous properties and constructions for objects of the category $\cC(\Sigma;\kk)$.

\subsection{Flabbiness and coflabbiness}
An object $\V = (\{\V(F)\}, \{\varphi_G^F\}, \{\psi_G^F\})$  in $\cC(\Sigma;\kk)$ has a dual nature.  The vector spaces $\V(F)$ together with the maps $\varphi_G^F$ can be viewed as a sheaf on $\Sigma$, viewed as a finite topological space whose open sets are upper sets \cite{Bac}.
Given an upper set $\Sigma'\subseteq \Sigma$, we define the sections on $\Sigma'$ by
\[\Gamma_*(\Sigma', \V) \coloneq \varprojlim_{F \in \Sigma'} \V(F),\]
where the limit is taken using the maps $\varphi_G^F$.  In other words, $\Gamma_*(\Sigma', \V)$ is the submodule of $\prod_{F\in \Sigma'} \V(F)$ given by tuples $(x_F)$ for which $\varphi_G^F(x_G) = x_F$ for all $G < F$ in $\Sigma'$.
There is a restriction map $\Gamma_*(\Sigma', \V) \to \Gamma_*(\Sigma'',\V)$ for any upper sets $\Sigma'' \subseteq \Sigma'\subseteq\Sigma$, and these maps define a sheaf on $\Sigma$.

Dually, we can view $(\{\V(F)\}, \{\psi_G^F\})$ as a cosheaf by defining its cosections on $\Sigma'$ to be
\[\Gamma_!(\Sigma',\V) = \varinjlim_{F\in \Sigma'} \V(F)[-\rk F],\]
with the colimit taken using the maps $\psi_G^F$.  In other words, it is the quotient of $\bigoplus_{F\in \Sigma'} \V(F)[-\rk F]$ by the relations $\psi_G^F(x) - x$ for all $G < F$ in $\Sigma'$ and $x\in \V(F)$.  We have a corestriction map $\Gamma_!(\Sigma'',\V) \to \Gamma_!(\Sigma',\V)$ for any   upper sets $\Sigma''\subseteq \Sigma' \subseteq \Sigma$, and these maps define a cosheaf on $\Sigma$. 

For any element $F \in \Sigma$, the sections and cosections on $\Sigma_{\ge F}$ are particularly simple.  Since 
$\Sigma_{\ge F}$ is the smallest open set containing $F$, we have
\begin{equation}\label{eqn:(co)sections above F}
	\Gamma_*(\Sigma_{\ge F},\V) \cong \V(F) \;\; \mbox{and}\;\;\Gamma_!(\Sigma_{\ge F},\V)\cong \V(F)[-\rk F].
\end{equation}

\begin{definition}
	We say that $\V$ is \textbf{flabby} if the restriction maps $\Gamma_*(\Sigma', \V) \to \Gamma_*(\Sigma'',\V)$ are all surjective.
	We say that $\V$ is \textbf{coflabby} if the corestrictions $\Gamma_!(\Sigma'',\V) \to \Gamma_!(\Sigma',\V)$ are all injective.
\end{definition}
	
\begin{remark}
Based on Definition \ref{pure module def}, the reader might expect us to define $\V$ to be {\bf pure} if it is both flabby and coflabby.
This is correct if $\kk$ is a field, but when $\kk$ is a DVR we will want to impose an additional condition, and similarly for $\H(\cL;\kk)$-modules;
see Definition \ref{new pure}.
\end{remark}

We next show that the explosion functor $\Lambda$ takes flabby (respectively coflabby) $\H(\cL;\kk)$-modules to flabby (respectively coflabby)
objects of $\cC(\cL;\kk)$.  For any $\H(\cL;\kk)$-module $\M$ and any upper set $\Sigma$, we have natural maps
\[\Gamma_!(\Sigma,\Lambda\M) \twoheadrightarrow \Span_\Sigma(\M) \and \M/\Ann_\Sigma(\M) \hookrightarrow \Gamma_*(\Sigma,\Lambda\M)\]
induced by the maps $(\Lambda\M)(F)[-\rk F] = y_F\M \to \Span_\Sigma(\M)$ and $\M/\Ann_\Sigma(\M) \to (\Lambda\M)(F)$ for all $F \in \Sigma$. 
	
	\begin{proposition}\label{prop:criteria for (co)flabbiness}
		Let $\M$ be an $\H(\cL;\kk)$-module.  The following are equivalent:
		\begin{enumerate}
			\item[(a)] $\M$ is coflabby.
			\item[(b)] $\Lambda\M$ is coflabby.
			\item[(c)] The map $\Gamma_!(\Sigma,\Lambda\M) \twoheadrightarrow \Span_\Sigma(\M)$ is an isomorphism for all upper sets $\Sigma\subseteq\cL$.
		\end{enumerate}
		Dually, the following are equivalent: 
		\begin{enumerate}
			\item[(d)] $\M$ is flabby.
			\item[(e)] $\Lambda\M$ is flabby.
			\item[(f)] The map $\M/\Ann_\Sigma(M) \hookrightarrow \Gamma_*(\Sigma,\Lambda\M)$ is an isomorphism for all upper sets $\Sigma\subseteq\cL$.
		\end{enumerate}
	\end{proposition}
	
	\begin{proof}
		We begin by observing that $\Lambda\M$ is coflabby if and only if, for all upper sets $\Sigma\subseteq\cL$, the corestriction map
		\begin{equation*}
			\Gamma_!(\Sigma,\Lambda\M) \longrightarrow \Gamma_!(\cL,\Lambda\M) = \M
		\end{equation*} is injective.  
		The image of this map is $\Span_\Sigma(\M)$, so (b) and (c) are equivalent.
		
		Suppose that (c) holds, and consider two upper sets $\Sigma_1$ and $\Sigma_2$.  The fact that $\Sigma \mapsto \Gamma_!(\Sigma,\Lambda\M)$ is a cosheaf gives an exact sequence
		\[\Gamma_!(\Sigma_1\cap \Sigma_2,\Lambda\M) \longrightarrow \Gamma_!(\Sigma_1,\Lambda\M)\oplus \Gamma_!(\Sigma_2,\Lambda\M) \longrightarrow \Gamma_!(\Sigma_1\cup \Sigma_2,\Lambda\M).\]
		Applying (c) we obtain an exact sequence
		\[\Span_{\Sigma_1\cap \Sigma_2}(\M) \longrightarrow \Span_{\Sigma_1}(\M) \oplus \Span_{\Sigma_2}(\M) \longrightarrow \Span_{\Sigma_1\cup \Sigma_2}(\M),\]
		where the maps are induced by the natural inclusions, up to a sign.  It follows that we have
		$\Span_{\Sigma_1}(\M) \cap \Span_{\Sigma_2}(\M) = \Span_{\Sigma_1\cap \Sigma_2}(\M)$, so (a) holds.
		
		Conversely, suppose that (a) holds.  We show that $\Gamma_!(\Sigma,\Lambda\M) \to \Span_\Sigma\M$ is an isomorphism by induction on $|\Sigma|$.  If $\Sigma$ is empty, then both sides vanish.  If $|\Sigma|\ge 1$, there are two cases.  If $\Sigma = \Sigma_{\ge F}$ for some $F$, then both sides 
		are equal to $y_F\M$.  Otherwise, choose a 
		minimal element $F$ and let $\Sigma' = \Sigma\setminus \{F\}$, so that $\Sigma = \Sigma' \cup \Sigma_{\ge F}$ and $\Sigma' \cap \Sigma_{\ge F} = \Sigma_{>F}$.  Consider the following diagram:
		\[\begin{tikzcd}
			0 \arrow[r] & \Gamma_!(\Sigma_{>F}, \Lambda\M) \arrow[r] \arrow[d] & \Gamma_!(\Sigma',\Lambda\M) \oplus \Gamma_!(\Sigma_{\ge F},\Lambda\M) \arrow[r]\arrow[d] & \Gamma_!(\Sigma,\Lambda\M) \arrow[r]\arrow[d] & 0 \\
			0 \arrow[r] & \Span_{\Sigma_{>F}}(\M) \arrow[r] & \Span_{\Sigma'}(\M)\oplus \Span_{\Sigma_{\ge F}}(\M) \arrow[r] & \Span_{\Sigma}(\M) \arrow[r] & 0 
		\end{tikzcd}\] 
		Since $\Sigma$ has more than one minimal element, both $\Sigma'$ and $\Sigma_{\ge F}$ have fewer elements than $\Sigma$.  The induction hypothesis implies that the first and second vertical maps are isomorphisms, so the result will follow by the five-lemma if we can show that both rows are exact.
		
		The bottom row is exact on the left because the components of the map are inclusions, it is exact on the right since $\Span_\Sigma(\M) = \Span_{\Sigma'}(\M) + \Span_{\Sigma_{\ge F}}(\M)$, and it is exact in the middle by applying property (a) with $\Sigma_1 = \Sigma'$ and $\Sigma_2 = \Sigma_{\ge F}$.  For the top row, exactness on the left follows from the injectivity of the left vertical map and exactness of the bottom row.  Exactness in the other two places follows from the fact that $\Sigma\mapsto \Gamma_!(\Sigma, \Lambda\M)$ is a cosheaf.  
		
		This completes the proof that (a), (b), (c) are equivalent.
		The proof that (d), (e), (f) are equivalent is analogous, with $\Gamma_!$ replaced by $\Gamma_*$, $\Span_\Sigma\M$ replaced by $\M/\Ann_\Sigma \M$, and all arrows reversed.	
	\end{proof}
	
	\begin{example}	
		The regular $\H(\cL;\kk)$-module $\H(\cL;\kk)$ is coflabby because $\Span_\Sigma \H(\cL;\kk)$ has a basis $\{y_F \mid F \in \Sigma\}$.
		On the other hand, $\H(\cL;\kk)$ is  in general not flabby.    For example, consider
		 the lattice of flats $\cL$ of the rank $3$ uniform matroid on $E = \{1,2,3,4\}$, so that the elements of $\cL$ are $E$
		itself along with all subsets $F \subseteq E$ with $|F| \le 2$.  
		Then $y_3-y_4\in\Ann_{\Sigma_{\ge 12}} \H(\cL;\kk)$, since $y_3y_{12} = y_E = y_4y_{12}$, but $y_3-y_4\notin\Ann_{\Sigma_{\ge 1}} \H(\cL;\kk) + \Ann_{\Sigma_{\ge 2}} \H(\cL;\kk)$.
		Since $\cL$ is not modular, this also follows from Theorem \ref{E=H}.  
	\end{example}
	
	\begin{example}		
		Let $\cL$ be the Boolean lattice on the ground set $\{1,2\}$, so $\H(\cL;\kk)$ is isomorphic to $\kk[y_1,y_2]/\langle y_1^2,y_2^2\rangle$,
		with $y_{12} = y_1y_2$.  Then the maximal ideal $\mathfrak{m}$ is not coflabby: we have $y_1\mathfrak{m} \cap y_2\mathfrak{m} = \kk y_{12}$, but $y_{12}\mathfrak{m} = 0$.  On the other hand, since $\Ann_{\Sigma_{\geq 1}}\mathfrak{m}+\Ann_{\Sigma_{\geq 2}}\mathfrak{m}=\mathfrak{m}=\Ann_{\Sigma_{\geq 1}\cap \Sigma_{\geq 2}}\mathfrak{m}$ and $(\Sigma_{\geq 1}, \Sigma_{\geq 2})$ is the only pair of incomparable upper sets of $\cL$, it follows that $\mathfrak{m}$ is flabby. 
	\end{example}
	
	\subsection{Stalks and costalks}
	We have already discussed the stalk and costalk functors 
	\[
	(\ )_F, (\ )_{[F]}\colon \H(\cL;\kk)\mmod \longrightarrow \kk\mmod.
	\]
	We define analogous functors on the category $\cC(\Sigma;\kk)$ as follows.  For any $F \in \Sigma$, and any object $\V = (\{\V(F)\}, \{\varphi_G^F\},\{\psi_G^F\})$ of $\cC(\Sigma;\kk)$, we define the {\bf stalk}
	\[
	\V_F \coloneq \coker \left(\sum_{G>F} \psi_F^G \colon \bigoplus_{G>F} \V(G)[\rk F - \rk G] \longrightarrow \V(F)\right)
	\]
	and the {\bf costalk}
	\[
	\V_{[F]} \coloneq \ker \left(\sum_{G>F} \varphi_F^G\colon \V(F) \longrightarrow \bigoplus_{G>F} \V(G)\right).
	\]
	These functors can also be described in terms of sections and cosections.  Using 
	the isomorphisms in Equation \eqref{eqn:(co)sections above F}, the restriction and corestriction maps for the inclusion $\Sigma_{>F} \subseteq \Sigma_{\ge F}$ give maps
	\[
\Gamma_!\bigl(\Sigma_{>F},\V\bigr)[\rk F]
\xrightarrow{\ \Psi_F\ }
\V(F)
\xrightarrow{\ \Phi_F\ }
\Gamma_*\bigl(\Sigma_{>F},\V\bigr).
\]
	Then, we have $\V_F = \coker \Psi_F$ and $\V_{[F]} = \ker \Phi_F$.  
		The following lemma is an immediate consequence of the definitions.
	
	\begin{lemma}\label{lem:sheaf and module stalks}
		For an $\H(\cL;\kk)$-module $\M$, we have natural isomorphisms
		\[(\Lambda\M)_F \cong \M_F\and(\Lambda\M)_{[F]} \cong \M_{[F]}.\]
	\end{lemma}
	
\begin{remark}
Lemma \ref{lem:sheaf and module stalks} illustrates why we have chosen to use the words
``stalk'' and ``costalk'' for $\V_F$ and $\V_{[F]}$.  However, the reader is warned that $\V_F$
is {\em not} isomorphic to the stalk of $\V$ in the sense of sheaf theory.  If we regard $\V$ as a sheaf on 
the finite topological space $\Sigma$, then $\Sigma_{\geq F}$ is the minimal open
set containing $F$, so the stalk of $\V$ at $F$ is equal to $\Gamma_*(\Sigma_{\geq F},\V) \cong \V(F)$, which is typically
not isomorphic to $\V_F$. 
\end{remark}

The following lemma is tautological when $\Sigma' = \Sigma_{\ge F}$, but we will need it for arbitrary $\Sigma'$.
	
\begin{lemma}\label{lem:boundaries and (co)stalks}
Suppose we have upper sets $\Sigma' \subseteq \Sigma$, and let $\Sigma'' \coloneq \Sigma' \setminus \{F\}$, where $F$ is a minimal element in $\Sigma'$.
Let $\V$ be an object of $\cC(\Sigma;\kk)$.
		\begin{enumerate}[(a)]
			\item There is an exact sequence
			\[
			0 \longrightarrow \V_{[F]} \longrightarrow  \Gamma_*(\Sigma',\V) \longrightarrow \Gamma_*(\Sigma'',\V),
			\]
			and the last map is surjective if $\V$ is flabby. 
			\item 	There is an exact sequence
			\[
			\Gamma_!(\Sigma'',\V)[\rk F] \longrightarrow  \Gamma_!(\Sigma',\V)[\rk F] \longrightarrow \V_F \longrightarrow 0,
			\]
			and the first map is injective if $\V$ is coflabby.
		\end{enumerate}
	\end{lemma}	
	\begin{proof}
		The fact that $\V$ is a sheaf on $\Sigma$ implies that we have an exact sequence
		\[
		0 \longrightarrow \Gamma_*(\Sigma',\V) \longrightarrow \Gamma_*(\Sigma'',\V) \oplus \Gamma_*(\Sigma_{\ge F},\V) \longrightarrow \Gamma_*(\Sigma_{>F},\V).
		\]  
		Thus, the kernel of the restriction map
		$\Gamma_*(\Sigma',\V) \to \Gamma_*(\Sigma'',\V)$ is isomorphic to the kernel of the map $\Phi_F\colon \V(F) = \Gamma_*(\Sigma_{\ge F},\V) \to \Gamma_*(\Sigma_{>F},\V)$.  This gives the first exact sequence, and flabbiness of $\V$ implies the surjectivity of the last map.
		The proof of (b) is analogous.
	\end{proof}
	
	\begin{lemma}\label{lem:(co)flabby extensions}  
		An object $\cV$ of $\cC(\Sigma;\kk)$ is flabby if and only if $\Phi_F\colon \V(F) \to \Gamma_*\bigl(\Sigma_{>F},\V\bigr)$ is surjective for all $F\in \Sigma$, and coflabby if and only if $\Psi_F\colon \Gamma_!\bigl(\Sigma_{>F},\V\bigr)[\rk F] \to\V(F)$ is injective for all $F \in \Sigma$.
		
	\end{lemma}
	
\begin{proof}
We prove the first statement; the proof of the second statement is similar.  The map $\Phi_F$ is by definition the restriction of 
sections from $\Sigma_{\ge F}$ to $\Sigma_{>F}$, so it is surjective when $\V$ is flabby.  Now assume that $\Phi_F$ is surjective for all $F\in\Sigma$,
and let $\Sigma''\subseteq\Sigma'\subseteq\Sigma$ be given.  We wish to show that  
		the restriction $\Gamma_*(\Sigma',\cV) \to \Gamma_*(\Sigma'',\cV)$ is a surjection.  
		
		By induction on $|\Sigma' \setminus \Sigma''|$, it suffices to prove surjectivity when $\Sigma' = \Sigma'' \cup \{F\}$
		for some $F\in\Sigma$. 
		Given a section $s \in \Gamma_*(\Sigma'',\V)$, the restriction $s|_{\Sigma_{>F}}$ extends to a section $t\in \Gamma_*(\Sigma_{\ge F},\V) = \V(F)$, since $\Phi_F$ is surjective.  The sections $s$ and $t$ agree on $\Sigma''\cap \Sigma_{\ge F} = \Sigma_{>F}$, so they glue to give a section on  $\Sigma'$ whose restriction to $\Sigma''$ is $s$.	
	\end{proof}

\subsection{Freeness and Duality}\label{sec:freeness}
When $\kk$ is a field, the endofunctor of $\H(\cL;\kk)\mmod$ that sends $\M$ to $\M^* \coloneq \Hom_\kk(\M,\kk)$ is exact, and the natural map $\M \to (\M^*)^*$ is an isomorphism.  (Recall that all of our modules are finitely generated, and that any finitely generated $\H(\cL;\kk)$-module is 
also finitely generated as a $\kk$-module.)  
	
When $\kk$ is a DVR, however, we only have $V \cong (V^*)^*$ when $V$ is a free $\kk$-module, and the dual of a \emph{split} exact sequence is exact, but the duals of exact sequences are only left exact in general.  Because of this, we need to add an additional freeness hypothesis when $\kk$ is a DVR.  
The case of a DVR is needed for the proof of Theorem \ref{it goes up}, which appears in Section \ref{sec:base change}.
For all other results in the paper, nothing is lost by assuming that $\kk$ is a field.
	
	\begin{definition}\label{new pure}
		An $\H(\cL;\kk)$-module $\M$ is called \textbf{locally free} if the stalk $\M_F$ is a free $\kk$-module for all $F \in \cL$.
		An object $\cV$ of $\cC(\Sigma;\kk)$ is called {\bf locally free} if the stalk $\cV_F$ is a free $\kk$-module for all $F\in \Sigma$.
		An object of either $\H(\cL;\kk)\mmod$ or $\cC(\Sigma;\kk)$ is called {\bf pure} if it is flabby, coflabby, and locally free.   
	\end{definition}
\begin{remark}\label{purepure}
Lemma \ref{lem:sheaf and module stalks} and Proposition \ref{prop:criteria for (co)flabbiness}
	imply that $\M$ is locally free if and only if $\Lambda\M$ is locally free, and $\M$ is pure if and only if $\Lambda\M$ is pure.
	\end{remark}
	
\begin{lemma}\label{pure implies regular}
Let $\V$ be an object of $\cC(\Sigma;\kk)$.  If $\V$ is pure, then it is regular.
\end{lemma}

\begin{proof}
We have $\Gamma_*(\Sigma_{\ge F},\V) \cong \V(F)$, and for any $G \leq F$, the restriction map 
from $\Gamma_*(\Sigma_{\ge G}, \V)$ to $\Gamma_*(\Sigma_{\ge F}, \V)$ is given by $\varphi_G^F$.  Thus flabbiness implies these maps are surjective.  A dual argument shows that coflabbiness implies the maps $\psi_G^F$ are injective.
	\end{proof}

	\begin{proposition}\label{prop:(local) freeness}
		If an object $\cV$ of $\cC(\Sigma;\kk)$ is coflabby and locally free, the following hold:
		\begin{enumerate}[(a)]
			\item For any upper set $\Sigma'\subseteq \Sigma$, the space $\Gamma_!(\Sigma',\V)$ of cosections is a free $\kk$-module.
			\item The corestriction homomorphisms $\Gamma_!(\Sigma'',\V) \to \Gamma_!(\Sigma',\V)$ are split injections.
			\item Every $\V(F)$ is a free $\kk$-module, and each map $\psi_G^F$ is a split injection.
			\item For any upper set $\Sigma'\subseteq \Sigma$, the space $\Gamma_*(\Sigma',\V)$ of sections is a free $\kk$-module.
		\end{enumerate}
	\end{proposition}
	\begin{proof}
		Consider upper sets $\Sigma'' \subseteq \Sigma' \subseteq \Sigma$, and suppose that $\Sigma'' = \Sigma' \setminus \{F\}$ for a minimal element $F \in \Sigma'$.  Lemma \ref{lem:boundaries and (co)stalks} gives an exact sequence of $\kk$-modules
		\[
		0 \longrightarrow \Gamma_!(\Sigma'',\V) \longrightarrow \Gamma_!(\Sigma',\V)\longrightarrow \V_F[-\rk F] \longrightarrow 0.
		\]
		Since $\V_F$ is free, the sequence splits, and if $\Gamma_!(\Sigma'',\V)$ is free, then so is
		$\Gamma_!(\Sigma',\V)$.  Induction on the quantity $|\Sigma'|$ shows that $\Gamma_!(\Sigma',\V)$ is free, and induction on 
		the quantity $|\Sigma'\setminus \Sigma''|$ shows that the corestriction map $\Gamma_!(\Sigma'',\V) \to \Gamma_!(\Sigma',\V)$ is a composition of split injections, hence also a split injection.  The statement (c) is a special case of (a) and (b), since $\cV(F) = \Gamma_!(\Sigma_{\ge F},\V)[\rk F]$ and the maps $\psi_G^F$ are corestrictions.
		
		To prove (d), we note that $\Gamma_*(\Sigma',\V)$ is a $\kk$-submodule of the free module $\bigoplus_{F \in \Sigma'} \V(F)$, hence it is torsion-free, and a finitely generated torsion-free module over a DVR is free. 
	\end{proof}
	
We next define the duality functor on $\cC(\Sigma;\kk)$.
For an object $\V = (\{\V(F)\}, \{\varphi_G^F\}, \{\psi_G^F\})$, let
	\[\W(F) = (\V(F))^*[\rk F].\]
	For any $G \le F\in\Sigma$, let $\bar\varphi_G^F\colon \W(G) \to \W(F)$ be the dual of the $\kk$-module homomorphism
	\[
	\psi_G^F[-\rk F]\colon \V(F)[-\rk F] \longrightarrow \V(G)[-\rk G].
	\]
	Similarly, let $\bar\psi_G^F\colon \W(F)\to \W(G)[\rk F - \rk G]$ be the dual of the homomorphism
	\[
	\varphi_G^F[-\rk F]\colon\V(G)[-\rk F] \longrightarrow \V(F)[-\rk F].
	\]
	  Then we define $\V^*$ to be the triple $(\{\W(F)\}, \{\bar\varphi_G^F\}, \{\bar\psi_G^F\})$.
	
\begin{lemma}\label{lem:sheafification and duality}
Let $\M$ be an $\H(\cL;\kk)$-module.  If $\kk$ is a field or if $\M$ is coflabby and locally free, then there is a natural isomorphism $\Lambda(\M^*) \cong (\Lambda \M)^*$.
\end{lemma}

\begin{proof}
By Proposition \ref{prop:(local) freeness}, $\M$ and $\M(F)$ are free $\kk$-modules, and $\psi^F=\psi^F_\hz \colon \M(F) \to \M[\rk F]$ is a split injection.  Dualizing and shifting gives a factorization of multiplication by $y_F$ on $\M^*$:
		\[
		\xymatrix{
			\M^*  \ar[rr]^{y_F} \ar[dr]_{(\psi^F)^*}&& \M^*[\rk F] 
			\\
			& \M(F)^*[\rk F] \ar[ur]_{(\varphi^F)^*}&
		}
		\]
Since $(\psi^F)^*$ is a surjection and $(\varphi^F)^*$ is an injection, we have an isomorphism $\M^*(F) \cong \M(F)^*[\rk F]$.  Dualizing the maps $\varphi_G^F$ and $\psi_G^F$ is handled similarly.
\end{proof}
	
\begin{lemma}
Let $\Sigma\subseteq \cL$ be an upper set and $\cV$ an object of $\cC_\reg(\Sigma)$.
If $\kk$ is a field or if $\cV$ is coflabby and locally free, then $\V^*$ is regular.
\end{lemma}
	
\begin{proof}
Since coflabbiness, local freeness, and regularity are all preserved under restriction from $\cC(\Sigma;\kk)$ to $\cC(\Sigma';\kk)$, it is enough to prove the assertion when $\Sigma=\Sigma_{\ge F}$ is a principal open set.  But $\Sigma_{\ge F} = \cL_F$ is a geometric lattice, so it is enough to prove the statement when $\Sigma = \cL$.  Let $\M \coloneq \Theta\V$; then Proposition \ref{thm:main equivalence} tells us that $\cV \cong \Lambda\M$.  
Since $\kk$ is a field or $\M$ is coflabby and locally free, Lemma \ref{lem:sheafification and duality} says that
		\[\cV^* \cong (\Lambda\M)^* \cong \Lambda(\M^*),\]
which is regular by Proposition \ref{thm:main equivalence}.
	\end{proof}
	
	When $\kk$ is a field, the duality functor on $\cC(\Sigma;\kk)$ interchanges the section and cosection functors.  This also holds when $\kk$ is a DVR, but additional hypotheses are needed.
	\begin{proposition}\label{prop:(co)sections and duality}  	
		Let $\V$ be an object of $\cC(\Sigma;\kk)$.
		\begin{enumerate}[(a)]
			\item There is a natural isomorphism $\Gamma_*(\Sigma,\V^*) \cong \Gamma_!(\Sigma,\V)^*$.
			For any $\Sigma'\subseteq\Sigma$,
			this isomorphism intertwines restriction to $\Sigma'$ for $\V^*$ with the dual of corestriction for $\V$.
			\item If $\kk$ is a field or $\V$ is pure, then there is a natural isomorphism
			$\Gamma_!(\Sigma,\V^*) \cong \Gamma_*(\Sigma,\V)^*$.
			For any $\Sigma'\subseteq \Sigma$, this isomorphism intertwines corestriction for $\V^*$ with the dual of restriction for $\V$.
		\end{enumerate}
	\end{proposition}
	
	\begin{proof}
		By the definition of the cosection functor, we have an exact sequence
		\[
		\bigoplus_{G \ge F \in \Sigma} \V(G)[-\rk G] \longrightarrow \bigoplus_{F \in \Sigma} \V(F)[-\rk F] \longrightarrow \Gamma_!(\Sigma,\V) \to 0.
		\]
		Dualizing gives an exact sequence
		\[
		0 \to \Gamma_!(\Sigma,\V)^* \longrightarrow \bigoplus_{F \in \Sigma} \V^*(F) \longrightarrow \bigoplus_{G \ge F \in \Sigma} \V^*(G).
		\]
		This is the sequence that defines the sections $\Gamma_*(\Sigma,\V^*)$, which proves (a).  
		
		When $\kk$ is a field, dualizing is an exact functor, and (b) can be proved in the same way.  When $\kk$ is a DVR and $\V$ is pure, 
		we proceed by induction on
		$|\Sigma|$.  If $\Sigma = \Sigma_{\ge F}$, the result follows by the identification of sections of $\V$ with $\V(F)$ and cosections of $\V^*$ with $\V^*(F)[-\rk F] = \V(F)^*$.  In particular, the result holds when $|\Sigma| = 1$.  When $|\Sigma|>1$, choose a minimal element $F\in \Sigma$ and let $\Sigma' = \Sigma \setminus \{F\}$.  Consider the sequence
		\[
		0 \longrightarrow \Gamma_*(\Sigma,\V) \longrightarrow \Gamma_*(\Sigma',\V) \oplus \V(F) \longrightarrow 
		\Gamma_*(\Sigma_{>F},\V) \longrightarrow 0.
		\]
		This sequence is exact on the left and in the middle by the sheaf property of $\Gamma_*$, and it is exact on the right because $\V$ is flabby.  Furthermore, $\Gamma_*(\Sigma_{>F},\V)$ is free by Proposition \ref{prop:(local) freeness}.  Thus the sequence splits, so dualizing it and using the inductive hypothesis gives an exact sequence
		\[
		0 \longrightarrow \Gamma_!(\Sigma_{>F},\V^*) \longrightarrow \Gamma_!(\Sigma',\V^*) \oplus \V^*(F)[-\rk F] \longrightarrow \Gamma_*(\Sigma,\V)^* \to 0,
		\]
		where the first map is induced by the corestrictions.  The isomorphism $\Gamma_!(\Sigma,\V^*) \cong \Gamma_*(\Sigma,\V)^*$ follows.
	\end{proof}
	
\begin{example}
This example illustrates the need for the purity hypothesis in Proposition \ref{prop:(co)sections and duality}(b).
Let $\kk=\Z_2$ be the ring of $2$-adic integers, and let $\cL$ be a Boolean lattice of rank $2$.  Let $\M = \H(\cL;\kk)$, 
endowed not with the structure of the regular module, but rather with the
module structure given by the map $\H(\cL;\kk)\to \H(\cL;\kk)$ that sends $y_i$ to $2y_i$.  Let $\cW \coloneq \Lambda\M$, 
which is flabby but not coflabby and not locally free.  
Let $\V \coloneq \cW^*$, and note that $\V^* \cong \cW$.  Thus $\Gamma_!(\cL\setminus \{\hz\},\V^*) = \Gamma_!(\cL\setminus \{\hz\},\W)$ has torsion, and therefore 
cannot be isomorphic to the dual of a $\kk$-module.
\end{example}
	
	\begin{proposition}\label{prop:duality preserves purity}
		Suppose that $\V$ is a pure object of $\cC(\Sigma;\kk)$.  Then the dual $\V^*$ is pure.
		For any $F \in \Sigma$, we have natural isomorphisms
		\[(\V^*)_F \cong (\V_{[F]})^*[\rk F]\and (\V^*)_{[F]} \cong (\V_F)^*[\rk F].\]
	\end{proposition}
	\begin{proof}
		By Proposition \ref{prop:(local) freeness}, the corestriction homomorphisms of $\V$ are split injections.  Proposition \ref{prop:(co)sections and duality}(a) then implies that the restriction homomorphisms for $\V^*$ are surjections, which means that $\V^*$ is flabby.  Similarly, the flabbiness of $\V$ together with Proposition \ref{prop:(co)sections and duality}(b) shows that $\V^*$ is coflabby.
		
		Lemma \ref{lem:boundaries and (co)stalks}(a) gives an exact sequence
		\[
		0 \to \V_{[F]} \longrightarrow \Gamma_*(\Sigma_{\ge F},\V) \longrightarrow \Gamma_*(\Sigma_{>F},\V) \longrightarrow 0,
		\]
		which is split since Proposition \ref{prop:(local) freeness} implies that $\Gamma_*(\Sigma_{>F},\V)$ is free.  Dualizing this sequence and applying Proposition \ref{prop:(co)sections and duality} gives the exact sequence
		\[
		0 \longrightarrow \Gamma_!(\Sigma_{>F},\V^*) \longrightarrow \Gamma_!(\Sigma_{\ge F},\V^*) \longrightarrow (\V_{[F]})^* \to 0.
		\]
		Then Lemma \ref{lem:boundaries and (co)stalks}(b) gives the isomorphism $(\V^*)_F \cong (\V_{[F]})^*[\rk F]$.  Since duals are always free, this shows that $\V^*$ is locally free.  The isomorphism $(\V^*)_{[F]} \cong (\V_F)^*[\rk F]$ is proved similarly.	
	\end{proof}

	\section{Construction and decomposition of pure objects}\label{sec:existence}
	In this section we construct and classify indecomposable pure objects in $\cC_\reg(\Sigma;\kk)$.  Specializing our results to $\Sigma=\cL$ 
	will allow us to prove Theorems \ref{def of E}, \ref{indecomposables}, and \ref{general decomposition} from the introduction.
	
	\subsection{Extending objects}\label{sec:extending objects}	
	We start by describing an inductive schema for building objects in $\cC(\Sigma;\kk)$.  	
	
	Consider an upper set $\Sigma \subseteq \cL$, and suppose that $\Sigma_{>F}\subseteq \Sigma$. 
	Let $\cV$ be an object of $\cC(\Sigma;\kk)$.  If $F \in \Sigma$, we have the maps
	\[
\Gamma_!\bigl(\Sigma_{>F},\V\bigr)[\rk F]
\xrightarrow{\ \Psi_F\ }
\V(F)
\xrightarrow{\ \Phi_F\ }
\Gamma_*\bigl(\Sigma_{>F},\V\bigr).
\]
We will see in Lemma \ref{composition} that the composition $\Phi_F\Psi_F$ depends only on the restriction of $\V$ to $\Sigma_{>F}$.	

	Consider the system of maps $\Xi_F(G,H)\colon \V(G)[\rk F-\rk G] \to \V(H)$ for all elements $G, H \in\Sigma_{>F}$ defined 
	by the formula
	\[\Xi_F(G,H) = \begin{cases}
		\psi_{H}^{G\vee H}\varphi_{G}^{G\vee H} & \mbox{if $F = G \wedge H$ and $G$ and $H$ form a modular pair,} \\
		0 & \mbox{otherwise}.
	\end{cases}\]
	If $F\in \Sigma$, then Proposition \ref{prop:psi phi relations}(c) tells us that $\Xi_F(G,H) = \varphi^H_F\psi^G_F$, but the map is defined even if $F \notin \Sigma$.

	\begin{lemma}\label{composition}
		The maps $\Xi_F(G,H)$ 
		induce a map $\Xi_F\colon\Gamma_!(\Sigma_{>F},\V)[\rk F] \to \Gamma_*(\Sigma_{>F},\V)$.
		If $F \in \Sigma$, then $\Xi_F = \Phi_F\Psi_F$.
	\end{lemma}
	\begin{proof}
		To see that the maps $\Xi_F(G,H)$ pass to the limit, we need to show that 
		\begin{equation}\label{eqn:Xi limits}
		    \Xi_F(G,H)\psi_G^{G'} = \Xi_F(G',H)  \ \ \text{and} \ \  \varphi_G^{G'}\Xi_F(H,G) = \Xi_F(H,G')	
		\end{equation}
		for any $G, G', H \in \Sigma_{>F}$ with $G' > G$.  We have $G' \vee (G\vee H) = G' \vee H$, and 
\begin{multline}\label{eqn:rank equation}
\rk G' + \rk H - \rk(G' \vee H) - \rk F   \\
=(\rk G' + \rk(G\vee H) - \rk(G' \vee H) - \rk G) + (\rk G + \rk H - \rk(G\vee H) -  \rk F).
\end{multline}
 If the left side of \eqref{eqn:rank equation} vanishes, then $(G',H)$ is a modular pair and $F = G' \wedge H$.  Since both summands on the right side of \eqref{eqn:rank equation} are nonnegative, they must also vanish, and so $(G', G\vee H)$ and $(G, H)$ are modular pairs and $G = G'\wedge (G\vee H)$, $F = G\wedge H$.  Then we have 
		\begin{align*}
			\Xi_F(G,H)\psi_G^{G'} & = \psi_H^{G\vee H}\varphi_G^{G\vee H}\psi_G^{G'}\\
			& = (H \swa G\vee H\nwa G \swa G')\\ & = (H \swa G \vee H \swa G' \vee H\nwa G') \\
			& = (H \swa G' \vee H\nwa G') =   
			\psi_H^{G'\vee H}\varphi_{G'}^{G'\vee H} \\ & = \Xi_F(G', H).	
		\end{align*}
		On the other hand, if the left side of \eqref{eqn:rank equation} is nonzero, then $\Xi_F(G',H) = 0$.  At least one of the summands on the right side must be positive.  If the first one is nonzero, then 
\[
\varphi^{G\vee H}_G\psi^{G'}_G = \Xi_G(G',G\vee H) = 0,
\]
while if the second is nonzero, then $\Xi_F(G,H)=0$.  In either case we get $\Xi_F(G,H)\psi^{G'}_G=0$, proving the first part of \eqref{eqn:Xi limits}.  The proof of the second part is similar.
		
		Finally, if $F \in \Sigma$, we have $\Xi_F(G,H)= \varphi_F^H\psi_F^G$ for all $G, H > F$, and it follows immediately that $\Xi_F = \Phi_F\Psi_F$. 
	\end{proof}
	
	Consider a pair of upper sets $\Sigma'\subseteq \Sigma$ with the property that $F$ is a minimal element in $\Sigma$ and $\Sigma' = \Sigma \setminus \{F\}$.  We wish to understand how to extend an object $\V'$ of $\cC(\Sigma';\kk)$ to $\Sigma$.  From the preceding discussion, giving an object $\V$ of $\cC(\Sigma;\kk)$ and an isomorphism $\V|_{\Sigma'} \cong \V'$ is equivalent to giving a graded $\kk$-module $\V(F)$ and maps $\Psi_F$, $\Phi_F$
	making the triangle
	\begin{equation}\label{eqn:extension triangle}
		\begin{tikzcd}
			\Gamma_!(\Sigma_{>F},\V')[\rk F] \arrow[rr, "\Xi_F"] \arrow[rd, "\Psi_F"'] &                           & \Gamma_*(\Sigma_{>F},\V') \\
			& \V(F) \arrow[ru, "\Phi_F"'] &                
		\end{tikzcd}
	\end{equation}
	commute. Indeed, given such a diagram we let $\V(G) = \V'(G)$ for $G \ne F$, and for $G > F$, we let $\psi_F^G$ be the composition 
	\[
\V(G)
=
\Gamma_!\bigl(\Sigma_{\geq G},\V'\bigr)[\rk G]
\longrightarrow
\Gamma_!\bigl(\Sigma_{>F},\V'\bigr)[\rk G]
\xrightarrow{\ \Psi_F\ }
\V(F)[\rk G-\rk F].
\]
Similarly, let $\varphi_F^G$ be the composition 
	of $\Phi_F$ with the restriction $\Gamma_*(\Sigma_{>F},\V') \to \Gamma_*(\Sigma_{\ge G},\V') = \V(G)$.  This defines an object of $\cC(\Sigma;\kk)$ with the required properties.

	\begin{lemma}
		If $\cV$ is the extension of $\cV'$ determined by the diagram \eqref{eqn:extension triangle}, then 
		\begin{enumerate}[(a)]
			\item $\cV$ is flabby if and only if $\cV'$ is flabby and $\Phi_F$ is surjective.
			\item $\cV$ is coflabby if and only if $\cV'$ is coflabby and $\Psi_F$ is injective.
			\item $\cV$ is coflabby and locally free if and only if $\cV'$ is coflabby and locally free, $\cV(F)$ is a free $\kk$-module, and $\Psi_F$ is a split injection.
			\item $\cV$ is pure if and only if $\cV'$ is pure, $\cV(F)$ is a free $\kk$-module, $\Phi_F$ is surjective, and $\Psi_F$ is a split injection.  
		\end{enumerate}
	\end{lemma}
	\begin{proof}
		Statements (a) and (b) follow immediately from Lemma \ref{lem:(co)flabby extensions}.  The statement (d) follows by combining statements (a) and (c), so it is enough to prove (c).
		
		If $\cV$ is coflabby and locally free, then its restriction $\cV'$ to $\Sigma'$ is again coflabby and locally free, and Proposition \ref{prop:(local) freeness} parts (b) and (c) show that $\cV(F)$ is locally free and $\Psi_F$ is a split injection.
		
	    Conversely, suppose that $\cV'$ is coflabby and locally free, $\cV(F)$ is free, and $\Psi_F$ is a split injection.  Then 
	    Lemma \ref{lem:(co)flabby extensions} shows that $\cV$ is coflabby, 
	    so we just need to check that $\cV$ is locally free.  For $G \in\Sigma'$, 
	    the stalk $\cV_G = \cV'_G$ is a free $\kk$-module because $\cV'$ is locally free, and since $\Psi_F$ is a split injection, its cokernel $\cV_F$ is isomorphic to a direct summand of the free $\kk$-module $\cV(F)$, and is therefore free. This completes the proof.
	\end{proof}
	
	\begin{example}\label{ex:extension example}
		Let $\cL$ be the lattice of flats of the rank $2$ uniform matroid on $E=\{1,2,3\}$, so $\cL=\{\hz,\{1\},\{2\},\{3\},E\}$.  Let $\kk = \Z_2$, the $2$-adic integers.  If we let $\cV = \Lambda \H(\cL;\kk)$, then the restriction $\cV|_{\Sigma_{>\hat{0}}}$ is pure, but $\cV$ itself is not pure.  Indeed, $\H(\cL;\kk)$ is coflabby, and $\Gamma_!(\Sigma_{>\hz},\V)$ is isomorphic to the graded maximal ideal $\mm \subseteq \H(\cL;\kk)$.  In degree 
		$1$, the diagram \eqref{eqn:extension triangle} becomes
		\begin{equation*}
			\begin{tikzcd}
				\H^1(\cL;\kk)  \arrow[rr, "\Xi_{\hz}"] \arrow[rd, 	"\Psi_{\hz} = \id"'] &                           & \Gamma_*\left(\Sigma_{>\hz},\V\right)^1 = \displaystyle\bigoplus_{F\in \cL^1} \H^2(\cL;\kk) \cong \kk^3,\\
					& \H^1(\cL;\kk) \arrow[ru, "\Phi_{\hz}"',end anchor={[xshift=-64pt,yshift=-11pt]center}] &                
			\end{tikzcd}
		\end{equation*}
		where the component of $\Xi_{\hz}$ corresponding to $F\in \cL^1$ is given by multiplication by $y_F$.  So with the usual bases of the left and right sides the matrix representing $\Xi_{\hz} = \Phi_{\hz}$ is
		\[\begin{pmatrix}
			0 & 1 & 1 \\
			1 & 0 & 1 \\
			1 & 1 & 0
		\end{pmatrix},\]
		which has determinant $2$.  This map is not surjective, thus $\cV$ is not flabby.  
	\end{example}
	
	\begin{example}	
Let $\cL$ and $\kk$ be as in Example \ref{ex:extension example}, but this time consider the object		
$$\cW := \cV^* = \Lambda(\H(\cL;\kk)^*).$$  For the purposes of this example, we will decorate our maps with a subscript of $\cW$
to distinguish them from the maps in Example \ref{ex:extension example}.
The restrictions of $\cW$ and $\cV$ to $\Sigma_{>\hz}$ are isomorphic up to a shift, but     	
		the map $\Phi_{\hz,\cW}$ is the identity, thus $\cW$ is flabby, while $\Psi_{\hz,\cW} = \Xi_{\hz,\cW}$ is given by the same matrix above.  So $\Psi_{\hz,\cW}$ is injective, thus $\cW$ is coflabby, but its cokernel has torsion, thus $\cW$ is not locally free.  On the other hand, 
		if we took $\kk$ to be a field of characteristic $2$, then $\Psi_{\hz,\cW}$ would not be injective, and $\cW$ would not be coflabby.  
	\end{example}
	
	\subsection{Extending morphisms}\label{sec:extending morphisms}
	We now study how morphisms in our category extend when a new element is added.	
	We continue with the notation of the previous section: $\Sigma$ and $\Sigma'$ are upper sets with $\Sigma = \Sigma'\cup \{F\}$.  Given objects $\V, \V'$ of $\cC(\Sigma;\kk)$, a homomorphism $f\colon \V \to \V'$ is given uniquely by its restrictions $f|_{\Sigma'}$ and $f|_{\Sigma_{\ge F}}$, subject to the condition that they agree when restricted to $\Sigma_{>F}$ (in other words, homomorphisms form a sheaf on $\cL$ considered as a topological space whose open sets are upper sets).  So, to extend a morphism from $\Sigma'$ to $\Sigma$, it is enough to consider the case 
	where $\Sigma = \Sigma_{\ge F}$ and $\Sigma' = \Sigma_{>F}$.
	
	Our main results will be consequences of the following simple linear algebra lemma.  Let $\olkk$ denote the residue field of $\kk$.  
	If $A$ is a $\kk$-module, we put $\bar{A} = A \otimes_\kk \olkk$, and if $f\colon A \to B$ is a map of $\kk$-modules, we let $\bar{f}\colon \bar{A} \to \bar{B}$ be the induced map of $\overline{\kk}$-modules.  If $A$ and $B$ are free $\kk$-modules, then $\bar{f}$ is surjective if and only if $f$ is surjective, and $\bar{f}$ is injective if and only if $f$ is split injective.

	\begin{lemma}\label{lem:extension lemma}
		Suppose that we are given a commutative rectangle of free $\kk$-modules and maps of the following form:
		\[
\begin{tikzcd}[column sep=large, row sep=large]
A_1 \arrow[r, "\Psi_1"] \arrow[d, "f_A"']
  & B_1 \arrow[r, "\Phi_1"]
  & C_1 \arrow[d, "f_C"] \\
A_2 \arrow[r, "\Psi_2"']
  & B_2 \arrow[r, "\Phi_2"']
  & C_2
\end{tikzcd}
\]
		Then the following two statements hold:
		
		\begin{enumerate}
			\item[(a)] If $\Psi_1$ is split injective and $\Phi_2$ is surjective, there exists a map $f_B\colon B_1
			\to B_2$ so that both squares commute.
			\item[(b)] If in addition $\Psi_2$ is split injective, $\Phi_1$ is surjective, $\ker \bar\Phi_1 \subseteq \im \bar\Psi_1$, $\ker \bar\Phi_2 \subseteq \im \bar\Psi_2$, and $f_A$ and $f_C$ are isomorphisms, then any map $f_B$ as in (a) must be an isomorphism.
		\end{enumerate}
	\end{lemma}	
	\begin{proof}
		Let us prove the statement (a).  Since $\Psi_1$ is split injective, we can find a complement $B_0$ to $\Psi_1(A_1)$ in $B_1$, and there is a unique way to define the restriction $f_B|_{\Psi_1(A_1)}$ so that the left square commutes, by letting
		\[f_B(\Psi_1(a_1)) = \Psi_2(f_A(a_1))\]
		for $a_1 \in A_1$. This also makes the right square commute on $\Psi_1(A_1)$, since
		\[\Phi_2f_B(\Psi_1(a_1)) = \Phi_2\Psi_2(f_A(a_1)) = f_C(\Phi_1(\Psi_1(a_1)))\]
		for all $a_1 \in A_1$.  Finally, since $\Phi_2$ is surjective and $B_0$ is free we can find $f_B|_{B_0}$ so that the right square commutes on $B_0$.
		
		
		To prove (b), first note that $f_B$ is an isomorphism if and only if $\bar{f}_B$ is an isomorphism.  Thus it is enough to consider the case where $\kk$ is a field.
		
		Suppose that $b_1$ is in the kernel of $f_B$.  Since $f_C$ is an isomorphism, we have $b_1 \in \ker \Phi_1$, so there exists $a_1 \in A_1$ with $\Psi_1(a_1) = b_1$.  Then $\Psi_2 f_A(a_1) = f_B\Psi_1(a_1) = f_B(b_1) = 0$, and since $\Psi_2$ is injective we have $a_1 = 0$ and thus $b_1 = 0$. 
		
		Next, take $b_2 \in B_2$.  Since $\Phi_1$ and $f_C$ are surjective, there exists $b_1 \in B_1$ with $f_C\Phi_1(b_1) = \Phi_2(b_2)$.  Then $b_2 - f_B(b_1)$ is in $\ker \Phi_2 \subseteq \im \Psi_2$.  Since $f_A$ is surjective, there exists $a_1 \in A_1$ with $\Psi_2f_A(a_1) = b_2 - f_B(b_1)$.  Then we have
		\[f_B(b_1 + \Psi_1(a_1)) = f_B(b_1) + \Psi_2f_A(a_1) = f_B(b_1) + b_2 - f_B(b_1) = b_2,\]
		so $f_B$ is surjective.
	\end{proof}
	
	\begin{proposition}\label{prop:extending morphisms}
		Given upper sets $\Sigma' \subseteq \Sigma$ and objects $\V,\W$ in $\cC(\Sigma;\kk)_\pure$, any morphism $\V|_{\Sigma'} \to \W|_{\Sigma'}$ extends to a morphism $\V\to \W$.
	\end{proposition}
	\begin{proof}
		By induction on $|\Sigma\setminus\Sigma'|$, it is enough to consider the case $\Sigma' = \Sigma\setminus \{F\}$.  A morphism $f\colon \V|_{\Sigma'} \to \W|_{\Sigma'}$ gives rise to the following commutative diagram, where we include $\cV$ and $\cW$ in the subscripts
		to distinguish between maps associated with different objects:	
		\[
\begin{tikzcd}[column sep=large, row sep=large]
\Gamma_!\bigl(\Sigma_{>F},\V\bigr)[\rk F]
  \arrow[r, "\Psi_{F,\V}"]
  \arrow[d]
&
\V(F)
  \arrow[r, "\Phi_{F,\V}"]
&
\Gamma_*\bigl(\Sigma_{>F},\V\bigr)
  \arrow[d]
\\
\Gamma_!\bigl(\Sigma_{>F},\W\bigr)[\rk F]
  \arrow[r, "\Psi_{F,\W}"']
&
\W(F)
  \arrow[r, "\Phi_{F,\W}"']
&
\Gamma_*\bigl(\Sigma_{>F},\W\bigr)
\end{tikzcd}
\]
Lemma \ref{lem:(co)flabby extensions} and Proposition \ref{prop:(local) freeness} imply that all of these $\kk$-modules are free, $\Psi_{F,\V}$ is split injective, and $\Phi_{F,\W}$ is surjective, so Lemma \ref{lem:extension lemma} says that there is a map $\V(F) \to \W(F)$ making the squares commute.  This defines the required extension of $f$ to $\Sigma$.
	\end{proof}
	
	\subsection{Indecomposable pure objects}	
	The results of Sections \ref{sec:extending objects} and \ref{sec:extending morphisms} allow us to construct pure objects in $\cC(\Sigma;\kk)$ and morphisms between them inductively by adding one element at a time to $\Sigma$.  Our next objective is to give a similar inductive description of indecomposability.  The key definition is the following.
	\begin{definition}
		For an object $\cV$ of $\cC(\Sigma;\kk)$ with $F\in \Sigma$, we say that $\V$ is \textbf{taut at} $F$ if $\ker \overline{\Phi_F} \subseteq \im \overline{\Psi_F}$.
	\end{definition}
	
	\begin{lemma}\label{lem:tautness and (co)stalks}
Let $\M$ be a pure $\H(\cL;\kk)$-module.  Then $\Lambda\M$ is taut at $F$ if and only if the $\olkk$-reduction $\overline{\beta_F(\M)}$ of the costalk-stalk homomorphism $\beta_F(\M) \colon \M_{[F]} \to \M_F$ vanishes.
	\end{lemma}
	\begin{proof}
	The exact sequences in Lemma \ref{lem:boundaries and (co)stalks} are split by Proposition \ref{prop:(local) freeness}.  So after tensoring with $\olkk$ they remain exact, and so we have 
     $\ker \overline{\Phi_F}=\overline{\M_{[F]}}$ and $\coker\overline{\Psi_F}=\overline{\M_{F}}$. Hence, $\ker \overline{\Phi_F} \subseteq \im \overline{\Psi_F}$ if and only if $\overline{\beta_F(\M)}=0$.
	\end{proof}

	\begin{proposition}\label{prop:max indecomp extension}
		Given an object $\V'$ of $\cC(\Sigma';\kk)_{\pure}$, there exists an object $\V$ of $\cC(\Sigma;\kk)_\pure$ with $\V|_{\Sigma'} \cong \V'$ that is taut at $F$.  It is unique up to a (non-unique) isomorphism.
	\end{proposition}
	\begin{proof}
		By the results of Section \ref{sec:extending objects}, to construct $\V$ we need to define $\V(F)$ along with maps $\Psi_F$, $\Phi_F$ factoring $\Xi_F$ as in \eqref{eqn:extension triangle}, so that 
		$\Psi_F$ is a split injection,
		$\Phi_F$ is surjective, and
		$\ker \overline{\Phi_F} \subseteq \im \overline{\Psi_F}$.  To do this, choose a complement $\overline{W} \subseteq \overline{\Gamma_*(\Sigma_{>F},\V')}$ to the image of $\bar\Xi_F$.  Since $\Gamma_*(\Sigma_{>F},\V')$ is a free $\kk$-module by Proposition \ref{prop:(local) freeness}(d), we can lift $\overline{W}$ to a direct summand $W$ of $\Gamma_*(\Sigma_{>F},\V')$.  Define $\V(F) = \Gamma_!(\Sigma_{>F},\V')[\rk F] \oplus W$, let $\Psi_F$ be the inclusion of $\Gamma_!(\Sigma_{>F},\V')[\rk F]$, and let $\Phi_F$ equal $\Xi_F$ on the first summand and the inclusion on the second.  The required properties are now easy to verify.
		
		Lemma \ref{lem:extension lemma} implies that, for any two such extensions, the identity map on their restrictions to $\Sigma'$ extends to an isomorphism, so the extension $\cV$ is unique up to isomorphism.
	\end{proof}

	\begin{example}\label{rank two example}
		Continuing with the notation of Example \ref{ex:extension example}, we illustrate the result of applying Proposition \ref{prop:max indecomp extension} to the pure object $\cV' = (\Lambda\H(\cL;\kk))|_{\Sigma_{>\hat{0}}}$ with coefficient ring $\kk=\Z_2$. The degree $1$ part of the map 
		$\overline{\Xi_{\hz,\cV'}}$ is a rank $2$ map $(\F_2)^3 \to (\F_2)^3$, and a complement to its image is given by the subspace $W = \kk\cdot(1,1,1)\subseteq (\F_2)^3$.  If $\cV$ is the object constructed in Proposition \ref{prop:max indecomp extension}, then the degree one part of $\cV(\hz)$ is a free $\kk$-module of rank four, and one can check that $\cV$ is isomorphic to $\Lambda(\CH(\cL;\kk))$.  Note that this agrees with the statement of Theorem \ref{E=CH}.
	\end{example}
	
	\begin{proposition}\label{prop:extending indecomposables}
		Let $\V \in \cC(\Sigma;\kk)_{\pure}$ be an object with $\V|_{\Sigma'}\neq 0$.
		Then $\V$ is indecomposable if and only if $\V|_{\Sigma'}$ is indecomposable and $\V$ is taut at $F$.
	\end{proposition}
	\begin{proof}
%
		First suppose that $\V$ is not taut at $F$, so $\ker \overline{\Phi_F}$ is not contained in $\im \overline{\Psi_F}$.  We will show that $\cV$ is decomposable.  Since $\Psi_F$ is a split injection by Proposition \ref{prop:(local) freeness}, tensoring the quotient map $\pi\colon \V(F) \to \V(F)/\im \Psi_F$ with $\bar\kk$ gives the quotient
		\[\bar\pi\colon \overline{\V(F)} \to \overline{\V(F)/\im \Psi_F} \cong \overline{\V(F)}/\im \overline{\Psi_F}.\]
		By lifting an element of $\ker \overline{\Phi_F}$ that is not in $\im \overline{\Psi_F}$, we obtain $v \in \ker \Phi_F$ whose image $$\bar{v}\in\overline{\ker \Phi_F} = \ker \overline{\Phi_F} \subseteq \overline{\V(F)}$$ is not in the kernel of $\bar\pi$.  This implies that $\pi(\kk\cdot v)$ is a direct summand of the free $\kk$-module $\V(F)/\im \Psi_F = \V_F$.  Choose a complement $W \subseteq \V(F)/\im \Psi_F$ to $\pi(\kk\cdot v)$.
		
		We define a nontrivial direct sum decomposition $\V = \V_1 \oplus \V_2$ as follows:
		\begin{enumerate}[(1)]
			\item  For $G \ne F$, let $\V_1(G) = \V(G)$ and $\V_2(G) = 0$.
			\item Let $\V_1(F) = \pi^{-1}(W)$, $\V_2(F) = \kk\cdot v$.
		\end{enumerate}
		Because $\im \Psi_F \subseteq \V_1(F)$ and $\V_2(F) \subseteq \ker \Phi_F$, restricting the $\varphi$ and $\psi$ maps of $\V$ gives objects $\V_1$ and $\V_2$ of $\cC(\Sigma;\kk)$.  Since $\V(G) = \V_1(G) \oplus \V_2(G)$ for all $G\in \Sigma$, we have constructed the required direct sum decomposition; it is nontrivial since $\V_1|_{\Sigma'} = \V|_{\Sigma'}$ is nonzero and $\V_2(F)$ is nonzero. 
		
		Next, suppose that $\V$ is indecomposable but there is a non-trivial direct sum decomposition $\V|_{\Sigma'} = \V'_1 \oplus \V'_2$.  By Proposition \ref{prop:max indecomp extension}, there are pure extensions $\V_1$ and $\V_2$ of $\V'_1$ and $\V'_2$ that are taut at $F$.  But we have just seen that $\V$ must be taut at $F$, and $\V_1 \oplus \V_2$ is also taut at $F$, so the
		uniqueness part of Proposition \ref{prop:max indecomp extension} implies that $\V \cong \V_1 \oplus \V_2$, contradicting the indecomposability of $\V$.
		
		Finally, suppose that we are given a non-trivial decomposition $\V = \V_1 \oplus \V_2$ but 
		$\V|_{\Sigma'}$ is indecomposable. Then one of $\V_1|_{\Sigma'}$ or $\V_2|_{\Sigma'}$ must vanish;  suppose without loss of generality that $\V_2|_{\Sigma'} = 0$.  Then we must have $\V_2(F) \ne 0$, and so $\overline{\V_2(F)} \ne 0$.   The vanishing of $\V_2|_{\Sigma'}$ implies that $\overline{\V_2(F)} \subseteq \ker \overline{\Phi_F}$ and $\im \overline{\Psi_F} \subseteq \overline{\V_1(F)}$, so it follows that $\V$ is not taut at $F$. 
	\end{proof}

	\begin{theorem}\label{thm:indecomposables}
		For each element $F \in \Sigma$, there is an object $\eE^F = \eE^F_{\Sigma,\kk}$ of $\cC(\Sigma;\kk)_\pure$ for which
		the following properties hold: 
		\begin{enumerate}[(a)]
			\item $\eE^F(G) = 0$ unless $G \le F$,
			\item $\eE^F(F) \cong \kk$, concentrated in degree $0$,
			\item $\eE^F$ is taut at all elements in $\Sigma \setminus \{F\}$.
		\end{enumerate}
		This object is indecomposable and unique up to isomorphism.  The collection of objects $\eE^F[k]$ over all $F\in \Sigma$ and all $k\in \Z$ gives a complete set of indecomposable objects in $\cC(\Sigma;\kk)_\pure$.	
	\end{theorem}
	\begin{proof}
		We proceed by induction on $|\Sigma|$.  Since the dependence on $\Sigma$ is important in this proof but the dependence on $\kk$ is not, for the remainder of the proof we will denote the indecomposables described in the theorem by $\eE^F_\Sigma$.
		
		When $|\Sigma| = 1$, $\cC(\Sigma;\kk)_\pure$ is equivalent to the category of free graded $\kk$-modules, and the result is immediate.  Now assume that $|\Sigma|>1$ and the result is known for all upper sets that are properly contained in $\Sigma$.  Let $F$ be a minimal element in $\Sigma$ and let $\Sigma' = \Sigma\setminus\{F\}$.  By the inductive hypothesis, we have modules $\eE_{\Sigma'}^G$ for $G \in \Sigma'$ representing all isomorphism classes of indecomposable objects in $\cC(\Sigma';\kk)_\pure$ up to shift.  By Proposition \ref{prop:max indecomp extension}, each object $\eE^G_{\Sigma'}$ has an extension $\eE_\Sigma^G$ that is pure and taut at $F$, and Proposition \ref{prop:extending indecomposables} implies that this extension is indecomposable.  We have one more indecomposable object $\eE_\Sigma^F$ in $\cC(\Sigma;\kk)_\pure$ given by setting $\eE^F_\Sigma(F) = \kk$ in degree $0$ and $\eE^F_\Sigma(G) = 0$ for all $G\ne F$.  It is easy to see that these objects satisfy the 
		required conditions.
		
		Now let $\V$ be any nonzero indecomposable object of $\cC(\Sigma;\kk)_\pure$.  If $\V|_{\Sigma'}$ is nonzero, then Proposition \ref{prop:extending indecomposables} implies that $\V|_{\Sigma'}$ is indecomposable and $\V$ is taut at $F$. Thus we must have $\V|_{\Sigma'}\cong \eE^G_{\Sigma'}[k]$ for some $G\ne F$ and some $k \in \Z$.  But then $\V \cong \eE_\Sigma^G[k]$ by the uniqueness part of Proposition \ref{prop:max indecomp extension}.  On the other hand, if $\V|_{\Sigma'} = 0$, then $\V(F)$ must be a free $\kk$-module of rank $1$, which means that  $\V \cong \eE_\Sigma^F[k]$ for some $k \in \Z$.	
	\end{proof}

\subsection{Parity modules}
In this section, we prove Theorems \ref{def of E}, \ref{indecomposables}, and \ref{general decomposition}.

\begin{definition}
	Let $\E(\cL;\kk) \coloneq \Theta\eE^{\hat{1}}_{\cL,\kk}$, where $\Theta\colon \cC \to \H(\cL;\kk)\mmod$ is the functor introduced in the proof
	of Proposition \ref{thm:main equivalence} and $\eE^{\hat{1}}_{\cL,\kk}$ is the indecomposable pure object from Theorem \ref{thm:indecomposables}.  
\end{definition}

Recall from the introduction that, for any $F\in\cL$, we have a natural surjective homomorphism $\H(\cL;\kk)\to\H(\cL^F;\kk)$,
and this means that $\E(\cL^F;\kk)$ is naturally a module over $\H(\cL;\kk)$.  The following lemma follows from Remark \ref{lower intervals}.

\begin{lemma}\label{unexploding indecomposables}
For any $F\in\cL$, we have $\E(\cL^F;\kk) \cong \Theta\eE^{F}_{\cL,\kk}$.
\end{lemma}

\begin{proof}[Proof of Theorem \ref{indecomposables}]
By Proposition \ref{thm:main equivalence}, Remark \ref{purepure},
and Lemma \ref{pure implies regular}, a complete set of indecomposable pure $\H(\cL;\kk)$-modules (up to shift and isomorphism)
is obtained from a complete set of indecomposable pure objects of $\cC(\cL;\kk)$ (up to shift and isomorphism) via the functor $\Theta$.
By Theorem \ref{thm:indecomposables} and Lemma \ref{unexploding indecomposables}, one obtains the modules $\E(\cL^F;\kk)$ for all $F\in\cL$.
\end{proof}

\begin{proof}[Proof of Theorem \ref{def of E}]
By Theorem \ref{indecomposables}, $\E(\cL;\kk)$ is pure and indecomposable.  Furthermore,
we have \[y_{\hat{1}}\E(\cL;\kk) = \eE^{\hat{1}}_{\cL,\kk}(\hat{1})[-\rk \cL] \cong \kk[-\rk \cL],\] which implies the existence of an
element $x\in \E^0(\cL;\kk)$ such that $y_{\hat 1}\cdot x$ is nonzero.

Conversely, let $\M$ be an indecomposable pure module and $x\in\M^0$ an element with $y_{\hat{1}}x\neq 0$.
By Theorem \ref{indecomposables}, there exists an element $F \in \cL$ and an integer $k\in \Z$ such that $\M\cong\E(\cL^F;\kk)[k]$.  
By Lemma \ref{unexploding indecomposables} and parts (a) and (b) of Theorem \ref{thm:indecomposables}, we must have $F = \ho$ and $k=0$.
\end{proof}

\begin{remark}\label{rmk:alternate contraction module}
Our construction of parity modules also gives an alternative proof of Proposition \ref{contraction module}.  To see this, note that the restriction of $\eE^\ho_{\cL,\kk}$ to the upper set $\cL_F$ is clearly isomorphic to $\eE^\ho_{\cL_F,\kk}$.  Applying the functor $\Theta_F\colon \cC(\cL_F;\kk) \to \H(\cL_F;\kk)\mmod$ to this isomorphism gives an isomorphism
\[y_F\E(\cL;\kk)[\rk F] \cong \E(\cL_F;\kk).\]	
\end{remark}

In \cite{SHCG}, a graded $\H(\cL;\Q)$-module is called \textbf{pure} if it is isomorphic to a direct sum of direct summands of 
modules of the form $\CH(\cL^F;\Q)[k]$, where $F \in \cL$ and $k \in \Z$.
The next result says that this definition agrees with our definition of pure modules.  
\begin{proposition}\label{prop:pure equals pure}
Let $\kk$ be a field.  An $\H(\cL;\kk)$-module is pure if and only if it is isomorphic to 
a direct sum of direct summands of modules of the form $\CH(\cL^F;\kk)[k]$, where $F \in \cL$ and $k \in \Z$.
\end{proposition}

\begin{proof}
	By Proposition \ref{true facts}(e), $\CH(\cL^F;\kk)[k]$ is pure, and since purity is clearly preserved under taking direct sums and direct summands, the ``if" statement holds. 
	
	Considering $\CH(\cL^F;\kk)$ as an $\H(\cL;\kk)$-module via the homomorphism $\H(\cL;\kk) \to \H(\cL^F;\kk)$, we have 
	$y_F\CH(\cL^F;\kk) \cong \kk[-\rk F]$ and $y_G\CH(\cL^F;\kk)=0$ for $G \not\le F$.  It follows that 
	$\E(\cL^F;\kk)$ is isomorphic to a direct summand of $\CH(\cL^F;\kk)$. Thus, by Theorem \ref{indecomposables},
	every indecomposable pure $\H(\cL;\kk)$-module appears as a shift of a direct summand of $\CH(\cL^F;\kk)$ for some $F\in\cL$.
	This proves the ``only if" direction.
\end{proof}

Theorem \ref{indecomposables} implies that any pure $\H(\cL;\kk)$-module $\M$ can be decomposed as a direct sum 
\begin{equation}\label{eqn:pure decomposition}
	\M \cong \bigoplus_{F\in \cL} W(\M,F)\otimes \E(\cL^F;\kk)
\end{equation}
for some finite-dimensional graded vector spaces $W(\M,F)$.  The next result, which proves Theorem \ref{general decomposition}, shows how to compute these multiplicity spaces. 

\begin{proof}[Proof of Theorem~\ref{general decomposition}]
	As with all of the results stated in the introduction, we are assuming that $\kk$ is a field.  We must show that $W(\M,F)$ is isomorphic to the image of the costalk-stalk homomorphism $\beta_F(\M) \colon \M_{[F]}\to \M_F$.  This follows by applying $\im\left(\beta_G(-)\right)$ to both sides of \eqref{eqn:pure decomposition}. Indeed, if $F\ne G$, then $\eE^F_{\cL,\kk} \cong\Lambda\E(\cL^F;\kk)$ is taut at $G$.  By Lemma \ref{lem:tautness and (co)stalks}, 
this implies that $\beta_G(\E(\cL^F;\kk)) = 0$, and therefore the $F$ summand of the right-hand side vanishes. On the other hand, the maps
	\[
	\kk\cong \E(\cL^G;\kk)_{[G]} \longrightarrow y_G\E(\cL^G;\kk)[\rk G] \longrightarrow \E(\cL^G;\kk)_G \cong \kk
	\]
are isomorphisms, therefore the $G$ summand of the right-hand side is isomorphic to $W(\M,G)$.  Thus we obtain the equation
$\im\left(\beta_G(\M)\right)\cong W(\M,G)$.
\end{proof}

\begin{remark}
	Theorem \ref{general decomposition} does not hold when $\kk$ is a DVR.  This is because  it is possible for a pure module $\M$ to be taut at $F$ even though $\beta_F(\M)$ is not zero.  The object in Example \ref{rank two example} is an example.  Instead, the multiplicities in the decomposition of a pure $\H(\cL;\kk)$-module are given by the graded dimensions of the images of the reduced maps $\overline{\beta_F(\M)}$.
\end{remark}

\section{Properties of parity sheaves}\label{sec:properties}
The purpose of this section is to prove Theorems \ref{props of E} and \ref{it goes up}.
	
	\subsection{Top and bottom degrees}\label{sec:computations}
	In addition to giving a classification of indecomposable pure objects, Theorem \ref{thm:indecomposables} (or rather its proof) also gives an algorithm to construct them.  Namely, to construct $\eE^F = \eE^F_{\cL,\kk}$, start with $\eE^F(G) = 0$ for $G \not \le F$, $\eE^F(F) = \kk$, and then use Proposition \ref{prop:max indecomp extension} to define $\eE^F(G)$ for $G <F$ in decreasing order of rank.
The following proposition is an application of this algorithm.
	
	\begin{proposition}\label{prop:E in small and large degree}
		The indecomposable pure object $\eE^F = \eE^F_{\cL,\kk}$ satisfies the following properties. 
		\begin{enumerate}
			\item[(a)] For any $G \le F$, we have $\eE^F(G)^i = 0$ if $i < 0$ or $i > \rk F - \rk G$.
			\item[(b)] For any $G \le F$, we have $\eE^F(G)^0 \cong \eE^F(G)^{\rk F - \rk G} \cong \kk$.
			\item[(c)] For any $G \le H \le F$, the restriction map
			\[\varphi_G^H\colon \eE^F(G)^0 \longrightarrow \eE^F(H)^0\]
			and the corestriction map
			\[\psi_G^H\colon \eE^F(H)^{\rk F -\rk H} \longrightarrow \eE^F(G)^{\rk F - \rk G}\]
			are both isomorphisms.
		\end{enumerate}
	\end{proposition}
	\begin{proof}
We will proceed by induction on the difference $\rk F - \rk G$.  When $G = F$, the statements (a) and (b) hold because $\eE^F(F)$ is equal to $\kk$ concentrated in degree $0$.  The only case of (c) to check is $G = H = F$, which holds trivially since $\varphi_F^F$ and $\psi_F^F$ are both the identity map. 
		
		Now suppose that we are given $G_0 < F$ such that the proposition holds for all $G$ with $G_0 < G\leq F$.
		We apply these statements to the sections and cosections of $\eE^F$ on $\Sigma_{>G_0}$.  Looking at sections
		\[\Gamma_*(\Sigma_{>G_0},\eE^F) = \varprojlim_{G_0< G} \eE^F(G)= \varprojlim_{G_0< G\leq F} \eE^F(G),\]
		the fact that $\eE^F(G)^i$ is zero when $i < 0$ or $i > \rk F - \rk G$ implies that $\Gamma_*(\Sigma_{>G_0},\eE^F)^i = 0$ for $i < 0$ and $i \ge \rk F - \rk G_0$.  Furthermore, statement (c) implies that
		all of the maps in the limit between elements $G_0 < G \le F$ are isomorphisms in degree $0$, and so 
		$\Gamma_*(\Sigma_{>G_0},\eE^F)^0 \cong \kk$ and the maps $\Gamma_*(\Sigma_{>G_0},\eE^F)^0 \to \eE^F(G)^0$ are isomorphisms for all $G_0 < G \le F$.  
		
		Dually, when we look at cosections
		\begin{equation*}\label{eqn:cosections above G_0}
			\Gamma_!(\Sigma_{>G_0},\eE^F) = \varinjlim_{G_0 < G} \eE^F(G)[-\rk G] = \varinjlim_{G_0 < G\leq F} \eE^F(G)[-\rk G],	
		\end{equation*} 
		the statement (a) implies that $\eE^F(G)[-\rk G]$ vanishes in all degrees $i < \rk G$ and $i> \rk F$, so the limit 
		vanishes for all degrees $i \le \rk G_0$ and $i > \rk F$.
		Statement (c) implies that all of the maps in the limit 
		between elements $G_0 < G \le F$ are isomorphisms in degree $\rk F$, so we have $\Gamma_!(\Sigma_{>G_0},\eE^F)^{\rk F} \cong \kk$ and the maps $\eE^F(G)[-\rk G]^{\rk F}\to \Gamma_!(\Sigma_{>G_0},\eE^F)^{\rk F}$ are isomorphisms for all $G_0 < G \le F$.  
		
		We now consider the maps
		\[
\Gamma_!\bigl(\Sigma_{>G_0},\eE^F\bigr)[\rk G_0]
\xrightarrow{\ \Psi_{G_0}\ }
\eE^F(G_0)
\xrightarrow{\ \Phi_{G_0}\ }
\Gamma_*\bigl(\Sigma_{>G_0},\eE^F\bigr).
\]
		In degrees $i < 0$ and $i>\rk F - \rk G_0$, both the sections and cosections vanish, so $\eE^F(G_0)^i = 0$, giving the statement (a).  In degree $i = 0$, the cosections vanish while the sections are $\kk$, so we have $\eE^F(G_0)^0 \cong \kk$.  In degree $\rk F - \rk G_0$, the sections vanish while the cosections are $\kk$,  giving $\eE^F(G_0)^{\rk F-\rk G_0} \cong \kk$, so we have proved the statement (b) for $G_0$.  It also follows that the  statement (c) holds for $G_0$.	 
	\end{proof}

\begin{remark}
Applying the functor $\Theta$ to $\eE^{\ho}$, Proposition \ref{prop:E in small and large degree} immediately gives us the first two statements
of Theorem \ref{props of E}.  We will tackle the third statement in Section \ref{sec:Kunneth}.
\end{remark}
	
\subsection{Next degrees in}\label{sec:next}
Assume in this section that $\kk$ is a field of characteristic $p$ (allowing the possibility that $p=0$).  
Our next result gives a simple description of the spaces $\eE^F(G)$ in the 
next degrees ``in from the outside", that is, in degrees $1$ and $\rk F-\rk G -1$. 
This gives an alternative proof of Proposition \ref{resonance linear} that does not involve the embedding $\E(\cL;\kk)$ in the augmented Chow ring 
$\CH(\cL;\kk)$.  A reader who is only interested in the proofs of Theorems \ref{props of E} and \ref{it goes up} may skip this section.
	
To simplify notation, we will restrict ourselves to the case where $F$ is the maximal element $\ho$.  Let $\eE = \eE^\ho$.
For an upper set $\Sigma\subseteq\cL$, we let $\Sigma(p)$ denote the set of $p$-resonant elements in $\Sigma$, where $p$ is the characteristic of $\kk$. 
	
	\begin{theorem}\label{thm:linear term}
		For any $G < \ho$, there are isomorphisms
		\[
		\eE(G)^1 \cong \kk^{\cL_G(p)} \cong\eE(G)^{\crk G - 1} 
		\]
		with the property that, for any $G < F <\ho$, the maps
		\[
		\psi_G^F\colon \eE(F)^{\crk F-1} \longrightarrow \eE(G)^{\crk G-1} \and  \varphi_G^F\colon \eE(G)^1 \longrightarrow \eE(F)^1
		\]
		are given by the coordinate inclusion $\kk^{\cL_F(p)} \to \kk^{\cL_G(p)}$ and projection  $\kk^{\cL_G(p)} \to \kk^{\cL_F(p)}$, respectively.
		 If in addition $F$ is a cover of $G$, so that $\crk F = \crk G - 1$, the map 
		\[
		\psi_G^F\colon \eE(F)^{0} \longrightarrow \eE(G)^{1}
		\]
		is the map $\kk \to \kk^{\cL_G(p)}$ sending $1$ to the characteristic function of $\cL_G(p) \smallsetminus \cL_{F\vee G}(p)$,		
		and the map $\varphi_G^F\colon \eE(G)^{\crk G - 1} \to \eE(F)^{\crk F} \cong \kk$ is its transpose.
	\end{theorem}

Before proving Theorem \ref{thm:linear term}, we formulate and prove a lemma.
Given a map $a\colon \cL^1 \to \kk$, its \textbf{complementary Radon transform} is the map $C_a\colon \cL \to \kk$ given by  
	\[ 
	C_a(F) = \sum_{D\in \cL^1 \,  \text{and}  \, D \not\le F} a(D).
	\]
The following lemma is proved in \cite[Section 3]{KungRadonI} under the additional assumption that $p=0$; we prove it below
for arbitrary fields.

	\begin{lemma}\label{lem:Radon}
		The map 
		$a \mapsto C_a|_{\cL(p)}$
		is injective.  If $\hz$ is $p$-resonant, then there exists an $a$ such that, for any $F \in \cL(p)$, we have 
		$C_a(F) = \delta_{\hz F}$.
	\end{lemma} 
	\begin{proof}
		First we prove injectivity.  Suppose that $C_a$ vanishes on $\cL(p)$.  If we can show that $C_a$ vanishes on all of $\cL$, it will follow that $a = 0$, since for any 
		$D \in \cL^1$ we have
		\[a(D) = C_a(\hz) - C_a(D).\]
		
		We prove that $C_a(F)=0$ for all $F\in\cL$ by induction on the corank of $F$.  When $F = \ho$ the result is obvious, and when $F$ has corank $1$ it is $p$-resonant, so $C_a(F) = 0$ by assumption.  
		Now take $F\in \cL$ of corank at least $2$, and assume that $C_a$ vanishes on all elements of smaller corank.
		If $F$ is $p$-resonant, then $C_a(F) = 0$ by assumption.  Otherwise, there is an element $H > F$ with $\rk H = \rk F + 2$ such that the 
		open interval $(F,H)$
		has cardinality $m$ with $m-1$ invertible in $\kk$.  
	 	An atom $D\in\cL^1$ with $D \le H$ and $D \not\le F$ lies below exactly one $G \in (F,H)$, and hence we have
		\[\sum_{G\in (F,H)} C_a(G) = (m-1)C_a(F) + C_a(H).\]
		Our inductive hypothesis says that the left sum and $C_a(H)$ both vanish, so $C_a(F) = 0$, as desired.  
		
		For the second statement, suppose that $\hz$ is $p$-resonant.  
		Let $a(D) = 1$ for all $D \in \cL^1$, so that \[C_a(F) = |\cL^1| - |(\cL^F)^1|.\]
		Lemma \ref{two is enough} implies that, for any $F>\hz$, $|(\cL^F)^1|$ is congruent to $1$ modulo $p$.
		Thus $C_a(F) = 1-1 = 0$ for all $F\neq\hz$, while $C_a(\hz) = 1-0 = 1$.
	\end{proof}

	\begin{proof}[Proof of Theorem \ref{thm:linear term}]
		We will prove the statements about $\eE(G)^1$; the statements about $\eE(G)^{\crk G -1}$ follow from duality.  We proceed by induction on $\crk G$.  When $\crk G = 1$, $G$ is the only $p$-resonant element of $\cL_G$,
and Proposition \ref{prop:E in small and large degree} says that
		\[\eE(G)^1 = \eE(G)^{\crk G} \cong \kk \cong \kk^{\cL_G(p)}.\]
The statements about $\psi_G^F$ and $\varphi_G^F$ are trivial in this case.
		
		Now suppose that $\crk G > 1$, and that the statements hold for all elements of smaller corank.
		We need to compute the degree $1$ part of the map $\Xi_G\colon \Gamma_!(\cL_{>G},\eE)[\rk G] \to \Gamma_*(\cL_{>G},\eE)$.  The source of this map is
\[
			\Gamma_!(\cL_{>G},\eE)[\rk G]^1  = \Gamma_!(\cL_{>G},\eE)^{\rk G +1}
			 = \varinjlim_{F > G} \eE(F)[-\rk F]^{\rk G + 1}
			 = \varinjlim_{F > G} \eE(F)^{\rk G - \rk F + 1}
			 \cong \kk^{(\cL_G)^1}.
\]
		The last isomorphism follows from Proposition \ref{prop:E in small and large degree} because, for $F>G$, the quantity $\rk G - \rk F + 1$ is negative unless $F$ covers $G$, in which case it is zero.  We represent elements of this space as functions $a\colon (\cL_G)^1 \to \kk$.
		
		For any $F'$ with $G<F'\leq F$, our inductive hypothesis tells us that the restriction map $\varphi^F_{F'}$
		is a coordinate projection, thus the degree one part of the target of $\Xi_G$ is 
		\[\Gamma_*(\cL_{>G},\eE)^1 = \varprojlim_{F>G} \kk^{\cL_F(p)} = \kk^{\Sigma_{>G}(p)}.\]
		If we represent elements of this space as functions $\Sigma_{>G}(p) \to \kk$, we claim that 
		\[\Xi_G(a) = C_a|_{\Sigma_{>G}(p)},\]
		where the complementary Radon transform $C_a$ is computed in the lattice $\cL_G = \Sigma_{\geq G}$.  
		
		To see this, let $a_F$ be the basis element in $\kk^{(\cL_G)^1}$ corresponding to a cover $F$ of $G$.  It is the image of $1 \in \eE(F)^0$ in $\Gamma_!(\cL_{>G},\eE)[\rk G]^1$.  Then $\Xi_G(a_F)$ is determined by its image in $\eE(H)^1$ over all $H > G$.  If $F \le H$, then the image is zero, since $F \wedge H = F \ne G$.  This agrees with the claim, since $C_{a^{}_F}$ is zero on $\cL_H$ when $H \ge F$.  On the other hand, if $F \not\le H$, then $F\vee H$ covers $H$, and the image of $\Xi_G(a_F)$ in $\eE(H)^1$ is the image of $1$ under the map
		\[
\kk
=
\eE(F)^0
\xrightarrow{\ \varphi_F^{F\vee H}\ }
\eE(F\vee H)^0
\xrightarrow{\ \psi_H^{F\vee H}\ }
\eE(H)^1
\cong
\kk^{\cL_H(p)}.
\]
	    By the inductive assumption, $\varphi_F^{F\vee H}$ is the identity and 
	    $\psi_H^{F\vee H}$ sends $1$ to the characteristic function of $\cL_H(p) \smallsetminus \cL_{F\vee H}(p)$, which is the restriction of $C_{a^{}_F}$ to $\cL_{H}(p)$.  The claim follows.
	    
	    Now consider the diagram 
	    	\begin{equation*}
	    	\begin{tikzcd}
	    		\kk^{(\cL_G)^1} \arrow[rr, "\Xi_G"] \arrow[rd, "\Psi_G"'] &                           & \kk^{\Sigma_{>G}(p)} \\
	    		&  \kk^{\cL_G(p)}\arrow[ru, "\Phi_G"'] &                
	    	\end{tikzcd}
	    \end{equation*}
	    where $\Psi_G$ is the map $a \mapsto C_a|_{\cL_G(p)}$ and $\Phi_G$ is the restriction to $\Sigma_{>G}(p)$.  The triangle commutes, so this diagram describes the degree $1$ part of an extension of $\eE|_{\Sigma_{>G}}$ to $\Sigma_{\geq G} = \cL_G$.  This extension is pure, because 	    
	    $\Phi_G$ is surjective and $\Psi_G$ is injective, by Lemma \ref{lem:Radon}.  Finally, we must check that the extension is taut at $G$.  If $G$ is not $p$-resonant, then $\Phi_G$ is an isomorphism, and tautness is automatic.  If $G$ is $p$-resonant, then the kernel of $\Phi_G$ is spanned by the characteristic function of $G$, which is in the image of $\Psi_G$ by the second part of Lemma \ref{lem:Radon}.  
	\end{proof}
	
	\subsection{K\"unneth isomorphism}\label{sec:Kunneth}
In this section we prove the last part of Theorem \ref{props of E}, 
which is a K\"unneth isomorphism for parity modules.  
Let $\cL_1$ and $\cL_2$ be geometric lattices.  Then $\cL_1 \times \cL_2$ is a geometric lattice, with rank function $\rk(F_1,F_2) = \rk F_1 + \rk F_2$.  We have an isomorphism of rings
	\[
	\H(\cL_1;\kk)\otimes \H(\cL_2;\kk) \cong \H(\cL_1 \times \cL_2;\kk), \quad y^{}_{F_1} \otimes y^{}_{F_2} \longmapsto y_{(F_1, F_2)}.
	\]


	
	\begin{lemma}\label{pure tensor}
		Let $\M_i$ be a graded $\H(\cL_i)$-module for $i=1,2$.  If $\M_1$ and $\M_2$ are both pure, then so is 
		their external tensor product $\M_1\otimes \M_2$.
	\end{lemma}
	\begin{proof}
Let $\M = \M_1\otimes M_2$.
To show that $\M$ is flabby and coflabby, we will use the criterion of Lemma~\ref{lem:(co)flabby extensions}.  Let $\V = \Lambda\M$, and take $(F_1,F_2) \in \cL_\times\cL_2$.  We wish to show that the corestriction 
		\[\Gamma_!(\Sigma_{>(F_1,F_2)},\cV) \longrightarrow \Gamma_!(\Sigma_{\ge(F_1,F_2)},\cV) = y_{(F_1,F_2)}\M = y_{F_1}\M_1 \otimes y_{F_2}\M_2\]
		is injective and the restriction
		\[\left(y_{F_1}\M_1 \otimes y_{F_2}\M_2\right)[\rk F_1 + \rk F_2] = \Gamma_*(\Sigma_{\ge (F_1,F_2)},\cV) \longrightarrow \Gamma_*(\Sigma_{>(F_1,F_2)},\cV)\]
		is surjective.
		
		Using the cover 
		\[\Sigma_{>(F_1,F_2)} = (\Sigma_{>F_1} \times \Sigma_{\ge F_2}) \cup (\Sigma_{\ge F_1}\times \Sigma_{>F_2})\]
		and the fact that $\Gamma_!(-,\cV)$ is a cosheaf gives us an exact sequence
		\[
		\Gamma_!(\Sigma_{>F_1} \times \Sigma_{>F_2},\V) \longrightarrow \Gamma_!(\Sigma_{>F_1} \times \Sigma_{\ge F_2},\V) \oplus \Gamma_!(\Sigma_{\ge F_1} \times \Sigma_{>F_2},\V) \longrightarrow \Gamma_!(\Sigma_{>(F_1,F_2)},\V) \longrightarrow 0.
		\]
		Since the subset $\Sigma_{>F_1}\times \{F_2\}\subseteq \Sigma_{>F_1} \times \Sigma_{\ge F_2}$ is cofinal, we have
\[
		     \Gamma_!(\Sigma_{> F_1} \times \Sigma_{\ge F_2},\V)  = \varinjlim_{\substack{G_1 > F_1\\G_2 \ge F_2}} y_{G_1}\M_1 \otimes y_{G_2}\M_2 
		      \cong \varinjlim_{G_1 > F_1} y_{G_1}\M_1 \otimes y_{F_2}\M_2 
		      \cong \left(\varinjlim_{G_1 > F_1} y_{G_1}\M_1\right) \otimes y_{F_2}\M_2,  	
		     \]
		where the last isomorphism holds because $y_{F_2}\M_2$ is a free $\kk$-module.  Since $\M_1$ is coflabby and locally free, the map from this space of cosections to $\M_1 \otimes \M_2$ is a split injection, with image
		\[\left( \Span_{\Sigma_{>F_1}} \M_1\right)\otimes y_{F_2}\M_2 = \Span_{\Sigma_{>F_1} \times \Sigma_{\geq F_2}} \M.\]
Similarly, $\Gamma_!(\Sigma_{\geq F_1} \times \Sigma_{> F_2},\V)$ maps by a split injection into $\M_1 \otimes \M_2$ with image
		\[y_{F_1}\M_1 \otimes \left( \Span_{\Sigma_{>F_2}} \M_2\right)  = \Span_{\Sigma_{\geq F_1} \times \Sigma_{>F_2}} \M.\]
		The intersection of these images is 
		\[\left( \Span_{\Sigma_{>F_1}} \M_1\right)\otimes \left( \Span_{\Sigma_{>F_2}} \M_2\right) = \Span_{\Sigma_{>F_1} \times \Sigma_{>F_2}} \M,\]
		which is the image of $\Gamma_!(\Sigma_{>F_1} \times \Sigma_{>F_2},\V)$ in $\M_1\otimes \M_2$.  
		It follows that the map
		\[\Gamma_!(\Sigma_{>(F_1,F_2)},\V) \to \Gamma_!(\Sigma_{\ge(F_1,F_2)},\V) = y_{F_1}\M_1\otimes y_{F_2}\M_2\]
		is a split injection, which proves that $\M$ is coflabby and locally free.
		
		The proof that $\M$ is flabby is similar, using the sheaf property of $\Gamma_*(-,\V)$.
	\end{proof}
	
\begin{proof}[Proof of Theorem~\ref{props of E}]
Parts (a) and (b) follow from Proposition \ref{prop:E in small and large degree}, so it will suffice to prove part (c).
We need to show that 
$\E(\cL_1) \otimes \E(\cL_2)$ satisfies the conditions of Theorem \ref{def of E}. 
 Indecomposability is clear, and purity follows from Lemma \ref{pure tensor}.  Finally,
 we have
		\[y_{(\hat{1},\hat{1})}\big(\E(\cL_1) \otimes \E(\cL_2)\big) = y_{\hat 1}\E(\cL_1) \otimes y_{\hat 1}\E(\cL_2) \cong \kk[-\rk \cL_1] \otimes \kk[-\rk \cL_2] \cong \kk[-\rk(\cL_1\times \cL_2)].\qedhere\]
\end{proof}

\subsection{Purity and base change}\label{sec:base change}
In this section, we consider how pure modules behave when the coefficient ring $\kk$ changes.
Our ultimate goal is to prove Theorem \ref{it goes up}, along with the statement that parity modules
behave nicely under field extensions.

Suppose that $\kk \to \kk'$ is a ring homomorphism, where $\kk$ is either a field or a DVR and $\kk'$ is a field.  We will mainly be interested in the cases (1) $\kk'$ is an extension field of $\kk$, (2) $\kk$ is a DVR and $\kk' = \olkk$ is its residue field, and (3) $\kk$ is a DVR and $\kk' = \mathbb{K}$ is its fraction field.
	Given an upper set $\Sigma$, we have a functor $- \otimes_\kk \kk'\colon \cC(\Sigma;\kk)\to \cC(\Sigma;\kk')$. 
	
	\begin{theorem}\label{thm:base change preserves purity}
		Let $\V$ be an object of $\cC(\Sigma;\kk)$.
		\begin{enumerate}[(a)]
			\item For any $\Sigma'\subseteq \Sigma$, the natural map
			$\Gamma_!(\Sigma',\V)\otimes_\kk\kk' \to \Gamma_!(\Sigma',\V\otimes_\kk\kk')$
			is an isomorphism.
			\item If $\V$ is pure, then the natural map
			$\Gamma_*(\Sigma',\V)\otimes_\kk\kk' \to \Gamma_*(\Sigma',\V\otimes_\kk\kk')$
			is an isomorphism.
			\item If $\V$ is pure, then so is $\V \otimes_\kk\kk'$. 
		\end{enumerate}
	\end{theorem} 
	\begin{proof}
		Tensoring with $\kk'$ is right exact, which means that it commutes with finite colimits, and in particular it commutes with taking cosections.  This proves (a).
		
		The proof of (b) is by induction on $|\Sigma'|$.  When $|\Sigma'|=1$, or more generally when 
		$\Sigma' = \Sigma_{\ge F}$ for an element $F$, the statement is true by definition.  So assume $|\Sigma'|>1$, $\Sigma'$ has more than one minimal element, and the statement is true for all smaller upper sets.  Let $F$ be a minimal element of $\Sigma'$ and let $\Sigma'' = \Sigma' \setminus \{F\}$.  
		
Since $\cV$ is flabby, Lemma \ref{lem:boundaries and (co)stalks}(a) gives us an exact sequence
			\[
			0 \longrightarrow \V_{[F]} \longrightarrow  \Gamma_*(\Sigma',\V) \longrightarrow \Gamma_*(\Sigma'',\V) \longrightarrow 0.
			\]
Proposition \ref{prop:(local) freeness}(d) implies that $\Gamma_*(\Sigma'',\V)$ is free, thus this exact sequence splits.
It follows that tensoring with $\kk'$ preserves exactness.
		Consider the following diagram:
		\[\xymatrix{
			0 \ar[r] & \V_{[F]} \otimes_\kk\kk' \ar[d]\ar[r] & \Gamma_*(\Sigma',\V)\otimes_\kk\kk' \ar[d]\ar[r] & \Gamma_*(\Sigma'',\V)\otimes_\kk\kk'	\ar[r]\ar[d] & 0\\
			0 \ar[r] & (\V \otimes_\kk\kk')_{[F]}  \ar[r] & \Gamma_*(\Sigma',\V\otimes_\kk\kk') \ar[r] & \Gamma_*(\Sigma'',\V\otimes_\kk\kk')	\ar[r] & 0
		}\]
We have already argued that the top row is exact.
Left-exactness of the bottom row follows from Lemma \ref{lem:boundaries and (co)stalks}.  The right vertical map is an isomorphism by our induction hypothesis, which implies that the bottom row is also right exact.
		
We want to prove the middle vertical map is an isomorphism, and by the Five Lemma, it is sufficient to show 
that the left vertical map is an isomorphism.  To do this, we consider the same diagram with $\Sigma'$
replaced by $\Sigma_{\ge F}$ and $\Sigma''$ replaced by $\Sigma_{>F}$.
Now the second and third vertical maps are isomorphisms by the induction hypothesis, which shows that $\V_{[F]} \otimes_\kk\kk' \to (\V \otimes_\kk\kk')_{[F]}$ is an isomorphism.  This completes the proof of (b).
		
		Finally, we prove (c).  If $\V$ is pure, then its restriction homomorphisms are surjections, and (b) implies that the restriction homomorphisms of $\V\otimes_\kk\kk'$ are surjections.  Thus $\V\otimes_\kk\kk'$ is flabby.  Similarly, the corestrictions of $\V$ are split injections by Proposition \ref{prop:(local) freeness}, so (a) implies that $\V\otimes_\kk\kk'$ is coflabby.  Since we are assuming $\kk'$ is a field, local freeness is automatic.  Thus $\V\otimes_\kk\kk'$ is pure.
	\end{proof}
	
	\begin{proposition}\label{prop:base change for IH}
		If $\kk\to \kk'$ is a field extension, or if $\kk$ is a DVR and $\kk' = \olkk$ is its residue field, then
		there is an isomorphism
		\[\E(\cL;\kk') \cong \E(\cL;\kk)\otimes_\kk\kk'.\]
	\end{proposition}
	\begin{proof}
		By Theorem \ref{thm:base change preserves purity},  $\E(\cL;\kk)\otimes_\kk\kk'$ is pure, and it is clearly not annihilated by $y_{\hat{1}}$, so it only remains to show that it is indecomposable.  By Theorem \ref{thm:indecomposables}, it is enough to show that it is taut at all elements $F \ne \hat{1}$.
		
		Consider first the case of a field extension.  Let $\V = \Lambda\E(\cL;\kk)$, and take an element $F\ne \hat{1}$.  Since $\kk$ is a field, $\V$ is taut at $F$ if and only if $\ker\Phi_F \subseteq \im \Psi_F$.  Since tensoring with $\kk'$ is exact, we have
		\[\ker (\Phi_F \otimes_\kk\kk') = (\ker \Phi_F)\otimes_\kk\kk' \subseteq (\im\Psi_F)\otimes_\kk\kk' = \im (\Psi_F \otimes_\kk\kk').\]
		By Theorem \ref{thm:base change preserves purity}, $\Phi_F \otimes_\kk\kk'$ and $\Psi_F \otimes_\kk\kk'$ are the $\Phi$ and $\Psi$ maps for $\V\otimes_\kk\kk'$.  So we have shown that 
		$\V\otimes_\kk\kk'$ is taut at $F$, as desired.
		
		The case $\kk \mapsto \olkk$ is similar.   Since $\V = \Lambda\E(\cL;\kk)$ is taut at $F$, we have $\ker \overline{\Phi_F} \subseteq \im \overline{\Psi_F}$.  But Theorem \ref{thm:base change preserves purity} implies that $\overline{\Phi_F}$ and $\overline{\Psi_F}$ are the $\Phi$ and $\Psi$ maps for $\bar\V = \V\otimes_\kk\olkk$.  Thus we see that $\E(\cL;\kk)\otimes_\kk\kk'$ is taut at $F$ for all $F\ne \hat{1}$, so it is indecomposable. 
	\end{proof}

The field extension case of Proposition \ref{prop:base change for IH} has the following immediate corollary.

\begin{corollary}\label{cor:field extensions}
Let $\kk$ be a field of characteristic $p$.
We have 
\[
{}_pZ_\cL(t) = \sum_{i\geq 0}t^i \dim \E^i(\cL;\kk) \and _pP_\cL(t) \coloneq \sum_{i\geq 0}t^i \dim \uE^i(\cL;\kk). 
\]
\end{corollary}

\excise{
\item[(b)]  We have the coefficientwise inequalities
			\[
			\Poin(\E(\cL;\kk),t) \ge \Poin(\E(\cL,\Q),t)\;\;\text{and}\;\;\Poin(\E(\cL;\kk)_\hz,t) \ge \Poin(\E(\cL,\Q)_\hz,t).
			\]
		\end{enumerate}
		}
		
\begin{proof}[Proof of Theorem \ref{it goes up}]
We use base change by the homomorphisms $\F_p \leftarrow \Z_p \to \Q_p$.
Proposition \ref{prop:base change for IH} tells us that
\[{}_pZ_\cL(t) = \sum_{i\geq 0}t^i \dim_{\F_p} \E^i(\cL;\F_p) = \sum_{i\geq 0}t^i \rk_{\Z_p} \E^i(\cL;\Z_p)
= \sum_{i\geq 0}t^i \dim_{\Q_p} \E^i(\cL;\Z_p) \otimes_{\,\Z_p}\!\Q_p.\]
Since $\E(\cL;\Z_p) \otimes_{\,\Z_p}\!\Q_p$ is pure and is not annihilated by $y_{\hat{1}}$, it has $\E(\cL;\Q_p)$ as a direct summand.  The first inequality follows.  The proof of the inequality on stalks is similar.
		\end{proof}

	\section{Dimension of the space of homomorphisms}\label{sec:hom}
	In this section we assume that $\kk$ is a field.  It is possible to extend these results to the case where $\kk$ is a DVR, but we do not need it.
	
	Let $\Sigma$ be an upper set, and consider objects $\V, \W$ in $\cC(\Sigma;\kk)$.  Let $\hom_{\cC(\Sigma;\kk)}(\V,\W)$ be the {\bf graded hom space} whose $i^{\text{th}}$ graded piece is $\Hom_{\cC(\Sigma;\kk)}(\V,\W[i])$.  The graded hom space between two graded $\kk$-vector spaces is defined in the same way.
	\begin{theorem}\label{thm:Soergel hom formula}
		If $\V$ and $\W$ are both pure,
		then there is a non-canonical isomorphism of graded vector spaces
		\[\hom_{\cC(\Sigma;\kk)}(\V,\W) \cong \bigoplus_{F\in \Sigma} \hom(\V_F,\W_{[F]}).\]
	\end{theorem}
	\begin{proof}
		We proceed by induction on $|\Sigma|$.  If $|\Sigma|=1$, then $\Sigma = \{\ho\}$ and the category $\cC(\Sigma;\kk)$ is equivalent to the category of graded $\kk$-vector spaces,  so the result is trivial.
		
		Suppose that $|\Sigma| > 1$, and let $F$ be a minimal element in $\Sigma$, so that $\Sigma' = \Sigma\setminus\{F\}$ is an upper set.  Then we have an exact sequence of graded vector spaces
		\[
		0  \longrightarrow \hom(\V_F,\W_{[F]}) \longrightarrow \hom_{\cC(\Sigma;\kk)}(\V,\W) \longrightarrow \hom_{\cC(\Sigma';\kk)}(\V|_{\Sigma'},\W|_{\Sigma'}) \longrightarrow 0.
		\]
The second map is given by restriction, and it is surjective by Proposition \ref{prop:extending morphisms}.  The first map is given by sending $\alpha\colon \coker\Psi_{\V,F} \to\ker\Phi_{\W,F}$ to a map that restricts to zero on $\Sigma'$, and whose component at $F$ is the composition
		\[
\V(F)
\twoheadrightarrow
\V_F
\xrightarrow{\ \alpha\ }
\W_{[F]}
\hookrightarrow
\W(F).
\]
		Applying the inductive hypothesis to morphisms $\V|_{\Sigma'} \to \W|_{\Sigma'}$ gives the result.   
	\end{proof}

Proposition \ref{thm:main equivalence}, Lemma \ref{lem:sheaf and module stalks}, and Theorem \ref{thm:Soergel hom formula} together imply the following
corollary.
	
	\begin{corollary}
		If $\M$ and $\M'$ are pure graded $\H(\cL;\kk)$-modules, then there is a non-canonical isomorphism of graded vector spaces
		\[\hom_{\H(\cL;\kk)}(\M,\M')\cong \bigoplus_{F\in \cL}\hom(\M_{F},\M'_{[F]}).\]
	\end{corollary}

Taking both $\M$ and $\M'$ to be the parity module $\E(\cL;\kk)$, we obtain the following result.

	\begin{corollary}
	We have $\End_{\H(\cL;\kk)} \E(\cL;\kk) \cong \kk$ if and only if, for every $F \neq \ho$ and every degree $k$,
	either the stalk $\E(\cL;\kk)_F$ or the costalk $\E(\cL;\kk)_{[F]}$ vanishes in degree $k$.
	\end{corollary}
	
	In \cite[Proposition 6.4]{SHCG}, we showed that the only endomorphisms of $\IH(\cL)  = \E(\cL;\Q)$ are scalar multiplication. 
	The proof relied on \cite[Proposition 6.3]{SHCG}, which states that the stalk
	 $\IH(\cL)_F$ vanishes in degrees greater than or equal to   $(\crk F)/2$ and  the costalk $\IH(\cL)_{[F]}$ vanishes in degrees less than or equal to  $ (\crk F)/2$.  
	Although the argument there was expressed in somewhat different language, it follows the same outline given here.
	
	\begin{example}
	Let $\cL = U_{2,n}$ be the lattice of flats for the uniform matroid of rank $2$ on $n$ elements, and let $\kk$ be a field of characteristic $p$.
	If $n\equiv 1$ modulo $p$, then $\E(\cL;\kk) \cong \CH(\cL;\kk)$ by Theorem~\ref{E=CH}. 
	The degree $1$ parts of the stalk 
	$\CH(\cL;\kk)_{\hat 0}$ and the costalk $\CH(\cL;\kk)_{[\hat 0]}$ are both isomorphic to $\kk$, and we have $\End_{\H(\cL;\Bbbk)} \E(\cL;\Bbbk) \cong \Bbbk[\varepsilon]/\langle \varepsilon^2\rangle$.
	\end{example}
	
\section{Relation with topology}\label{sec:relation with topology}

In this section, we assume that $\cL$ is the lattice of flats of a matroid,
and that this matroid is realizable by a linear subspace $V\subseteq \C^{\cL^1}$,
and we write $Y = Y_V \coloneq \overline{V} \subseteq (\PP^1)^{\cL^1}$ to denote the associated arrangement Schubert variety.
Given $p\in Y$ and $G\in\cL^1$, we will write $p_G\in\PP^1 = \C\cup\{\infty\}$ to denote the corresponding coordinate.
For any $F\in\cL$, let
\[S_F \coloneq \{p\in Y \mid p_G = \infty \iff G \not\leq F\}.\] 
By \cite[Lemma 7.6]{PXY}, we have $S_F\cong\C^{\rk F}$ and the stratification
\[Y = \bigsqcup_{F\in\cL} S_F.\]
The strata $\cS = \{S_F\mid F\in\cL\}$ are the orbits of the action of the additive group $(V,+)$ on $Y$. This implies that they form an algebraic stratification in the sense of Chriss--Ginzburg \cite[Definition 3.2.23]{Chriss-Ginzburg} and a good stratification in the sense of Achar \cite[Definition 2.3.20]{Achar}.  It follows that for any inclusion $f$ of locally closed unions of strata, 
the functors $f_*, f^*, f_!, f^!$ all preserve $\cS$-constructibility.


Let $\kk$ be either a field of arbitrary characteristic or a complete DVR of characteristic 0.
Let $D^b_\cS(Y;\kk)$ denote the derived category of sheaves of $\kk$-modules on $Y$ 
that are constructible with respect to $\cS$.  
For each element $F\in \cL$, let $i_F\colon S_F\to Y$ denote the inclusion.  An object $\cF$ in $D^b_\cS(Y;\kk)$ is called \textbf{$*$-even} (resp.\ \textbf{$!$-even}) if, for every $F\in \cL$, the cohomology sheaves
$\cH^k(i_F^*\cF)$ (resp.\ $\cH^k(i_F^!\cF)$) vanish for all odd $k\in\Z$ and are free $\kk$-local systems for all even $k\in\Z$.  It is called \textbf{even} if it is both $*$-even and $!$-even.  Juteau, Mautner, and Williamson defined a class of objects in $D^b_\cS(Y;\kk)$ called \textbf{parity sheaves}, described by the following theorem \cite{JMW14}.

\begin{theorem}
There exists a collection $\{\cE(F;\kk)\mid F\in\cL\}$ of objects of $D^b_\cS(Y;\kk)$
satisfying the following properties:
	\begin{enumerate}
		\item $\cE(F;\kk)$ is even and indecomposable.
		\item The support of $\cE(F;\kk)$ is equal to $\overline{S_F}$.
		\item The pullback $i_F^*\cE(F;\kk)$ to $S_F$ is a rank one constant $\kk$-local system in degree zero.
	\end{enumerate}
This collection is unique up to isomorphism.  
Every even object in $D^b_\cS(Y;\kk)$ is isomorphic to a direct sum of objects of the form $\cE(F;\kk)[2k]$ for some $k\in \Z$.
\end{theorem}

\begin{proof}
	The uniqueness and support condition are \cite[Theorem 2.12]{JMW14}.  Existence of the complexes $\cE(F;\kk)$ follows from \cite[Corollary 2.35]{JMW14}, since the closures of strata $\overline{S_F}$, which are smaller arrangement Schubert varieties, have $(V,+)$-equivariant resolutions whose fibers have no odd cohomology.
\end{proof}

For any object $\cF$ in $D^b_\cS(Y;\kk)$, the hypercohomology $\HH^\bullet(\cF)$ is naturally a graded module over $\H^\bullet(Y;\kk)$.  
The next result, which is the main result of this section, identifies the hypercohomology of the parity sheaves on $Y$ with the parity modules of $\cL$.  The special case where $F = \ho$ and $\kk$ is a field
gives us Theorem \ref{topological interpretation} from the introduction.

\begin{theorem}\label{thm:parity sheaf to parity module}
	For any $F\in \cL$, there is a degree-doubling isomorphism
	\[\E(\cL^F;\kk)\cong \HH^\bullet(\cE(F;\kk))\]
	of modules over $\H(\cL;\kk)\cong\H^\bullet(Y;\kk)$.
\end{theorem}

The proof of this theorem will proceed in several stages. We begin by associating a sheaf and a cosheaf to an even object in $D^b_\cS(Y;\kk)$.
For an upper set $\Sigma\subseteq \cL$, let 
\[U_\Sigma \coloneq \bigcup_{F\in \Sigma} S_F\]
be the corresponding open union of strata, and let $j_\Sigma\colon U_\Sigma\to Y$ be the inclusion.  Let $\cF \in D^b_\cS(Y;\kk)$ be even.

\begin{proposition}\label{prop:even gives sheaf and cosheaf}
	  The assignment
	\[\Sigma \longmapsto \HH^\bullet(j_\Sigma^*\cF)\]
	is a flabby sheaf on $\cL$, and the assignment
	\[\Sigma \longmapsto \HH_c^\bullet(j_\Sigma^*\cF)\]
	is a coflabby cosheaf on $\cL$.  For any upper set $\Sigma$, both of these graded $\kk$-modules vanish in odd degrees.
\end{proposition}

\begin{proof}
	We first prove the parity vanishing statements, using induction on the cardinality of $\Sigma$.  The statements are trivial when $\Sigma$ is empty. Assume it is nonempty, and choose a minimal element $F \in \Sigma$, so that $\Sigma' \coloneq \Sigma \setminus \{F\}$ is an upper set.  By induction, both $\HH^\bullet(j_{\Sigma'}^*\cF)$ and $\HH_c^\bullet(j_{\Sigma'}^*\cF)$ vanish in odd degrees. 
	Applying hypercohomology to the distinguished triangles of \cite[Theorem 1.2.10]{Achar} gives exact sequences
	\begin{equation}\label{eqn:costalk exact sequence} 
		\cdots \longrightarrow \HH^\bullet(i_F^!\cF) \longrightarrow \HH^\bullet(j_\Sigma^*\cF) \longrightarrow \HH^\bullet(j_{\Sigma'}^*\cF) \stackrel{[1]}{\longrightarrow} \cdots
	\end{equation}
	and
	\begin{equation}\label{eqn:stalk exact sequence} 
		\cdots \longrightarrow \HH^\bullet_c(j_{\Sigma'}^*\cF) \longrightarrow \HH^\bullet_c(j_{\Sigma}^*\cF) \longrightarrow \HH^\bullet_c(i_F^*\cF)  \stackrel{[1]}{\longrightarrow} \cdots.
	\end{equation}
	Since $i^*_F\cF$ and $i^!_F\cF$ are both isomorphic to direct sums of constant free $\kk$-local systems with only even shifts, it follows that  $\HH^\bullet(i_F^!\cF)$ and $\HH^\bullet_c(i_F^*\cF)$ both vanish in odd degrees and are free in even degrees.  Since the outer terms of the exact sequences vanish in odd degrees, the same is true of the middle terms, completing the induction.
	
	This argument also shows that the connecting homomorphisms in these exact sequences vanish, so $\HH^\bullet(j_\Sigma^*\cF) \to \HH^\bullet(j_{\Sigma'}^*\cF)$ is surjective and $\HH^\bullet_c(j_{\Sigma'}^*\cF) \to \HH^\bullet_c(j_{\Sigma}^*\cF)$ is 
	split injective for any inclusion $\Sigma' \subseteq \Sigma$ of upper sets with $|\Sigma'| = |\Sigma| - 1$.  By induction, the same holds for any inclusion of upper sets, which proves flabbiness and coflabbiness.
	
	Finally, the parity vanishing proved above implies that, for any upper sets $\Sigma_1,\Sigma_2\subseteq \cL$, the Mayer--Vietoris sequences
	\[\cdots \longrightarrow \HH^\bullet(j_{\Sigma_1\cup\Sigma_2}^*\cF) \longrightarrow \HH^\bullet(j_{\Sigma_1}^*\cF) \oplus \HH^\bullet(j_{\Sigma_2}^*\cF)  \longrightarrow \HH^\bullet(j_{\Sigma_1\cap\Sigma_2}^*\cF) \stackrel{[1]}{\longrightarrow} \cdots\]
	and
	\[\cdots \longrightarrow \HH^\bullet_c(j_{\Sigma_1\cap\Sigma_2}^*\cF) \longrightarrow  \HH^\bullet_c(j_{\Sigma_1}^*\cF) \oplus \HH^\bullet_c(j_{\Sigma_2}^*\cF)  \longrightarrow \HH^\bullet_c(j_{\Sigma_1\cup\Sigma_2}^*\cF) \stackrel{[1]}{\longrightarrow} \cdots\]
	split into short exact sequences.  The first exact sequence says that $\Sigma \mapsto \HH^\bullet(j_\Sigma^*\cF)$ is a sheaf, and the second says that $\Sigma \mapsto \HH^\bullet_c(j_\Sigma^*\cF)$ is a cosheaf.
\end{proof}

Proposition \ref{prop:even gives sheaf and cosheaf} holds for even objects on any variety stratified by affine spaces; 
for example, it also holds for the Schubert stratification of a flag variety.  
But the exploded modules introduced in Section \ref{sec:exploded modules} have additional structure, 
namely an identification (up to shift) of the sections of the sheaf with the cosections of the cosheaf on principal open sets.
In order to prove Theorem \ref{thm:parity sheaf to parity module}, 
we need this additional structure.  It is easy to see that the sheaf and cosheaf of Proposition \ref{prop:even gives sheaf and cosheaf} agree on the maximal open set $U_\cL = Y$, since $Y$ is compact and therefore \[\HH^\bullet(\cF)=\HH^\bullet_c(\cF)\]
for any $\cF\in D^b_\cS(Y;\kk)$.  But to relate the sheaf and cosheaf beyond this extreme case we need the following proposition, which is particular to the geometry of the arrangement Schubert variety $Y$.  

For all $F\in \cL$, let $U_F = U_{\Sigma_{\geq F}}$ and 
$j_F = j_{\Sigma_{\geq F}}$.  

\begin{proposition}\label{prop:action of yf}
	For any $\cF \in D^b_\cS(Y;\kk)$ and any $F \in \cL$, there is a natural isomorphism
	\[\HH^\bullet(j_F^*\cF) \stackrel{\sim}{\longrightarrow} \HH^\bullet_c(j_F^*\cF)[2\rk F]\]
	of functors from $D^b_\cS(Y;\kk)$ to graded $\kk$-modules
	with the property that the composition 
	\[\HH^\bullet(\cF) \longrightarrow \HH^\bullet(j_F^*\cF) \longrightarrow \HH^\bullet_c(j_F^*\cF)[2 \rk F] \longrightarrow \HH^\bullet_c(\cF)[2 \rk F]=\HH^\bullet(\cF)[2 \rk F]\]
	with the restriction and corestriction maps coincides with the action of $y_F \in \H(\cL;\kk)$.
\end{proposition}

Before proving Proposition \ref{prop:action of yf}, we show how it implies  Theorem \ref{thm:parity sheaf to parity module}.  The proof will be based on the following result, which is an immediate consequence of Propositions \ref{prop:even gives sheaf and cosheaf} and \ref{prop:action of yf}.

\begin{corollary}\label{cor:span and ann}
	For any even object $\cF \in D^b_\cS(Y;\kk)$ and any $F \in \cL$, the kernel of the multiplication by $y_F$ on $\HH^\bullet(\cF)$ is equal to the kernel of restriction $\HH^\bullet(\cF) \to  \HH^\bullet(j_F^*\cF)$ and the image is equal to the image of the corestriction $\HH^\bullet_c(j_F^*\cF) \to \HH^\bullet_c(\cF)$.
\end{corollary}

\begin{remark}
This corollary can be interpreted as saying that the exploded module 
$\cV \coloneq \Lambda\HH^\bullet(\cF)$ is built from the spaces $\cV(F) = \HH^\bullet(j^*_F\cF) \cong \HH^\bullet_c(j^*_F\cF)[2 \rk F]$, with the maps $\varphi^G_F$ given by restriction and the maps $\psi^G_F$ given by corestriction.	
\end{remark}

\begin{proof}[Proof of Theorem \ref{thm:parity sheaf to parity module}]
	We first show that, for any even object $\cF$ in $D^b_\cS(Y;\kk)$, the $\H(\cL;\kk)$-module $\M \coloneq \HH^\bullet(\cF)$ (with degrees halved) 
	is pure.  More precisely, following Definition \ref{new pure}, we show that the exploded module $\cV \coloneq \Lambda\M$ is flabby, coflabby, and locally free.	
	
	Let $\Sigma\subseteq \cL$ be an upper set.  We have isomorphisms
	\[\Gamma_!(\Sigma,\cV) \cong \varinjlim_{F\in \Sigma} \HH^\bullet_c(j^*_F\cF) \cong \HH^\bullet_c(j^*_\Sigma\cF), \]
	where the first isomorphism follows from Corollary \ref{cor:span and ann} and the second follows from the cosheaf property of Proposition \ref{prop:even gives sheaf and cosheaf}.  Proposition  \ref{prop:even gives sheaf and cosheaf} then shows that $\cV$ is coflabby.
	
	In the same way, we have isomorphisms 
	\[\Gamma_*(\Sigma,\cV) \cong \varprojlim_{F\in \Sigma} \HH^\bullet(j^*_F\cF) \cong \HH^\bullet(j^*_\Sigma\cF), \] 
	and we conclude by Proposition  \ref{prop:even gives sheaf and cosheaf} that $\cV$ is flabby.
	
	Next, we show that the stalk and costalk of $\cV$ at any $F\in \cL$ are identified with their topological counterparts.  For the stalk, we have
	\begin{align*}
		\cV_F & = \coker \Psi_{F,\cV}\\
		& \cong \coker\left(\HH^\bullet_c(j^*_{\Sigma_{>F}}\cF)[2\rk F] \to \HH^\bullet_c(j^*_F\cF)[2\rk F] \right)\\
		& \cong \HH^\bullet_c(i^*_F\cF)[2\rk F]\\
		&\cong \HH^\bullet(i^*_F\cF),
	\end{align*}
	where the second isomorphism follows from the exact sequence \eqref{eqn:stalk exact sequence} (recall that the connecting homomorphisms of this sequence vanish).  Since the cohomology sheaves of $i^*_F\cF$ are constant free $\kk$-local systems, $\HH^\bullet(i^*_F\cF)$ is free, and so $\cV$ is locally free.  We have completed the proof that $\cV$ is pure.
	
	We have a dual isomorphism for the costalk:
	\begin{align*}
		\cV_{[F]} & = \ker \Phi_{F,\cV} \\
		& \cong \ker\left(\HH^\bullet(j^*_F\cF) \to \HH^\bullet(j^*_{\Sigma_{>F}}\cF)\right)\\
		& \cong \HH^\bullet(i_F^!\cF),
	\end{align*} 
	where the second isomorphism follows from the exact sequence \eqref{eqn:costalk exact sequence}.  Combining these isomorphisms, we see that the 
	costalk-stalk homomorphism $\beta_F(\M)$, which is the composition
	\[\cV_{[F]} \longrightarrow \cV(F) \longrightarrow \cV_F,\]
	is identified with the natural map 
	\[\HH^\bullet(i_F^!\cF) \longrightarrow \HH^\bullet(i_F^*\cF).\]
	
	Now suppose that $\cF$ is the parity sheaf $\cE(F;\kk)$.  We wish to show that the exploded module $\cV = \Lambda\HH^\bullet(\cF)$ is isomorphic to $\eE^F_{\cL,\kk}$ as defined in Theorem \ref{thm:indecomposables}, since  Lemma \ref{unexploding indecomposables} then yields $\M \cong \Theta\cV \cong \E(\cL^F;\kk)$.  We know that $\cV$ is pure, so we must verify the remaining conditions of Theorem \ref{thm:indecomposables}.
	
	For any $G \not\le F$ we have $\cV(G) \cong \HH^\bullet(j_G^*\cF) = 0$, since $\cE(F;\kk)$ is supported on $\overline{S_F}$, and we also have $\cV(F) \cong \HH^\bullet(j_F^*\cF) \cong\kk$, since $\cE(F;\kk)$ restricts to a constant rank one local system on $S_F$.
	
	It remains to show that $\cV$ is taut at all $G\ne F$.   This is trivial if $G \not\le F$ since $\cV(G) = 0$, so suppose that $G < F$.  The restriction $j^*_G\cF$ is indecomposable by \cite[Proposition 2.11]{JMW14}, and $S_G$ is closed in $U_G$, so we can apply \cite[Lemma 2.21]{JMW14} together with the fact that $S_G$ is contractible to deduce that the 
	image of the natural homomorphism
	\[\HH^\bullet(i^!_G\cF) \longrightarrow \HH^\bullet(i^*_G\cF)\]
	does not contain any non-trivial direct summands as $\kk$-modules.  
	Equivalently, the induced homomorphism of $\olkk$-modules 
	\[\overline{\HH^\bullet(i^!_G\cF)} \longrightarrow \overline{\HH^\bullet(i^*_G\cF)}\]
	is zero.\footnote{Note that, if $\kk$ is a field, then $\kk=\olkk$ and the two homomorphisms are the same.}  We have seen that this map is identified with $\overline{\beta_G(\M)}$, so Lemma \ref{lem:tautness and (co)stalks} shows that $\cV$ is taut at $G$, as desired.  
\end{proof}

\begin{proof}[Proof of Proposition \ref{prop:action of yf}]
	Let $B \subseteq \cL^1$ be a basis for $F$.  In other words, $B$ is a minimal set of atoms whose join is $F$.  We have $|B| = \rk F$.
	
	Let $X = (\PP^1)^B$, and let $p\colon Y \to X$ denote the inclusion of $Y$ into $(\PP^1)^{\cL^1}$ followed by projection onto the factors indexed by $B$.  This map is equivariant for the action of $(V,+)$, where the action of $v\in V$ on $X$ forgets the coordinates not in $B$.  If $W = \C^B \subseteq X$ is the locus of points where no coordinate is $\infty$, then we have $p^{-1}(W) = U_F$, the union of all $S_G$ with $G \ge F$.

	The class $y_F\in \H^{\rk F}(\cL;\kk) \cong \H^{2\rk F}(Y;\kk)$ is the product of $y_b$ for $b\in B$.  This implies that $y_F = p^*\alpha$, where 
	$\alpha\in \H^{2\rk F}(X;\kk)$ is the standard generator.  We can describe the class $\alpha$ as a map of sheaves in the following way.   Let $i$ be the inclusion of $\{0\} \subseteq W$ into $X$.	
	Under the natural isomorphism 
	\[\H^{2\rk F}(X;\kk) \cong \Hom_{D^b(X;\kk)}(\ukk_X,\ukk_X[2\rk F]),\]
	the class $\alpha$ is given by the composition
	\begin{equation}\label{eqn:alpha as a sheaf map}
		\ukk_X \longrightarrow i_*i^*\ukk_X \stackrel\cong\longrightarrow i_!i^!\ukk_X[2\rk F] \longrightarrow \ukk_X[2\rk F],	
	\end{equation}
	where the first and third maps are the adjunctions and the middle map is obtained using the standard orientation of $X$ as a smooth complex variety.  More precisely, if $a \colon X \to \{0\}$ is the projection, the orientation induces an isomorphism $a^!\ukk_{\{0\}} \cong \ukk_X[2\rk F]$, and applying $i^!$ gives an isomorphism
	\[\ukk_{\{0\}} \cong i^!a^!\ukk_{\{0\}} \cong i^!\ukk_X[2\rk F]\]
	which, together with the fact that $i_* = i_!$, gives the required map.
	
	We can obtain a map of sheaves representing the class $y_F = p^*\alpha$ by pulling back the map \eqref{eqn:alpha as a sheaf map} by $p$.  We have $p^*\ukk_X = \ukk_Y$, so this will give a map $\ukk_Y \to \ukk_Y[2 \rk F]$.  To see what the pullback of the middle sheaf $i_*i^*\ukk_X$ is, consider $Y_F \coloneq p^{-1}(0)$. It is the arrangement Schubert variety associated with the vector subspace $V_F \coloneq V \cap p^{-1}(0)$, considered as a subspace of $\C^{\cL^1\setminus F}$.  Let $f\colon Y_F\to Y$ denote the inclusion.	
	Since $i_* = i_!$, proper base change implies that $$p^*i_*i^*\ukk_X\cong f_*f^*p^*\ukk_X \cong f_*f^*\ukk_Y.$$  
	Thus pulling back \eqref{eqn:alpha as a sheaf map} by $p$ expresses the cohomology class $y_F$ as a composition
	\[\ukk_Y \longrightarrow f_*f^*\ukk_Y \longrightarrow \ukk_Y[2\rk F].\]
	The action of $y_F$ on the hypercohomology of $\cF\in D^b_\cS(Y;\kk)$ 
	is then given by tensoring this map with $\cF$, using the natural isomorphism $\ukk_Y\otimes \cF \simeq \cF$, and taking hypercohomology.  The projection formula and compatibility of tensor with pullbacks give
	\begin{align*}
		f_*f^*\ukk_Y \otimes \cF & \cong f_!f^*\ukk_Y \otimes \cF \\
		& \cong f_!(f^*\ukk_Y \otimes f^*\cF) \\
		& \cong f_*f^*(\ukk_Y \otimes \cF)\\
		& \cong f_*f^*\cF.
	\end{align*}
	This gives the bottom row in the following diagram:
	\[\begin{tikzcd}
		& (j_F)_*j_F^*\cF \ar[rr] \ar[dr, "\zeta"] & & (j_F)_!j_F^*\cF[2\rk F] \ar[dr]& \\
		\cF \ar[ur] \ar[rr] & & f_*f^*\cF \ar[ur, "\eta"]\ar[rr] & & \cF[2\rk F]
	\end{tikzcd}\]
	The outer diagonal maps come from the adjunctions, together with the fact that $j_F^* = j_F^!$.  To construct the inner diagonal maps $\zeta$ and $\eta$, apply $(j_F)_*j_F^*$ to the bottom left map, apply $(j_F)_!j_F^*$ to the bottom right map, and use the fact that 
	the adjunction maps 
	\[(j_F)_!j_F^*f_*f^*\cF \longrightarrow f_*f^*\cF \longrightarrow (j_F)_*j_F^*f_*f^*\cF\] 
	are isomorphisms, because the support of $f_*f^*\cF$ is contained in $Y_F$, which is a closed subset of $U_F$.
	
	Taking hypercohomology of the upper path in the diagram gives homomorphisms
	\[\HH^\bullet(\cF) \longrightarrow \HH^\bullet(j_F^*\cF) \stackrel{\HH^\bullet(\eta\zeta)}{\longrightarrow} \HH^\bullet_c(j_F^*\cF)[2\rk F] \longrightarrow \HH^\bullet_c(\cF)[2\rk F] = \HH^\bullet(\cF)[2\rk F]\]
	factoring the action of $y_F$.  What remains is to show that $\HH^\bullet(\eta\zeta)$ is an isomorphism.  In fact, we will show that $\HH^\bullet(\eta)$ and $\HH^\bullet(\zeta)$ are both isomorphisms.  This will follow from understanding the local geometry of $Y$ near $Y_F$.  We have an isomorphism $(p,q)\colon U_F \stackrel{\sim}{\longrightarrow} W \times Y_F$, where $q\colon U_F\to Y_F$ is defined as follows.  Choose a complement  $\widetilde{W}$ to $V_F$ in $V$.  The restriction $\tilde{p}=p|_{\widetilde{W}}\colon \widetilde{W} \to W$ is an isomorphism, and we let $q(x) = x - \tilde{p}^{-1}(p(x))$ for $x\in U_F$.   The inverse map is given by $(w,x)\mapsto x + \tilde{p}^{-1}(w)$ for $w \in W$, $x\in Y_F$. 
	
	Using this, one can construct a stratification-preserving action of $\C^\times$ on $U_F$ with fixed point set $Y_F$, and which contracts $U_F$ onto $Y_F$, given by the formula\footnote{ Beware that this action does \emph{not} extend to an action on $Y$.}
	\[(p(t\cdot x),q(t \cdot x)) = (tp(x), q(x)),\;\; t\in \C^\times, x \in U_F.\]
	
	If we let $g\colon Y_F \to U_F$ denote the inclusion, so $f = j_F \circ g$, then the contracting lemma \cite[Theorem 2.10.3]{Achar} gives isomorphisms $q_* \simeq g^*$ and $q_! \simeq g^!$ of functors $D^b_\cS(U_F;\kk) \to D^b(Y_F;\kk)$.  In particular, for an object $\cG$ of $D^b_\cS(U_F;\kk)$, we have natural isomorphisms 
	\[\HH^\bullet(\cG) \cong \HH^\bullet(q_*\cG) \cong \HH^\bullet(g^*\cG),\;\;
	\HH^\bullet_c(\cG)  \cong \HH^\bullet(q_!\cG) \cong \HH^\bullet(g^!\cG).\]
	Since both the source and target of $\zeta$ are pushforwards of objects on $U_F$, to show that $\HH^\bullet(\zeta)$ is an isomorphism it is enough to note that all of the maps in the left triangle become isomorphisms after applying $f^* = g^*j_F^*$.  Similarly, the fact that $\HH^\bullet(\eta)$ is an isomorphism follows by noting that the maps in the right triangle become isomorphisms after applying $f^! = g^!j_F^!$.
\end{proof}

\bibliography{refs}
\bibliographystyle{amsalpha}

\end{document}